\documentclass[11pt,letterpaper]{article}
\usepackage{amssymb}
\usepackage{amsrefs}
\usepackage{amsthm,amsmath}
\usepackage{enumitem}
\usepackage{verbatim}

\usepackage{tikz-cd} 

\DeclareFontFamily{U}{mathx}{\hyphenchar\font45}
\DeclareFontShape{U}{mathx}{m}{n}{
      <5> <6> <7> <8> <9> <10>
      <10.95> <12> <14.4> <17.28> <20.74> <24.88>
      mathx10
      }{}
\DeclareSymbolFont{mathx}{U}{mathx}{m}{n}
\DeclareFontSubstitution{U}{mathx}{m}{n}
\DeclareMathAccent{\widecheck}{0}{mathx}{"71}
\DeclareMathAccent{\wideparen}{0}{mathx}{"75}

\theoremstyle{plain}
\newtheorem{theorem}[equation]{Theorem}
\newtheorem{corollary}[equation]{Corollary}
\newtheorem{lemma}[equation]{Lemma}
\newtheorem{conjecture}[equation]{Conjecture}
\newtheorem{proposition}[equation]{Proposition}
\newtheorem{quesprop}[equation]{\,``Proposition''\,}

\newtheoremstyle{indenteddefinition}{\topsep}{\topsep}{\addtolength{\leftskip}{2.0em}}{-0em}{\bfseries}{.}{
}{} 
\theoremstyle{indenteddefinition}
\newtheorem{definition}[equation]{Definition}
\newtheorem{example}[equation]{Example}
\newtheorem{algorithm}[equation]{Algorithm}

\newcommand{\eqdef}{=_{\text{\textnormal{def}}}}
\DeclareMathOperator\AV{AV}
\DeclareMathOperator\Ann{Ann}
\DeclareMathOperator\Lie{Lie}
\DeclareMathOperator\LL{L}
\DeclareMathOperator\dashLL{\!-L}
\DeclareMathOperator\point{point}
\DeclareMathOperator\unip{unip}
\DeclareMathOperator\red{red}
\DeclareMathOperator\rank{rank}
\DeclareMathOperator\alg{alg}
\DeclareMathOperator\orbalg{orbalg}
\DeclareMathOperator\abs{abs}
\DeclareMathOperator\Ad{Ad}
\DeclareMathOperator\ad{ad}
\DeclareMathOperator\Aut{Aut}

\DeclareMathOperator\Ind{Ind}
\DeclareMathOperator\Span{span}

\DeclareMathOperator\Cone{Cone}

\DeclareMathOperator\tr{tr}
\DeclareMathOperator\triv{triv}
\DeclareMathOperator\sgn{sgn}
\DeclareMathOperator\gr{gr}
\DeclareMathOperator\RE{Re}
\DeclareMathOperator\IM{Im}

\DeclareMathOperator\RRE{RE}

\DeclareMathOperator\opp{opp}

\DeclareMathOperator\Poly{Poly}

\DeclareMathOperator\Rep{Rep}
\DeclareMathOperator\supp{supp}
\DeclareMathOperator\Coh{Coh}
\DeclareMathOperator\QCoh{QCoh}
\DeclareMathOperator\Mod{\!--Mod}
\DeclareMathOperator\Modfg{\!--Mod^{fg}}

\DeclareMathOperator\res{res}
\DeclareMathOperator\weak{weak}
\DeclareMathOperator\WF{WF}

\newcommand{\xdownarrow}[1]{%
  {\left\downarrow\vbox to #1{}\right.\kern-\nulldelimiterspace}
} 
\newcommand{\xuparrow}[1]{%
  {\left\uparrow\vbox to #1{}\right.\kern-\nulldelimiterspace}
} 
\newcommand{\K}{\mathbf K}
\newcommand{\R}{\mathbf R}

\begin{document}


 \title{\vskip -2.4cm Associated varieties for real reductive groups} 


\author{Jeffrey Adams \\Department of
  Mathematics \\ University of Maryland
\and David A. Vogan, Jr.\\Department of Mathematics\\ MIT,
  Cambridge, MA 02139}
\date{\today}
\maketitle
\begin{center}
  \large \emph{In fond memory of our teacher and friend, Bert
  Kostant.}
  \vskip .9cm 
\end{center}

\begin{abstract}
We give an algorithm to compute the associated variety of a
Harish-Chandra module for a real reductive group $G({\mathbb R})$. The
algorithm is implemented in the {\tt atlas} software package.
\end{abstract}


\tableofcontents


\section{Introduction} \label{sec:intro}
\setcounter{equation}{0}
\setcounter{equation}{0}

A great guiding principle of infinite-dimensional
representation theory is the {\em method of coadjoint orbits}
\begin{subequations}\label{se:orbmethod}
of Alexandre Kirillov and Bertram Kostant. It says that there should be a
close relationship
\begin{equation}\label{eq:orbcorr}\begin{aligned}
 i{\mathfrak g}_{\mathbb R}^*/G_{\mathbb R}\quad &=\ 
 \text{\parbox{.42\textwidth}{orbits of a real
   Lie group on the imaginary dual of its Lie algebra}}\\
      \xdownarrow{.4cm} \Pi\quad&\\
\left(\widehat{G_{\mathbb
    R}}\right)_{\text{\textnormal{unitary}}}
&=\  \text{irreducible unitary representations}.
  \end{aligned}
\end{equation}
The phrase ``coadjoint orbit'' means an orbit of a Lie group on
the vector space dual of its Lie algebra. Here
$${\mathcal O}_{i{\mathbb R}} \mapsto \Pi({\mathcal O}_{i{\mathbb R}})$$
is informal notation for
the desired construction attaching a unitary representation to a
coadjoint orbit. This map $\Pi$ is not intended to be precisely defined or even
definable: in the cases where such a correspondence is known, the
domain of $\Pi$ consists of just {\em certain} coadjoint orbits
(satisfying integrality requirements), and endowed with some
additional structure (something like local systems). We introduce the
name $\Pi$ just to talk about the problem.

We will be concerned here with the case of real reductive groups.
For the remainder of this introduction, we therefore assume
\begin{equation}\label{eq:realred}\begin{aligned}
    G &=\  \text{\parbox{.52\textwidth}{complex connected reductive
      algebraic group defined over ${\mathbb R}$,}}\\[.4ex]
    G({\mathbb R}) &= \ \text{group of real points of $G$.}
  \end{aligned}
\end{equation}

The status of the orbit method for real reductive groups is discussed in
some detail for example in
\cite{Vorb}. There it is explained that
\begin{equation}\label{eq:pireduction}
  \text{\parbox{.85\textwidth}{the construction of a
      map $\Pi$ (from orbits to representations) reduces
      to the case of nilpotent coadjoint orbits.}}
\end{equation}
(The phrase {\em nilpotent coadjoint orbit} is defined in
\eqref{se:nilcone} below.)

This nilpotent case remains open in general. We write
\begin{equation}\label{eq:nilcone}
  {\mathcal N}^*_{i{\mathbb R}} = \text{nilpotent elements in $i
    {\mathfrak g}({\mathbb R})^*$.}
\end{equation}
(A precise definition appears in Section \ref{sec:nilp}.) We write
informally
\begin{equation}\label{eq:unipreps}
  \widehat{G({\mathbb R})}_{\text{\textnormal{unip}}} =\ 
  \text{\parbox{.45\textwidth}{representations
      corresponding to nilpotent coadjoint orbits,}}
\end{equation}
the {\em unipotent representations}; this is not a definition, because
the Kirillov-Kostant orbit correspondence $\Pi$ has not been defined.

Harish-Chandra found that the study of irreducible unitary
representations could proceed more smoothly inside the larger set
\begin{equation}\label{eq:admrep}
  \widehat{G({\mathbb R})} \supset \widehat{G({\mathbb
    R})}_{\text{\textnormal{unitary}}}
\end{equation}
of irreducible {\em quasisimple} representations. These are the
irreducible objects of the category introduced in \eqref{eq:GRrep}
below. (These are irreducible topological representations on nice
topological vector spaces. ``Quasisimple'' means that the center of
the enveloping algebra is required to act by scalars, as Schur's lemma
(not available in this topological setting) suggests that it
should.)

The present paper is concerned with how to tell whether a proposed
map $\Pi$ is reasonable. The idea comes from \cite{HoweWF}, \cite{BV}, and
\cite{Vunip}. To each coadjoint orbit we can attach an
{\em asymptotic cone}, a closed $G_{\mathbb R}$-invariant cone
(Definition \ref{def:asympcone})
\begin{equation}\label{eq:asympcone}
{\mathcal O}_{i{\mathbb R}} \in  i{\mathfrak g}({\mathbb R})^*/G({\mathbb
  R}) \longrightarrow \Cone_{\mathbb R}({\mathcal O}_{i{\mathbb
  R}}) \subset {\mathcal N}^*_{i{\mathbb R}}.
\end{equation}
An easy but important property is that the asymptotic cone of a nilpotent
orbit is just its closure:
\begin{equation}\label{eq:asympnilp}
\Cone_{\mathbb R}({\mathcal O}_{i{\mathbb
  R}}) = \overline{{\mathcal O}_{i{\mathbb R}}}, \qquad {\mathcal O}_{i{\mathbb R}} \in
{\mathcal N}^*_{i{\mathbb R}}/G({\mathbb R}).
\end{equation}

In a parallel way, to each irreducible quasisimple representation,
Howe in \cite{HoweWF} (see also \cite{BV}) attached a {\em wavefront set}, a
closed cone
\begin{equation}\label{eq:WFpiA}
\pi \in \widehat{G({\mathbb R})} \longrightarrow \WF_{\mathbb R}(\pi)
\subset {\mathcal N}^*_{i{\mathbb R}}/G({\mathbb R}).
\end{equation}
Here is an outline of Howe's definition. If $D$ is a generalized function on a
manifold $M$, and $m$ is a point of $M$, then the {\em wavefront set
  of $D$ at $M$} is
\begin{equation}\label{eq:wavefrontset}
  0 \in \WF_m(D) \subset iT^*_m M,
\end{equation}
a nonzero closed cone in the cotangent space at $m$. The size of
$\WF_m(D)$ measures the singularity of $D$ near $m$: the wavefront set of
a smooth function is just the point zero, and the wavefront set of the
Dirac delta function is the full cotangent space at $m$. The factor of
$i$ is helpful because the definition of $\WF_m$ involves the Fourier
transform of $D$ ``near $m$;'' and this Fourier transform is most
naturally a function on $iT^*_m M$.

Harish-Chandra attached to the irreducible quasisimple representation
$\pi$ a {\em distribution character} $\Theta_\pi$, which is a
generalized function on $G({\mathbb R})$. If $\pi$ is
finite-dimensional, then $\Theta_\pi$ is a smooth function (whose
value at $g \in G({\mathbb R})$ is $\tr\pi(g)$), so $\WF_e(\Theta_\pi) =
\{0\}$. If $\pi$ is infinite-dimensional, then $\Theta_\pi$ must be
singular at the identity, since its ``value'' would be the dimension
of $\pi$; so in this case $\WF_e(\Theta_\pi)$ is a nonzero cone.

In general Howe's definition amounts to
\begin{equation}\label{eq:WFpiB}
  \WF_{\mathbb R}(\pi) \eqdef \WF_e(\Theta_\pi) \subset i{\mathfrak
    g}_{\mathbb R}^*
\end{equation}
Howe proves (\cite{HoweWF}*{Proposition 2.4}) that
\begin{equation}\label{eq:WFpiC}
  \WF_{\mathbb R}(\pi) \subset {\mathcal N}^*_{i\mathbb R},
\end{equation}
a closed finite union of nilpotent coadjoint orbits for $G({\mathbb R})$.

One of the desiderata of
the orbit method is that the asymptotic cone and wavefront set
constructions should be compatible with the proposed map $\Pi$ of
\eqref{eq:orbcorr}: if ${\mathcal O}_{i{\mathbb R}}$ is a
coadjoint orbit, then
\begin{equation}\label{eq:orbrep}
  \WF_{\mathbb R}(\Pi({\mathcal O}_{i{\mathbb R}})) \buildrel?\over=
  \Cone_{\mathbb R}({\mathcal O}_{i{\mathbb R}}).
\end{equation}
When ${\mathcal O}_{i{\mathbb R}}$ is {\em nilpotent},
\eqref{eq:asympnilp} shows that this desideratum simplifies to
\begin{equation}\label{eq:nilporbrep}
  \WF_{\mathbb R}(\Pi({\mathcal O}_{i{\mathbb R}})) \buildrel?\over=
  \overline{{\mathcal O}_{i{\mathbb R}}} \qquad \text{(${\mathcal
      O}_{i{\mathbb R}}$ nilpotent).}
\end{equation}

Our motivation (not achieved) is the construction of a
Kirillov-Kostant orbit-to-representation
correspondence $\Pi$ as in \eqref{eq:orbcorr}. According to
\eqref{eq:pireduction}, it is enough to construct $\Pi({\mathcal
  O}_{i{\mathbb R}})$
for each {\em nilpotent} orbit ${\mathcal O}_{i{\mathbb R}}$. A common
method to do this has been to construct a candidate representation
$\pi$, and then to test whether the requirement \eqref{eq:nilporbrep} is
satisfied. That is,
\begin{equation}\label{eq:unipcand}
    \text{\parbox{.48\textwidth}{if ${\mathcal O}_{i{\mathbb R}}$ is
        nilpotent, candidates for $\Pi({\mathcal O}_{i{\mathbb
          R}})$ must satisfy $\WF_{\mathbb R}(\pi) = \overline{{\mathcal
          O}_{i{\mathbb R}}}$.}}
\end{equation}
In order to use this idea to guide the construction of $\Pi$, we
therefore need to know how to
\begin{equation}\label{eq:WFprob}
  \text{\parbox{.6\textwidth}{compute the wavefront set of
    any quasisimple irreducible representation.}}
\end{equation}
That is the problem solved in this paper.
\end{subequations} 

Everything so far has been phrased in terms of {\em real} nilpotent
coadjoint orbits, but all of the ideas that we will use come from
{\em complex} algebraic geometry. In Sections \ref{sec:nilp} and
\ref{sec:assvar} we will recall results of Kostant-Sekiguchi
\cite{Sek} and Schmid-Vilonen \cite{SV} allowing a reformulation of
\eqref{eq:WFprob} in complex-algebraic terms.

\begin{subequations}\label{se:introR}
We do not know even how properly to formulate our main results except
in this complex-algebraic language, so a proper summary of them will
appear only in Section \ref{sec:cplxalg}. For the moment we will
continue as if it were possible to make a
real-groups formulation of the solution to \eqref{eq:WFprob}. The
reader can take this as an outline of an interesting problem: to make
precise sense of the statements in the rest of the introduction.

We continue with the assumption \eqref{eq:realred}
that $G({\mathbb R})$ is a real reductive algebraic group. Write
\begin{equation}\label{eq:GRrep}
  {\mathcal F}_{\text{\textnormal{mod}}}(G({\mathbb R})) =
  \text{\parbox{.56\textwidth}{${\mathfrak Z}({\mathfrak g})$-finite
      finite length smooth Fr\'echet representations of moderate growth}}
\end{equation}
(see \cite{WallachII}*{Chapter 11.6}). Casselman and Wallach proved
that this is a nice category; the irreducible objects are precisely
the irreducible quasisimple representations $\widehat{G({\mathbb R})}$,
so the Grothendieck group of the category is
\begin{equation}
\K_0({\mathcal F}_{\text{\textnormal{mod}}}(G({\mathbb R}))) = {\mathbb
  Z}\cdot \widehat{G({\mathbb R})},
\end{equation}
a free abelian group with basis the irreducible quasisimple
representations.

About notation: the maximal compact subgroup of a reductive group is
more or less universally denoted $K$, and we are unwilling to change
that notation. This paper makes extensive use of $\K$-theory, beginning
with the Grothendieck group $\K_0$. To try to reduce the confusion with
the compact group $K$, we will write $\K$ to refer to $\K$-theory.

We do not know a good notion of equivariant $\K$-theory for
real algebraic groups. But such a notion ought to exist; and there
ought to be an ``associated graded'' map
\begin{equation}\label{eq:grR}
  \gr_{\mathbb R}\colon \K_0({\mathcal F}_{\text{\textnormal{mod}}}(G({\mathbb
    R}))) \buildrel? \over\rightarrow \K^{G({\mathbb R})}({\mathcal
    N}^*_{i{\mathbb R}}).
\end{equation}
Each element of $\K^{G({\mathbb R})}({\mathcal N}^*_{i{\mathbb R}})$
should have a well-defined ``support,'' which should be a closed
$G({\mathbb R})$-invariant subset of ${\mathcal N}^*_{i{\mathbb
    R}}$. In the case of a quasisimple representation $\pi$ of finite
length, this support should be the wavefront set of
\eqref{eq:WFpiA}:
\begin{equation}\label{eq:suppgrR}
\supp_{\mathbb R}\gr_{\mathbb R}([\pi]) \buildrel ?\over = \WF_{\mathbb R}(\pi).
\end{equation}
The question mark is included because the left side is for the moment
undefined; in the algebraic geometry translation of Definition
\ref{def:assvar}, this equality will become meaningful and true. The
problem \eqref{eq:WFprob} becomes
\begin{equation}\label{eq:WFprobR}
  \text{compute explicitly the map $\supp_{\mathbb R}\circ\gr_{\mathbb R}$.}
\end{equation}
Evidently this can be done in two stages: to compute explicitly the
map $\gr_{\mathbb R}$, and then to compute explicitly the map
$\supp_{\mathbb R}$.

Here is the first step. Just as in the case of highest weight
representations, each irreducible quasisimple representation $\pi$ is
described by the Langlands classification as the unique irreducible
quotient of a ``standard representation.''
Standard representations have very concrete parameters
\begin{equation}\label{eq:param}
  \Gamma = (\Lambda,\nu), \qquad \Gamma \in {\mathcal P}_{\LL}(G({\mathbb R}))
\end{equation}
which we will explain in Section \ref{sec:realalg} (see in particular
\eqref{e:realparams}). For the moment,
the main points are that
\begin{equation}\label{eq:discparam}
  \Lambda \in {\mathcal P}_{\text{\textnormal{disc}}}(G({\mathbb R}))
\end{equation}
runs over a countable {\em discrete} set, and
\begin{equation}\label{eq:contparam}
  \nu \in {\mathfrak a}^*(\Lambda)
  \end{equation}
runs over a complex vector space associated to the discrete parameter
$\Lambda$. Attached to each parameter $\Gamma$ we have
\begin{equation}\label{eq:stdirr}
  I(\Gamma) \twoheadrightarrow J(\Gamma),
\end{equation}
a standard representation and its unique irreducible quotient.
\end{subequations} 

\begin{quesprop}\label{prop:grR} Suppose we are in the setting of
  \eqref{se:introR}.
  \begin{enumerate}
  \item The irreducible modules
    $$\{J(\Gamma)\mid \Gamma \in {\mathcal P}_{\LL}(G({\mathbb R}))\}$$
    are a ${\mathbb Z}$ basis of the
      Grothendieck group $\K_0({\mathcal F}_{\text{\textnormal{mod}}}
      (G({\mathbb R})))$.
    \item The standard modules
      $$\{I(\Gamma)\mid \Gamma \in {\mathcal
      P}_{\LL}(G({\mathbb R}))\}$$
      are a ${\mathbb Z}$ basis of $\K_0({\mathcal
        F}_{\text{\textnormal{mod}}} (G({\mathbb R})))$.
      \item The change of basis matrix
        $$J(\Gamma) = \sum_\Xi M(\Xi,\Gamma)I(\Xi)$$
        is computed by Kazhdan-Lusztig theory (\cite{LV}).
      \item The image
        $$\gr_{\mathbb R}(I(\Lambda,\nu)) \in \K^{G({\mathbb R})}({\mathcal
        N}^*_{i{\mathbb R}})$$
        is independent of the continuous parameter $\nu\in {\mathfrak
          a}(\Lambda)^*$.
      \item The classes
        $$\{\gr_{\mathbb R}(I(\Lambda,0)) \mid \Lambda \in {\mathcal
        P}_{\text{\textnormal{disc}}}(G({\mathbb R}))\}$$
        are a ${\mathbb Z}$-basis of the equivariant $\K$-theory
        $\K^{G({\mathbb R})}({\mathcal N}^*_{i{\mathbb R}})$.
\end{enumerate}\end{quesprop}

The quotation marks are around the proposition for two reasons. First,
we do not have a definition of $G({\mathbb R})$-equivariant
$\K$-theory; we will actually prove algebraic geometry analogues of
(4) and (5) (Corollary \ref{cor:KKrealbasis} and Theorem
\ref{thm:normalize}). Second, the description of the Langlands
classification above is slightly imprecise; the corrected statement is
just as concrete and precise, but slightly more complicated.

One way to think about (4) is that $\K$-theory is a topological notion,
which ought to be invariant under homotopy. Varying the continuous
parameter in a standard representation is a continuous deformation of
the representations, and so does not change the class in $\K$-theory.

This proposition is a complete computation of
$\gr_{\mathbb R}$: it provides ${\mathbb Z}$ bases for the range and
domain, and says that the map is given by identifying certain
continuous families of basis vectors. Furthermore it explains how to
write each irreducible module in the specified basis.

We turn next to the explicit computation of $\supp_{\mathbb R}$. Again
the key point is a change of basis: this time from the
representation-theoretic basis of equivariant $K$-theory given by
Proposition \ref{prop:grR}(5) (or rather Corollary
\ref{cor:KKrealbasis}) to one related to the geometry of ${\mathcal
  N}^*_{i{\mathbb R}}$.

Suppose that $H({\mathbb R})$ is any real algebraic subgroup of
$G({\mathbb R})$. Assuming that there is a reasonable notion of
equivariant $K$-theory for real algebraic groups, it ought to be true
that
\begin{equation}
  \K^{G({\mathbb R})}(G({\mathbb R})/H({\mathbb R})) \buildrel ? \over
  \simeq \K^{H({\mathbb R})}(\point),
\end{equation}
(As usual the question mark is a reminder that we do not know how to
define this $\K$-theory.) The right side in turn should be a free
abelian group with natural 
basis indexed by the irreducible representations of a maximal compact
subgroup $H_K({\mathbb R})$. Combining these facts with the notion of
support in equivariant $\K$-theory, we get

\begin{quesprop}\label{prop:suppR}
Suppose $Y$ is a closed $G({\mathbb R})$-invariant subset of
${\mathcal N}^*_{i{\mathbb R}}$ (a union of orbit closures). Write
$$\{Y_1,\cdots,Y_r\}, \quad Y_j \simeq G({\mathbb R})/H_j({\mathbb
  R})$$
for the open orbits in $Y$, and
$$\partial Y = Y - \bigcup_j Y_j$$
for their closed complement. Write finally
$$\K_{Y}^{G({\mathbb R})}({\mathcal N}^*_{i{\mathbb R}})$$
for the subspace of classes supported on $Y$.
\begin{enumerate}
\item There is a natural short exact sequence
$$0 \rightarrow \K_{\partial Y}^{G({\mathbb R})}({\mathcal
    N}^*_{i{\mathbb R}}) \rightarrow \K_{Y}^{G({\mathbb R})}({\mathcal
    N}^*_{i{\mathbb R}}) \rightarrow \sum_j \K^{G({\mathbb R})}(Y_j)
  \rightarrow 0.$$
\item There are natural isomorphisms
  $$\K^{G({\mathbb R})}(Y_j) \simeq \K^{H_j({\mathbb R})}(\point) \simeq
  {\mathbb Z}\cdot \widehat{H_{j,K}({\mathbb R})},$$
  a free abelian group with basis indexed by irreducible
  representations of a maximal compact subgroup $H_{j,K}({\mathbb
    R}) \subset H_j({\mathbb R})$.
\item The equivariant $\K$-theory space $\K^{G({\mathbb R})}({\mathcal
  N}^*_{i{\mathbb R}})$ has a ${\mathbb Z}$-basis
  $$\{e({\mathcal O}_{i{\mathbb R}},\tau)\}$$
  indexed by pairs $({\mathcal O}_{i{\mathbb R}},\tau)$, with
  $${\mathcal O}_{i{\mathbb R}} \simeq G({\mathbb R})/H({\mathbb R})
  \subset {\mathcal N}^*_{i{\mathbb R}}$$
  a nilpotent coadjoint orbit, and $\tau\in \widehat{H_K({\mathbb
      R})}$ an irreducible representation of a maximal compact
  subgroup of $H({\mathbb R})$.
  \item The basis vector $e({\mathcal O}_{i{\mathbb R}},\tau)$ is
    supported on $\overline {\mathcal O}_{i{\mathbb R}}$, and has a
    well-defined image in
  $$\K^Y_{G({\mathbb R})}/\K^{\partial Y}_{G({\mathbb
      R})};$$
  that is, it is unique up to a combination of basis vectors
  $e({\mathcal O}'_{i{\mathbb R}},\tau')$, with
  $${\mathcal O}'_{i{\mathbb R}} \subset \partial\overline{{\mathcal O}_{i\mathbb
      R}}.$$
\item Each subspace
  $$\Span\left(\{e({\mathcal O}'_{i{\mathbb R}},\tau')\mid {\mathcal
  O}'_{i{\mathbb R}} \subset \overline{{\mathcal O}_{i\mathbb
    R}} \}\right)$$
  has an explicitly computable spanning set $S({\mathcal O}_{i{\mathbb
    R}})$ expressed in the basis of Proposition \ref{prop:grR}(5).
\item Suppose that $\sigma \in \K^{G({\mathbb R})}({\mathcal
  N}^*_{i{\mathbb R}})$ is a class in equivariant $\K$-theory; write
  $$\sigma = \sum_{{\mathcal O}_{i{\mathbb R}},\tau} m_{{\mathcal
    O}_{i{\mathbb R}},\tau} e({\mathcal O}_{i{\mathbb R}},\tau).$$
  Then
$$\supp_{\mathbb R}(\sigma) = \bigcup_{\substack{{\mathcal O}_{i{\mathbb
        R}} \subset {\mathcal N}^*_{i{\mathbb R}},\\
    \text{some $m_{{\mathcal O}_{i{\mathbb R}},\tau}\ne 0$}}} \overline{{\mathcal
      O}_{i{\mathbb R}}}.$$
  \item The open orbits ${\mathcal O}_{i{\mathbb R},j}$ in
    $\supp_{\mathbb R}(\sigma)$ are the minimal
    ones so that
    $$\sigma \in \sum_j S({\mathcal O}_{i{\mathbb R},j}).$$
  \end{enumerate}
\end{quesprop}

The quotation marks are around this proposition again because we do
not know a good definition of $G({\mathbb R})$-equivariant $\K$-theory,
much less whether it has these nice properties; we will actually
prove versions in algebraic geometry (Theorem \ref{thm:finorb} and
Corollary \ref{cor:KRres}). The ``explicitly computable'' assertion is
explained in Algorithm \ref{alg:acharReal}.

Part (7) of the proposition provides a (linear algebra) computation of
the support from the expression of $\sigma$ in the basis for
$K$-theory of Proposition \ref{prop:grR}(5) ((or rather Corollary
\ref{cor:KKbasis}): we must decide whether a
vector of integers (a Kazhdan-Lusztig character formula, computed
using deep results about perverse sheaves) is in the span of other
vectors of integers (the spanning sets $S({\mathcal O}_{i{\mathbb
    R}})$, computed by much more elementary geometry in part (5)).

In case $G({\mathbb R})$ is a complex group regarded as a real group,
these two propositions (and therefore the algorithm for
\eqref{eq:WFprob}) are closely related to a conjecture of Lusztig,
proved by Bezrukavnikov in \cite{Bezr},
establishing a bijection between
some objects on nilpotent orbits (related to equivariant $\K$-theory)
and dominant weights. These ideas of Lusztig and Bezrukavnikov, and
especially Achar's work in \cite{acharTH}, guided all of our work.
This may be clearest in Section \ref{sec:cplxalgweights}, which
explains (still in the complex case) Achar's ideas for computing the spanning
set of Proposition \ref{prop:suppR}(5).

Section \ref{sec:compassvar} will explain how to
solve \eqref{eq:WFprob} for a complex reductive algebraic group.

Section \ref{sec:realalg} explains the general formalism for
extending matters to real groups. Section \ref{sec:stdK} has some
information about the geometry of cohomological induction, needed to
relate the geometry of nilpotent orbits to the Langlands
classification. This is used in Section \ref{sec:realalgweights} to
complete the proof of Proposition \ref{prop:suppR} for real groups.

\begin{subequations}\label{se:asympcycle}
The wavefront set of \eqref{eq:WFpiA} has a refinement, the
{\em wavefront cycle}:

\begin{equation}\label{eq:asympcycle}
{\mathcal W}{\mathcal F}_{\mathbb R}(\pi) =
\sum_{\substack{\text{${\mathcal O}_{i{\mathbb R}}$ open}\\ \text{ in
      $\WF_{\mathbb R}(\pi)$}}} \mu_{{\mathcal O}_{i{\mathbb R}}}(\pi)
{\mathcal O}_{i{\mathbb R}}.
\end{equation}
Here the coefficient $\mu_{{\mathcal O}_{i{\mathbb R}}}(\pi)$ is a genuine
virtual representation of a maximal compact subgroup of the isotropy
group $G({\mathbb R})_y$ of a point $y\in {\mathcal O}_{i{\mathbb
  R}}$. In the formalism explained in Proposition \ref{prop:suppR},
this means that there should be a natural definition
\begin{equation}\label{eq:asympcycleB}
\mu_{{\mathcal O}_{i{\mathbb R}}}(\pi) \in \K^{G({\mathbb
    R})}({\mathcal O}_{i{\mathbb R}});
\end{equation}
but we will actually use an algebraic geometry definition (Definition
\ref{def:assvar}). In the
setting of Proposition \ref{prop:suppR}(6), the coefficient
$\mu_{{\mathcal O}_{i{\mathbb R}}}(\pi)$ is
\begin{equation}
  \mu_{{\mathcal O}_{i{\mathbb R}}}(\pi) = \sum m_{{\mathcal
      O}_{i{\mathbb R}},\tau}(\pi)\tau;
\end{equation}
so {\em the wavefront cycle can be computed from knowledge of the
  basis vectors $e({\mathcal O}_{i{\mathbb R}},\tau)$}. But what we
actually know how to compute, as explained in Proposition
\ref{prop:suppR}(5), is not these individual basis vectors but rather
the span of all those attached to a single ${\mathcal O}_{i{\mathbb
  R}}$. For this reason we cannot compute the full wavefront
cycle.

There is a weaker invariant, the {\em weak wavefront cycle}:

\begin{equation}\label{eq:weakasympcycle}
{\mathcal W}{\mathcal F}_{\weak,{\mathbb R}}(\pi) =
\sum_{\substack{\text{${\mathcal O}_{i{\mathbb R}}$ open}\\ \text{
    in $\WF_{\mathbb R}(\pi)$}}} m_{{\mathcal O}_{i{\mathbb R}}}(\pi) {\mathcal
  O}_{i{\mathbb R}}, \qquad m_{{\mathcal O}_{i{\mathbb R}}}(\pi) =
\dim\mu_{{\mathcal O}_{i{\mathbb R}}}(\pi).
\end{equation}
Here the coefficient $m_{{\mathcal O}_{i{\mathbb R}}}(\pi)$ is just a positive
integer instead of a compact group representation. The algorithm
computing the spanning set $S({\mathcal O}_{i{\mathbb R}})$
computes the multiplicity $m_s$ for each of its spanning vectors $s$;
so the algorithm of Proposition \ref{prop:suppR} actually computes the
weak wavefront cycle as well as the wavefront set for any finite
length representation $\pi$.
\end{subequations}

\section{Kostant-Sekiguchi correspondence} \label{sec:nilp}
\setcounter{equation}{0}
\begin{subequations}\label{se:realgrps}
We work in the setting \eqref{eq:realred}, with $G({\mathbb R})$ the
group of real points of a complex connected reductive algebraic group
$G$. Write
\begin{equation}\label{eq:realform}
  \sigma_{\mathbb R}\colon G \rightarrow G, \qquad G^{\sigma_{\mathbb
      R}} = G({\mathbb R})
\end{equation}
for the Galois action. As usual we fix also a {\em compact} real form
$\sigma_0$ of $G$, so that
\begin{equation}\label{eq:theta}
  \sigma_{\mathbb R}\sigma_0 = \sigma_0\sigma_{\mathbb R}
  =_{\text{\textnormal{def}}} \theta\colon G \rightarrow G
\end{equation}
is an (algebraic) involutive automorphism of $G$, the {\em Cartan
  involution}. The group
\begin{equation}\label{eq:K}
  K  =_{\text{\textnormal{def}}} G^\theta, \qquad K({\mathbb R}) =
  K\cap G({\mathbb R})
\end{equation}
is a (possibly disconnected) complex reductive algebraic group, and
$K({\mathbb R})$ is a compact real form. What Cartan showed is that
$K({\mathbb R})$ is a maximal compact subgroup of $G({\mathbb R})$.
We write
\begin{equation}
  {\mathfrak g}({\mathbb R}) = \Lie(G({\mathbb R})), \qquad {\mathfrak
    g} = {\mathfrak g}({\mathbb R}) \otimes_{\mathbb R}{\mathbb C}
  \simeq \Lie(G),
\end{equation}
and use parallel notation for other algebraic groups. The very
familiar decomposition
\begin{equation}
  {\mathfrak g} =  {\mathfrak g}({\mathbb R}) + i {\mathfrak
    g}({\mathbb R})
\end{equation}
is the $+1$ and $-1$ eigenspaces of $\sigma_{\mathbb R}$. The analogue
for $\theta$ is the {\em Cartan decomposition}
\begin{equation}\label{eq:Cartan}
  {\mathfrak g} = {\mathfrak k} + {\mathfrak s}, \qquad {\mathfrak s}
  = {\mathfrak g}^{-\theta}.
  \end{equation}

In this setting, we can define
\begin{equation}\label{eq:gKmod}
  {\mathcal M}_f({\mathfrak g},K) = \text{finite length $({\mathfrak
        g},K)$-modules}
    \end{equation}
    (\cite{Vgreen}). An {\em invariant Hermitian form} on a $({\mathfrak
  g},K)$-module $X$ is a Hermitian bilinear form $\langle,\rangle$ on
$X$ satisfying
\begin{equation}\label{eq:gKherm} \begin{aligned}
\langle k\cdot x,y \rangle &= \langle x,\sigma_{\mathbb
  R}(k^{-1})y\rangle \qquad (x,y \in X, k\in K)\\
\langle Z\cdot x,y \rangle &= \langle x,\sigma_{\mathbb
  R}(-Z)y\rangle \qquad (x,y \in X, Z\in {\mathfrak g}).
\end{aligned}
\end{equation}

Harish-Chandra showed that many questions about functional analysis
and representations of $G({\mathbb R})$ on Hilbert spaces could be
reduced to algebraic questions about $({\mathfrak g},K)$-modules. Here
are some of his main results, and a related result of Casselman and
Wallach. We will use Theorem \ref{thm:algrep}(2) to identify the
objects of our ultimate interest (irreducible unitary representations)
with something easier (irreducible $({\mathfrak g},K)$-modules). 
\end{subequations}

\begin{theorem}\label{thm:algrep}(Harish-Chandra and Casselman-Wallach;
  see \cite{HCI}*{Theorems 2, 3, 6, 8, and 9} and \cite{WallachII}*{11.6.8})
Suppose we are in the setting \eqref{eq:GRrep} and \eqref{se:realgrps}.
    \begin{enumerate}
    \item The functor
      $$V\mapsto V_{K({\mathbb R})} \eqdef \{v\in V\mid
      \dim\Span\langle k\cdot v \mid k\in K({\mathbb R})\rangle
      <\infty\}$$
      is an equivalence of categories from ${\mathcal
        F}_{\text{\textnormal{mod}}}(G({\mathbb R}))$ to ${\mathcal
        M}_f({\mathfrak g},K)$. (Here we implicitly extend the
        differentiated action of ${\mathfrak g}({\mathbb R})$ on $V$
        to the complexification ${\mathfrak g}$, and the locally
        finite representation of the compact group $K({\mathbb R})$ to
        an algebraic representation of its complexification $K$.) In
        particular, the set $\widehat{G({\mathbb R})}$ of irreducible
        quasisimple smooth Fr\'echet representations of moderate
        growth is naturally identified with the set of irreducible
        $({\mathfrak g},K)$-modules.
        \item If $(\pi,{\mathcal H})$ is a unitary representation of $G$
          of finite length, then
          $${\mathcal H}^\infty = \{v\in {\mathcal H} \mid
          G\rightarrow {\mathcal H}, \ g\mapsto \pi(g)v\ \text{is smooth} \}$$
          is a ${\mathfrak Z}({\mathfrak g})$-finite finite length
          smooth Fr\'echet representation of moderate growth. This
          functor defines an inclusion
          $$\widehat{G({\mathbb R})}_{\text{\textnormal{unitary}}}
          \subset \widehat{G({\mathbb R})}.$$
        \item The image of the functor
          $${\mathcal H} \mapsto {\mathcal H}^\infty_{K({\mathbb
            R})}$$
          (from finite length unitary representations to  ${\mathcal
        M}_f({\mathfrak g},K)$) consists precisely of those $({\mathfrak
              g},K)$-modules $X$ admitting a positive definite
            invariant Hermitian form.
    \end{enumerate}
\end{theorem}

\begin{subequations}\label{se:nilcone}
We now describe the geometry that we will use to make geometric
invariants of finite-length representations.

The {\em complex nilpotent cone} consists of elements of ${\mathfrak g}^*$
whose orbits are weak (complex) cones (Definition \ref{def:algcone} below):
\begin{equation}\label{e:cplxnilcone}
{\mathcal N}^* = \{\xi\in {\mathfrak g}^* \mid {\mathbb C}^\times
\cdot \xi \subset G\cdot \xi\}.
\end{equation}
The {\em imaginary nilpotent cone} consists of elements of $i{\mathfrak
  g}({\mathbb R})^*$
whose orbits are (positive real) cones:
\begin{equation}\label{e:realnilcone}\begin{aligned}
    {\mathcal N}^*_{i{\mathbb R}} &= {\mathcal N}^* \cap i{\mathfrak
      g}({\mathbb R})^* \\
    &= \{\text{$-1$ eigenspace of $\sigma_{\mathbb R}$ on ${\mathcal N}^*$}\}\\
    &= \{i\xi\in i{\mathfrak g}({\mathbb R})^*
\mid {\mathbb R}_+^\times
\cdot i\xi \subset G({\mathbb R})\cdot i\xi\}.\end{aligned}
\end{equation}

It is classical that $G$ acts on ${\mathcal N}^*$ with finitely many
orbits; consequently $G({\mathbb R})$ acts on ${\mathcal
  N}^*_{i{\mathbb R}}$ with finitely many orbits.

The {\em $K$-nilpotent cone} is
\begin{equation}\label{e:Knilcone}
\begin{aligned}
{\mathcal N}^*_\theta &= {\mathcal N}^* \cap \left({\mathfrak g}/{\mathfrak
    k}\right)^* \\
&= \{\text{$-1$ eigenspace of $\theta$ on ${\mathcal N}^*$}\}\\
&= \{\xi\in {\mathfrak s}^* \mid {\mathbb C}^\times
\cdot \xi \subset K\cdot \xi\}.
\end{aligned}
\end{equation}
Kostant and Rallis proved in \cite{KR} that $K$ acts on ${\mathcal
  N}^*_\theta$ with finitely many orbits.

\end{subequations} 

We wish now to describe the Kostant-Sekiguchi relationship between
${\mathcal N}^*_{i{\mathbb R}}$ and ${\mathcal N}^*_\theta$. It is a gap
in our understanding that there is no really satisfactory description
of this relationship in terms of orbits on ${\mathfrak g}^*$; rather
we need to use an identification of ${\mathfrak g}^*$ with ${\mathfrak
  g}$.
\begin{subequations}\label{se:gg*}
Our reductive algebraic group $G$ may always be realized as a group of
matrices, in a way respecting the real form and the Cartan involution:
\begin{equation}\label{e:GGL}
  G\subset GL(n,{\mathbb C}),\quad \sigma_{\mathbb R}(g) = \overline
  g, \quad \theta(g) = {}^tg^{-1}
  \quad (g\in G).
\end{equation}
This provides first of all inclusions
\begin{equation}\label{eq:gln}\begin{aligned}
    {\mathfrak g} &\subset {\mathfrak g}{\mathfrak l}(n,{\mathbb C}),\\
    {\mathfrak k}({\mathbb R}) &\subset \text{real skew-symmetric
      matrices}\\
    {\mathfrak s} &\subset \text{complex symmetric matrices}
  \end{aligned}
\end{equation}
(and others of a similar nature)
and then an invariant bilinear form on ${\mathfrak g}$,
\begin{equation}\label{eq:form}
  \langle X,Y\rangle = \tr(XY)
\end{equation}
taking positive real values on ${\mathfrak k}({\mathbb R})$ and
negative real values on ${\mathfrak s}({\mathbb R})$, and making these
spaces orthogonal. It follows that the (complex-valued) form
$\langle,\rangle$ is {\em nondegenerate} on ${\mathfrak g}$ and makes
$\theta$ orthogonal (so that the Cartan decomposition
\eqref{eq:Cartan} is orthogonal).

We use the nondegenerate form $\langle,\rangle$ to identify
\begin{equation}\label{eq:adad*}
  {\mathfrak g} \simeq {\mathfrak g}^*, \quad X \mapsto \xi_X, \quad
  \xi_X(Y) = \langle X,Y\rangle,
\end{equation}
and so also to define a nondegenerate form (still written
$\langle,\rangle$) on ${\mathfrak g}^*$.

We define {\em adjoint nilpotent cones} by
\begin{equation}\label{e:Kadnilcone}\begin{aligned}
    {\mathcal N} &= \{\text{nilpotent\ } X\in {\mathfrak g}\} \\
    {\mathcal N}_{i{\mathbb R}} &= \{\text{nilpotent\ } X\in i{\mathfrak
      g}({\mathbb R})\} \\
  {\mathcal N}_\theta &= \{\text{nilpotent\ } X\in {\mathfrak s}\}\\
\end{aligned}\end{equation}
In each line the term ``nilpotent'' can be interpreted equivalently as
``nilpotent in ${\mathfrak g}{\mathfrak l}(n,{\mathbb C})$ (see
\eqref{eq:gln})'' or as ``$X$ belongs to $[{\mathfrak g},{\mathfrak
    g}]$ and $\ad(X)$ is nilpotent.''

The identification \eqref{eq:adad*} provides equivariant
identifications ${\mathcal N}^* \simeq {\mathcal N}$, ${\mathcal
  N}^*_{i{\mathbb R}} \simeq {\mathcal N}_{i{\mathbb R}}$, and so on. The
choice of form is unique up to a positive scalar on each simple factor
of ${\mathfrak g}$, so the identification of nilpotent adjoint and
coadjoint orbits that it provides is independent of choices.
\end{subequations}

Before discussing nilpotent orbits, we record a familiar but critical
fact about the form $\langle,\rangle$.

\begin{proposition} \label{prop:formpos}
In the setting of \eqref{se:gg*}, suppose that $H\subset G$ is a
complex maximal torus, so that ${\mathfrak h}\subset
{\mathfrak g}$ is a Cartan subalgebra. Write $X^*(H)$ for the lattice
of weights (algebraic characters) of $H$, so that
$$X^*(H) \subset {\mathfrak h}^*, \qquad {\mathfrak h}^* =
X^*(H)\otimes_{\mathbb Z} {\mathbb C}.$$
Then the bilinear form $\langle,\rangle$ has nondegenerate restriction
to ${\mathfrak h}$ and ${\mathfrak h}^*$. It is real-valued and
positive on $X^*(H)$, and therefore positive definite on the
``canonical real form''
$${\mathfrak h}^*_{\RRE} =_{\text{\textnormal{def}}}
X^*(H)\otimes_{\mathbb Z} {\mathbb R}$$
(see \cite{ALTV}*{Definition 5.5}).
\end{proposition}

\begin{theorem}[Jacobson-Morozov]\label{thm:JM}
In the setting of \eqref{se:gg*}, suppose that $\xi \in {\mathcal
  N}^*$ is a nilpotent linear functional. Define
$$E\in {\mathcal N} \subset {\mathfrak g}$$
by the requirement $\xi_E = \xi$ ({\em cf.} \eqref{eq:adad*}).
  \begin{enumerate}
    \item We can find elements $D$ and $F$ in ${\mathfrak g}$ so that
      $$[D,E] = 2E, \quad [D,F] = -2F, \quad [E,F] = D.$$
      These elements specify an algebraic map
      $$\phi = \phi_{D,E,F}\colon SL(2,\mathbb C) \rightarrow G,$$
      $$d\phi\begin{pmatrix} 1 & 0 \\ 0 & -1\end{pmatrix} = D, \quad
      d\phi\begin{pmatrix} 0 & 1 \\ 0 & 0\end{pmatrix} = E, \quad
      d\phi\begin{pmatrix} 0 & 0 \\ 1 & 0\end{pmatrix} = F.$$
      \item The element $F$ is uniquely determined up to the adjoint
        action of
        $$G^E \eqdef \{g\in G \mid \Ad(g)(E) = E\};$$
        and the elements $E$ and $F$ determine $D$ and $\phi$.
      \item The Lie algebra grading
        $${\mathfrak g}_r =_{\text{\textnormal{def}}} \{X\in
        {\mathfrak g} \mid [D,X] = rX\}$$
        is by integers. Consequently
        $${\mathfrak q} =_{\text{\textnormal{def}}} \sum_{r\ge 0}
        {\mathfrak g}_r$$
        is a parabolic subalgebra of ${\mathfrak g}$, with Levi
        decomposition
        $${\mathfrak l} = {\mathfrak g}_0 = {\mathfrak g}^D, \qquad
        {\mathfrak u} = \sum_{r>0}{\mathfrak g}_r.$$
        We write
        $$Q=LU, \quad L = G^D$$
        for the corresponding parabolic subgroup.
        \item The centralizer $G^E$ (defined in (2)) is contained in
          $Q$. Hence ${\mathfrak q}$ depends only on $E$ (and
          not on the choice of $F$ and $D$ used to define it).
          \item The Levi decomposition $Q=LU$ of $Q$ restricts to a
            Levi decomposition
            $$G^E = L^E U^E = G^{\phi(SL(2))} U^E.$$
            Here the first factor is reductive (but possibly
            disconnected), and the second is connected, unipotent, and
            normal.
          \item The orbit
            $$L\cdot E \simeq L/L^E \subset {\mathfrak g}_2$$
            is open and dense. Furthermore
            $$Q\cdot E \simeq Q/G^E = (L\cdot E) + \sum_{r>2} {\mathfrak g}_r.$$
\end{enumerate}
\end{theorem}
A convenient reference for the proof is \cite{CM}*{Theorem 3.3.1}.

It is standard to call the elements of the ``$SL(2)$ triple''
$(E,H,F)$, but we prefer to reserve the letter $H$ for algebraic
groups and particularly for maximal tori. The letter $D$ may be taken
to stand for ``diagonal,'' or just to be the predecessor of $E$ and
$F$.

\begin{corollary}[Mal'cev \cite{Mal}]\label{cor:JM}
Suppose $E$ and $E'$ are nilpotent elements of ${\mathfrak g}$; choose
Lie triples $(E,D,F)$ and $(E',D',F')$ as in Theorem \ref{thm:JM}. Then $D$
is conjugate to $D'$ if and only if $E$ is conjugate to $E'$.
\end{corollary}
\begin{proof} The assertion ``if'' follows from Theorem
  \ref{thm:JM}(2). So assume that $D$ and $D'$ are conjugate; we may
  as well assume that they are equal. Then the parabolic subalgebras
  ${\mathfrak q}$ and ${\mathfrak q}'$ are equal, along with their
  gradings. By Theorem \ref{thm:JM}(6), the two orbits $L\cdot E$
  and $L\cdot E'$ are both open and Zariski dense in ${\mathfrak
    g}_2$, so they must coincide. That is, $E'$ is conjugate to $E$ by
  $L$.\end{proof}

It was our intention to credit the preceding corollary to Jacobson and
Morozov, of whose work in the 1940s this seemed to be an immediate
corollary. A referee has suggested that more careful attribution is
appropriate. Kostant offers a proof in \cite{TDS}*{Corollary 4.2},
and he attributes the statement to Mal'cev.

\begin{theorem}[Kostant-Rallis \cite{KR}, Kostant-Sekiguchi
    \cite{Sek}]\label{thm:KR}
Use the notation of \eqref{se:realgrps}, \eqref{se:nilcone}, and
Theorem \ref{thm:JM}.
  \begin{enumerate}
\item Assume that $i\xi_{\mathbb R}\in{\mathcal N}^*_{i{\mathbb R}}$ is a real
  nilpotent element, or equivalently that the element $iE_{\mathbb R}$ belongs to
  ${\mathcal N}_{i{\mathbb R}}$. Then the element $iF_{\mathbb R}$ may also
  be chosen to belong to ${\mathcal N}_{i{\mathbb R}}$ (so that
  automatically $D_{\mathbb R} \in {\mathfrak g}({\mathbb R})$). Such
  a choice is unique up to conjugation by $G({\mathbb R})^E$. Equivalently,
  the map $\phi$ may be chosen to be defined over ${\mathbb R}$:
 $$\phi_{\mathbb R}\colon SL(2,{\mathbb R}) \rightarrow G({\mathbb
    R}).$$
    $$d\phi_{\mathbb R}\begin{pmatrix} 1 & 0 \\ 0 & -1\end{pmatrix} =
    D_{\mathbb R}, \quad d\phi_{\mathbb R}\begin{pmatrix} 0 & 1 \\ 0 &
      0\end{pmatrix} = E_{\mathbb R}, \quad d\phi_{\mathbb
        R}\begin{pmatrix} 0 & 0 \\ 1 & 0\end{pmatrix} =  F_{\mathbb R}.$$
\item With choices as in (1), the Jacobson-Morozov parabolic $Q=LU$
    of Theorem \ref{thm:JM}(2) is defined over ${\mathbb R}$. The
    Levi decomposition
    $$Q({\mathbb R}) = L({\mathbb R})U({\mathbb R})$$
    restricts to a Levi decomposition
    $$G^{iE_{\mathbb R}} = L^{iE_{\mathbb R}} U^{iE_{\mathbb R}}.$$
    The first factor is $G({\mathbb R})^{\phi_{\mathbb R}}$,
    a (possibly disconnected) real reductive algebraic group, and the
    second factor is connected, unipotent, and normal.
 \item The orbit
    $$L({\mathbb R})\cdot iE_{\mathbb R} \simeq L({\mathbb
      R})/L({\mathbb R})^{iE_{\mathbb R}} \subset i{\mathfrak
      g}_2({\mathbb R})$$
    is open but not necessarily dense. The open orbits of $L({\mathbb
      R})$ on this vector space are in one-to-one correspondence with
    the $G({\mathbb R})$ orbits on ${\mathcal N}_{i{\mathbb R}}$ having
   associated semisimple element (Corollary \ref{cor:JM}) conjugate to
   $D_{\mathbb R}$. Furthermore
    $$Q({\mathbb R})\cdot iE_{\mathbb R} \simeq Q({\mathbb
      R})/Q({\mathbb R})^{iE_{\mathbb R}} = \left(L({\mathbb R})\cdot
    iE_{\mathbb R}\right) + \sum_{r>2} i{\mathfrak g}({\mathbb R})_r.$$
  \item After replacing $(iE_{\mathbb R},\phi_{\mathbb R})$ by a
    conjugate $(iE_{{\mathbb R},\theta},\phi_{{\mathbb R},\theta})$ under
    $G({\mathbb R})$, we may assume that the map $\phi$ also respects
    the Cartan involution:
    $$\phi_{{\mathbb R},\theta}({}^t g^{-1}) = \theta\left(\phi_{{\mathbb
        R},\theta}(g)\right), \qquad iF_{{\mathbb R},\theta} =
    -\theta(iE_{{\mathbb R},\theta})$$
    \item Assume that $\xi_{\theta}\in{\mathcal N}^*_{\theta}$ is a
      $K$-nilpotent element, or equivalently that the element $E_{\theta}$
      belongs to ${\mathcal N}_{\theta} \subset {\mathfrak s}$ (the $-1$
      eigenspace of $\theta$). Then the element
      $F_{\theta}$ may also be chosen in ${\mathfrak s}$, and
      in this case $D_\theta$ belongs to ${\mathfrak k}$. Such choices
  are unique up to conjugation by $K^{E_\theta}$. They define
  an algebraic map {\small
      $$\phi_\theta \colon SL(2,\mathbb C) \rightarrow G,$$
      $$d\phi_\theta\begin{pmatrix} 0 & i \\ -i & 0\end{pmatrix} =
  D_\theta, \quad
  {\frac{1}{2}}\cdot d\phi_\theta\begin{pmatrix} 1 & -i \\ -i &
  -1\end{pmatrix} = E_\theta, \quad {\frac{1}{2}}\cdot
  d\phi_\theta\begin{pmatrix} 1 & i \\ i & -1\end{pmatrix} = F_\theta$$}
which respects $\theta$:
$$\phi_{\theta}({}^tg^{-1}) = \theta\left(\phi_\theta(g)\right).$$
  \item With choices as in (5), the Jacobson-Morozov parabolic
    $Q_\theta=L_\theta U_\theta$ (defined as in Theorem
    \ref{thm:JM}(2) using $D_\theta$) is $\theta$-stable. The
    Levi decomposition
    $$Q_\theta\cap K = (L_\theta\cap K)(U_\theta\cap K)$$
    restricts to a Levi decomposition
    $$K^{E_\theta} = (L_\theta\cap K)^{E_\theta}(U_\theta\cap K)^{E_\theta}.$$
    The first factor is $K^{\phi_\theta}$,
    a (possibly disconnected) complex reductive algebraic group, and
    the second factor is connected, unipotent, and normal.
 \item The orbit
    $$(L_\theta\cap K)\cdot E_\theta \simeq
    (L_\theta\cap K)/(L_\theta\cap K)^{E_\theta} \subset {\mathfrak s}_2$$
    is open and (Zariski) dense. Furthermore
    $$(Q_\theta\cap K)\cdot E_\theta \simeq (Q_\theta\cap
    K)/(Q_\theta\cap K)^{E_\theta} = \left((L_\theta\cap K)\cdot
    E_\theta\right) + \sum_{r>2} {\mathfrak s}_r.$$
  \item After replacing $(E_\theta,\phi_\theta)$ by a conjugate
    $(E_{\theta,{\mathbb R}},\phi_{\theta,{\mathbb R}})$ under
    $K$, we may assume that the map $\phi_\theta$ also respects
    the real form:
    $$\phi_{\theta,{\mathbb R}}(\overline g) = \sigma_{\mathbb
      R}(\phi_{\theta,{\mathbb R}}(g)), \qquad F_{\theta,{\mathbb R}} =
    \sigma_{\mathbb R}(E_{\theta,{\mathbb R}}).$$
  \end{enumerate}
\end{theorem}

\begin{corollary}
Suppose that $iE_{\mathbb R}$ and $iE'_{\mathbb R}$ are nilpotent
  elements of $i{\mathfrak g}({\mathbb R})$; choose
Lie triples $(iE_{\mathbb R},D_{\mathbb R},iF_{\mathbb R})$ and
$(iE_{\mathbb R}',D_{\mathbb R}',iF_{\mathbb R}')$ as in Proposition
\ref{thm:KR}.
\begin{enumerate}
\item The  semisimple Lie algebra elements
  $$iE_{\mathbb R} - iF_{\mathbb R}\quad \text{and}\quad iE_{\mathbb R}' -
  iF_{\mathbb R}'$$
 are conjugate by $G({\mathbb R})$ if and only if
  $iE_{\mathbb R}$ is conjugate to $iE_{\mathbb R}'$ by $G({\mathbb R})$
\item Suppose $E_{\theta}$ and $E'_{\theta}$ are nilpotent
  elements of ${\mathfrak s}$; choose
Lie triples $(E_{\theta},D_{\theta},F_{\theta})$ and
$(E_{\theta}',D_{\theta}',F_{\theta}')$ as in Proposition
\ref{thm:KR}. Then $D_{\theta}$ is conjugate to $D_{\theta}'$
by $K$ if and only if $E_{\theta}$ is conjugate to $E_{\theta}'$ by
$K$.
\end{enumerate}
\end{corollary}

The first assertion is not quite immediate from the proposition, and
we will not use it; we include it only to show that there {\em is} a
way of parametrizing real nilpotent classes by real semisimple
classes.

\begin{corollary}\label{cor:KS}
In the setting of \eqref{se:realgrps} and
\eqref{se:nilcone}, there are bijections among the following sets:
\begin{enumerate}
\item $G({\mathbb R})$ orbits on ${\mathcal N}^*_{i{\mathbb R}}$;
\item $G({\mathbb R})$ orbits on ${\mathcal N}_{i{\mathbb R}}$;
\item $G({\mathbb R})$ orbits of group homomorphisms
  $$\phi_{\mathbb R}\colon SL(2,{\mathbb R}) \rightarrow G({\mathbb
  R});$$
\item $K({\mathbb R})$ orbits of group homomorphisms
  $$\phi_{{\mathbb R},\theta}\colon SL(2,{\mathbb R}) \rightarrow G({\mathbb
  R})$$
sending inverse transpose to the Cartan involution $\theta$;
\item $K({\mathbb R})$ orbits of group homomorphisms
  $$\phi\colon SL(2) \rightarrow G$$
  sending inverse transpose to the Cartan involution $\theta$, and
  sending complex conjugation to $\sigma_{\mathbb R}$;
\item $K$ orbits of group homomorphisms
  $$\phi\colon SL(2) \rightarrow G$$
sending inverse transpose to the Cartan involution $\theta$;
\item $K$ orbits on ${\mathcal N}_\theta$; and
\item $K$ orbits on ${\mathcal N}^*_\theta$.
\end{enumerate}
The correspondences (1)$\leftrightarrow$(2) and
(7)$\leftrightarrow$(8) are given by \eqref{se:gg*};
(2)$\leftrightarrow$(3) by Theorem \ref{thm:KR}(2); and so on.

All the maximal compact subgroups of the isotropy groups for the orbits
above are naturally isomorphic, with isomorphisms defined up to inner
automorphisms, to $K({\mathbb R})^{\phi_{{\mathbb R},\theta}}$.

The bijection (1)$\leftrightarrow$(8) preserves the closure relations
between orbits. Corresponding orbits ${\mathcal O}_{\mathbb R}$ and
${\mathcal O}_\theta$ are $K({\mathbb R})$-equivariantly diffeomorphic.
\end{corollary}

The assertions in the last paragraph are due to Barbasch-Sepanski
\cite{BS} and Vergne \cite{Vergne} respectively.

The bijection may also be characterized by either of the following
equivalent conditions:
\begin{equation}\begin{aligned}
\frac{1}{2}\left(-iE_{\mathbb R} - iF_{\mathbb R} +D_{\mathbb
  R}\right) &\text{\ is conjugate by $K$ to\ } E_\theta\\
iE_{\mathbb R} - iF_{\mathbb R} &\text{\ is conjugate by $K$ to\ } D_\theta.
\end{aligned}\end{equation}
This formulation makes clear what is slightly hidden in the formulas
of Theorem \ref{thm:KR}(5): that the bijection {\em does not
  depend on a chosen square root of -1}.
Changing the choice replaces $iE_{\mathbb R}$ by $-iE_{\mathbb R}$,
and therefore twists the $SL(2,{\mathbb R})$ homomorphism
$\phi_{\mathbb R}$ by inverse
transpose. At the same time $E_\theta$ and $F_\theta$ are
interchanged, which has the effect of twisting $\phi_\theta$ by
inverse transpose.

\begin{definition}\label{def:KS}
If ${\mathcal O}$ is a $G$-orbit on ${\mathcal N}^*$, then a
$G({\mathbb R})$ orbit
$${\mathcal O}_{\mathbb R}\subset {\mathcal N}^*_{i{\mathbb R}} \cap
{\mathcal O}$$
is called a {\em real form} of ${\mathcal O}$. We will call a $K$ orbit
$${\mathcal O}_\theta\subset {\mathcal N}^*_\theta \cap {\mathcal O}$$
a {\em $\theta$ form} of ${\mathcal O}$. The Kostant-Sekiguchi theorem
says that there is a natural bijection between real forms and
$\theta$ forms.
\end{definition}

\begin{definition}\label{def:geom}
A {\em (global) geometric parameter} for $(G,K)$ is a nilpotent $K$-orbit
$Y \subset {\mathcal N}^*_\theta$, together with an
irreducible $K$-equivariant vector bundle
$${\mathcal E} \rightarrow Y.$$
Equivalently, a {\em (local) geometric parameter} is a $K$-conjugacy
class of pairs
$$(\xi,(\tau,E)),$$
with $\xi\in {\mathcal N}^*_\theta$ a nilpotent element,
and $(\tau,E)$ an irreducible (algebraic) representation of the isotropy group
$K^\xi$. This bijection between local and global parameters identifies
$(\xi,(\tau,E))$ with the pair
$$Y = K\cdot \xi \simeq K/K^\xi, \qquad {\mathcal E} \simeq
K\times_{K^\xi} E.$$
We write ${\mathcal P}_{g}(G,K)$ for the collection of
geometric parameters.
\end{definition}

\section{Asymptotic cones}\label{sec:cones}
\setcounter{equation}{0}
This section is a digression, intended as another kind of motivation
for the orbit method. The (very elementary)
ideas play no role in the proofs of our main theorems. They appear
only in the desideratum \eqref{eq:orbrep} for deciding which
representations might reasonably be attached to which coadjoint
orbits. In order to provide some mathematical excuse for the material, we
will include a single serious conjecture (Conjecture
\ref{conj:autrep}) about automorphic forms.

\begin{subequations}\label{se:realasymp}
Suppose
\begin{equation}\label{eq:coord}
  V \simeq {\mathbb R}^n
\end{equation}
is a finite-dimensional real vector space. A {\em ray} in $V$ is by
definition a subset
\begin{equation}
  R(v) = {\mathbb R}_{\ge 0}\cdot v \subset V \qquad (0\ne v\in V).
\end{equation}
We write
\begin{equation}
  {\mathcal R}(V) = \{\text{rays in $V$} \} \simeq S^{n-1};
\end{equation}
an isomorphism with the $(n-1)$-sphere is induced by an isomorphism
\eqref{eq:coord}. The resulting smooth manifold structure on
${\mathcal R}(V)$ is of course independent of the isomorphism. There
is a natural fiber bundle
\begin{equation}\label{eq:tautR}
{\mathcal B}(V) = \{(v,r)\mid r\in {\mathcal R}(V),\ v\in r\}
\buildrel{\pi}\over{\longrightarrow} {\mathcal R}(V),
\end{equation}
the {\em tautological ray bundle} over ${\mathcal R}(V)$. Projection
on the first factor defines a proper map
\begin{equation}
  {\mathcal B}(V)\buildrel{\mu}\over{\longrightarrow} V, \quad (v,r)
  \mapsto v;
\end{equation}
the map $\mu$ is an isomorphism over the preimage of $V\backslash
\{0\}$ (consisting of the open rays in the bundle), and $\mu^{-1}(0) =
     {\mathcal R}(V)$ (the compact sphere).
\end{subequations} 

\begin{definition}\label{def:cone} In the setting of
  \eqref{se:realasymp}, a {\em cone} $C \subset V$ is any subset
  closed under scalar multiplication by ${\mathbb R}_{\ge 0}$. A {\em
    weak cone} is any subset closed under scalar multiplication by
  ${\mathbb R}_{>0}$.
\end{definition}

\begin{definition} \label{def:asympcone} In the setting of
  \eqref{se:realasymp}, suppose
  $S\subset V$ is an arbitrary subset. The {\em asymptotic cone of S}
  is
  $$\Cone_{\mathbb R}(S) = \{v \in V \mid \exists \epsilon_i \rightarrow +0, s_i\in
  S, \lim_{i\rightarrow \infty} \epsilon_is_i = v\}.$$
  Here $\{\epsilon_i\}$ is a sequence of positive real numbers going to
  $0$, and $s_i$ is any sequence of elements of $S$.
\end{definition}

Here are some elementary properties of the asymptotic cone.
\begin{enumerate}\label{facts:asympcone}
\item The set $\Cone_{\mathbb R}(S)$ is a closed cone.
\item The set $\Cone_{\mathbb R}(S)$ is nonempty if and only if $S$ is nonempty.
\item The set $\Cone_{\mathbb R}(S)$ is contained in $\{0\} \subset V$
  if and only if $S$ is bounded.
\item If $C$ is a weak cone (Definition \ref{def:cone}), then the
  asymptotic cone is the closure of $C$:
  $$\Cone_{\mathbb R}(C) = \overline C.$$
\end{enumerate}
The last assertion includes \eqref{eq:asympnilp} from the
introduction.

Here are the same ideas in the setting of algebraic geometry.
\begin{subequations}\label{se:algasymp}
Suppose
\begin{equation}\label{eq:algcoord}
  V \simeq {\mathbb C}^n
\end{equation}
is a finite-dimensional complex vector space. We write
\begin{equation}
  {\mathbb P}(V) = \{\text{(complex) lines through the origin in $V$} \}.
\end{equation}
There is a natural line bundle
\begin{equation}\label{eq:tautalg}
{\mathcal O}(-1)(V) = \{(v,\ell)\mid \ell \in {\mathbb P}(V), v\in \ell\}
\buildrel{\pi}\over{\longrightarrow} {\mathbb P}(V),
\end{equation}
the {\em tautological line bundle} over ${\mathbb P}(V)$. Projection
on the first factor defines a proper map
\begin{equation}
  {\mathcal O}(-1)(V)\buildrel{\mu}\over{\longrightarrow} V, \quad (v,\ell)
  \mapsto v;
\end{equation}
the map $\mu$ is an isomorphism over the preimage of $V\backslash
\{0\}$ (consisting of the bundle minus the zero section), and $\mu^{-1}(0) =
     {\mathbb P}(V)$.
\end{subequations} 

\begin{definition}\label{def:algcone} In the setting of
  \eqref{se:algasymp}, a {\em cone} $C \subset V$ is any subset
  closed under scalar multiplication by ${\mathbb C}$. A {\em
    weak cone} is any subset closed under scalar multiplication by
  ${\mathbb C}^\times$.
\end{definition}

\begin{definition} \label{def:algasymp} In the setting of
  \eqref{se:algasymp}, suppose
  $S\subset V$ is an arbitrary subset. Define
  $$I(S) = \{p\in \Poly(V) \mid p|_S = 0\},$$
  the ideal of polynomial functions vanishing on $S$. This ideal is
  filtered by the degree filtration on polynomial functions, so we can
  define
  $$\gr I(S) \subset \Poly(V)$$
  a graded ideal. This is the ideal generated by the highest degree
  term of each nonzero polynomial vanishing on $S$. The {\em algebraic
    asymptotic cone of S} is
  $$\Cone_{\text{\textnormal{alg}}}(S) = \{v \in V\mid q(v) = 0
  \text{\ all\ } q\in \gr I(S)\},$$
  a Zariski-closed cone in $V$; or, equivalently, a closed subvariety
  of ${\mathbb P}(V)$.
\end{definition}

This definition looks formally like the definition of the
tangent cone to $S$ at $\{0\}$ (see for example \cite{Harris}*{Lecture
  20}). In that definition one considers the graded ideal generated by
{\em lowest} degree terms in the ideal of $S$.

Here are some elementary properties of the algebraic asymptotic cone.
\begin{enumerate}\label{facts:algcone}
\item The set $\Cone_{\text{\textnormal{alg}}}(S)$ is a closed cone,
  of dimension equal to $\dim \overline S$ (the Krull dimension of the
  Zariski closure of $S$).
\item The set $\Cone_{\text{\textnormal{alg}}}(S)$ is nonempty if and
  only if $S$ is nonempty.
\item The set $\Cone_{\text{\textnormal{alg}}}(S)$ is contained in
  $\{0\} \subset V$ if and only 
  if $S$ is finite.
\item If $C$ is a weak cone (Definition \ref{def:cone}), then the
  asymptotic cone is the Zariski closure of $C$:
  $$\Cone_{\text{\textnormal{alg}}}(C) = \overline C.$$
  \item If $S$ is a constructible algebraic set (finite union of
    Zariski closed intersect Zariski open) then
$$\Cone_{\text{\textnormal{alg}}}(S) = \{v \in V \mid \exists
    \epsilon_i \rightarrow 0, s_i\in S, \lim_{i\rightarrow \infty}
    \epsilon_is_i = v\}.$$
Here $\{\epsilon_i\}$ is a sequence of nonzero complex numbers going
to zero, and $\{s_i\}$ is any sequence of elements of $S$.
\end{enumerate}

Finally, we note that the definition given for asymptotic cones over
${\mathbb R}$ extends to any local field. It is not quite clear what
dilations ought to be allowed ``in general''; we make a choice that
behaves well for the coadjoint orbits we are interested in.

\begin{subequations}\label{se:localasymp}
Suppose $K$ is a local field of characteristic not $2$, and
\begin{equation}\label{eq:localcoord}
  V \simeq {\mathbb K}^n
\end{equation}
is a finite-dimensional $K$-vector space. A {\em ray} in $V$ is by
definition a subset
\begin{equation}
  R(v) = (K^\times)^2\cdot v \subset V \qquad (0\ne v\in V).
\end{equation}
We write
\begin{equation}
  {\mathcal R}(V) = \{\text{rays in $V$} \} \rightarrow {\mathbb P}(V);
\end{equation}
the map is $\#(K^\times/(K^\times)^2)$ to one. It follows that there
is a natural compact $K$-manifold topology on ${\mathcal R}(V)$, of
dimension equal to $n-1$. There
is a natural fiber bundle
\begin{equation}\label{eq:tautlocal}
{\mathcal B}(V) = \{(v,r)\mid r\in {\mathcal R}(V), v\in r\}
\buildrel{\pi}\over{\longrightarrow} {\mathcal R}(V), \quad (v,r)
\mapsto r
\end{equation}
the {\em tautological ray bundle} over ${\mathcal R}(V)$. Projection
on the first factor defines a proper map
\begin{equation}
  {\mathcal B}(V)\buildrel{\mu}\over{\longrightarrow} V, \quad (v,r)
  \mapsto v;
\end{equation}
the map $\mu$ is an isomorphism over the preimage of $V\backslash
\{0\}$ (consisting of the open rays in the bundle), and $\mu^{-1}(0) =
     {\mathcal R}(V)$ (the compact space of all rays in $V$).
\end{subequations} 

\begin{definition}\label{def:localcone} In the setting of
  \eqref{se:localasymp}, a {\em cone} $C \subset V$ is any subset
  closed under scalar multiplication by $K^2$. A {\em
    weak cone} is any subset closed under scalar multiplication by
  $(K^{\times})^2$.
\end{definition}

\begin{definition} \label{def:localasympcone} In the setting of
  \eqref{se:localasymp}, suppose
  $S\subset V$ is an arbitrary subset. The {\em asymptotic cone of S}
  is
  $$\Cone_{K}(S) = \{v \in V \mid \exists \epsilon_i \rightarrow 0, s_i\in
  S, \lim_{i\rightarrow \infty} \epsilon_is_i = v\}.$$
  Here $\{\epsilon_i\}$ is a sequence in $(K^\times)^2$ going to
  $0$, and $s_i$ is any sequence of elements of $S$.
\end{definition}

Here are some elementary properties of the asymptotic cone.
\begin{enumerate}\label{facts:localcone}
\item The set $\Cone_{K}(S)$ is a closed cone.
\item The set $\Cone_{K}(S)$ is nonempty if and only if $S$ is nonempty.
\item The set $\Cone_{K}(S)$ is contained in $\{0\} \subset V$
  if and only if $S$ is bounded.
\item If $C$ is a weak cone (Definition \ref{def:localcone}), then the
  asymptotic cone is the closure of $C$:
  $$\Cone_{K}(C) = \overline C.$$
\end{enumerate}

\begin{subequations}\label{seq:padic}
The reformulation in \cite{BV} of Howe's definition from \cite{HoweWF}
of the wavefront set of a
representation (see \eqref{eq:WFpiA}) extends to groups over other local
fields, by means of the germ expansion of characters. If
$G$ is a reductive algebraic group defined over a $p$-adic field $K$ of
characteristic zero, then we write
\begin{equation}\label{eq:localnilcone}\begin{aligned}
    {\mathcal N}^*_K &= \text{nilpotent elements in\ }{\mathfrak g}(K)^*\\
      &= \{\xi\in {\mathfrak g}(K)^* \mid t^2\xi \in \Ad^*(G(K))(\xi)
      \ (t\in K^\times)\}\\
\end{aligned}
\end{equation}
for the nilpotent cone in the dual of the Lie algebra. (That a nilpotent
linear functional $\xi$ is conjugate to $t^2\xi$ can be proved by
transferring the statement to ${\mathfrak g}(K)$ as in
\eqref{eq:adad*}, and using the Jacobson-Morozov theorem to reduce to
the case of $SL(2)$.) Dilation therefore defines an action of the
finite group $K^\times/(K^\times)^2$ on $G(K)$ orbits in ${\mathcal
  N}^*_K$.

The Fourier transform is an isomorphism
\begin{equation}\label{eq:localFT}\begin{aligned}
 C_c^\infty({\mathfrak g}(K)) &\rightarrow C_c^\infty({\mathfrak g}(K)^*)\\
\widehat{f}(\xi) &= \int_{X \in {\mathfrak g}(K)} f(X) \psi(\xi(X)) dX.
  \end{aligned}
\end{equation}
Here $dX$ is a choice of Lebesgue measure on ${\mathfrak g}(K)$:
scaling the measure scales the values of the Fourier
transform (and so does not change its support). Furthermore $\psi$ is
a nontrivial additive unitary character of $K$,
which is therefore unique up to dilation by $t\in K^\times$. Dilating
$\psi$ dilates the Fourier transform as a function on ${\mathfrak
  g}(K)^*$, and therefore dilates the support.

Taking the transpose of the Fourier transform defines an isomorphism of
distributions
\begin{equation}\label{eq:localdualFT}\begin{aligned}
 C^{-\infty}({\mathfrak g}(K))^* &\rightarrow C^{-\infty}({\mathfrak g}(K))\\
\widecheck{D}(f) &= D(\widehat f).
  \end{aligned}
\end{equation}
Each nilpotent coadjoint orbit ${\mathcal O} \subset {\mathcal N}^*_K$
carries a natural $G(K)$-invariant measure $d\mu({\mathcal O})$. This
distribution has a Fourier transform
\begin{equation}\label{eq:OFT}
  \widecheck{\mathcal O} \in C^{-\infty}({\mathfrak g}(K)).
\end{equation}
This map from orbits to generalized functions will {\em change}
according to the $K^\times/(K^\times)^2$ action on orbits when the
character $\psi$ of $K$ changes.

Harish-Chandra proved in \cite{HCp} that for every smooth admissible
irreducible representation $\pi$ of $G(K)$, the distribution character
$\Theta_\pi$ has a unique {\em germ expansion}
\begin{equation}\label{eq:pexp}
  \Theta_\pi(\exp(X)) = \sum_{{\mathcal O} \in {\mathcal N}^*_K/G(K)}
  c_{\mathcal O}\widecheck {\mathcal O},
\end{equation}
valid for sufficiently small $X\in {\mathfrak g}(K)$. We may therefore define
\begin{equation}\label{eq:localWF}
  \WF(\pi) = \bigcup_{c_{\mathcal O} \ne 0}
     \overline {\mathcal O} \subset {\mathcal N}^*_K,
\end{equation}
a $G(K)$-invariant closed cone. (Our understanding is that there may
be a way to make sense of \eqref{eq:localWF} also for $K$ local of positive
characteristic, but that it is not completely established.)

 \end{subequations}
 \begin{conjecture} {(global coherence of WF sets)}\label{conj:autrep}
   Suppose that $k$ is a number field, and that $G$ is a reductive
   algebraic group defined over $k$. Suppose that
   $$\pi = \otimes_v \pi_v$$
is an automorphic representation. This means (among other things) that
$\{v\}$ is the set of places of $k$, so that each corresponding
completion
$$k_v \supset k$$
is a local field, $G(k_v)$ is a reductive group over a local field as
above, and $\pi_v$ is a smooth irreducible admissible representation
of $G(k_v)$. Accordingly for each place we get a closed
$G(k_v)$-invariant cone
$$\WF(\pi_v) \subset {\mathcal N}^*_{k_v} \subset {\mathfrak
  g}(k_v)^*.$$
Suppose that all the local characters $\psi_v$ of $k_v$ (used in
defining the Fourier transforms behind the wavefront sets) are chosen
in such a way that
$$\prod_v \psi_v(x) = 1, \qquad (x\in k).$$
The conjecture is that there is a coadjoint orbit (depending on $\pi$)
$${\mathcal O} = G(k)\cdot \xi \subset {\mathfrak g}(k)^*$$
with the property that
$$\WF(\pi_v) = \Cone_{k_v}({\mathcal O})$$
for every place $v$.
 \end{conjecture}

 It is easy to see (by considering characteristic polynomials, for
example) that
$$\Cone_{k_v}({\mathcal O}) \subset {\mathcal N}^*_{k_v}.$$
For similar reasons,
$$\Cone_{\text{\textnormal{alg}}}(G(\overline k)\cdot \xi) =
Z \subset {\mathfrak g}(\overline k)^*,$$
the Zariski closure of a single nilpotent coadjoint orbit $Z^0$ over the
algebraic closure $\overline k$. It follows that
$$\Cone_{k_v}({\mathcal O}) \subset Z(k_v),$$
the $k_v$-points of $Z$. But this local cone need not meet the open
orbit $Z^0$. (We believe it should be possible to show that
$\Cone_{k_v}({\mathcal O})$ {\em does} meet $Z^0$ for
all but finitely many $v$.)

 The local cones control expansions of the local characters near the
identity. It would be nice if there were some global analogue of these
local expansions, involving orbits like $G(k)\cdot \xi$. But we can
offer no suggestion about how to formulate any such thing.

Since all the local cones are nilpotent, it is natural to ask whether
the global orbit $G(k)\cdot\xi$ in the conjecture can be taken to be
nilpotent. This is {\em not} possible. If $G$ is anisotropic over $k$
(admitting no nontrivial $k$-split torus) then ${\mathcal N}^*_k =
\{0\}$; so the conjecture would require that all the
local factors $\pi_v$ of any automorphic representation would have to
be finite-dimensional. This does not happen (for nonabelian $G$).

\section{Equivariant $\K$-theory}\label{sec:Kthy}
\setcounter{equation}{0}
In this section we recall the algebraic geometry version of
equivariant $\K$-theory, which for us will replace the $G({\mathbb
  R})$-equivariant theory discussed in the introduction (for which we
lack even many definitions, and certainly lack proofs of good
properties).

\begin{subequations}\label{se:repring}

Suppose $H$ is a complex algebraic group. Write $H^{\unip}$ for the
unipotent radical of $H$, and
\begin{equation}
  H^{\red} \subset H, \qquad H^{\red} \simeq H/H^{\unip}
\end{equation}
for a choice of Levi subgroup. We are interested in the category
\begin{equation}
\Rep(H) = \text{finite-dimensional algebraic representations of $H$}.
\end{equation}
The Grothendieck group of this category is called the {\em
  representation ring of $H$}, and written
\begin{equation}
\R(H) \eqdef \K\Rep(H).
\end{equation}
We use the bold $\R$ as a reminder that this is a Grothendieck group,
one of the $\K$-theory tools that we use constantly.

If we write
\begin{equation}\begin{aligned}
    \widehat H &= \text{equivalence classes of irreducible}\\
    &\qquad \text{algebraic representations of $H$}\\
    &= \widehat{H^{\red}},\\
\end{aligned}
\end{equation}
then elementary representation theory says that
\begin{equation}\label{e:RHirr}
\R(H) = \K\Rep(H) \simeq \sum_{\rho \in \widehat H}
{\mathbb Z}\rho \simeq \R(H^{\red}).
\end{equation}

The ring structure on $\R(H)$ arises from tensor product of
representations. Another way to say it is using the {\em character} of
a representation $(\tau,E)$ of $H$:
\begin{equation*}
\Theta_E(h) = \tr\tau(h) \qquad (h\in H).
\end{equation*}
Clearly the character $\Theta_E$ is an algebraic class function on
$H$. If $(\sigma,S)$ and $(\kappa,Q)$ are representations forming a
short exact sequence
\begin{equation*}
0 \rightarrow S \rightarrow E \rightarrow Q \rightarrow 0,
\end{equation*}
then $\Theta_E = \Theta_S +\Theta_Q$. Consequently $\Theta$ descends
to a ${\mathbb Z}$-linear map
\begin{equation}\label{e:character}
\Theta\colon \R(H) \hookrightarrow \text{algebraic class functions on $H$.}
\end{equation}
The reason for the injectivity is the standard fact that characters of
inequivalent irreducible representations are linearly independent
functions on $H$. The trace of the tensor product of two linear maps
is the product of the traces, so $\Theta_{E\otimes F} =
\Theta_E\Theta_F$. Therefore the map of \eqref{e:character} is also a
ring homomorphism.  Because irreducible representations are trivial on
the unipotent radical $H^{\unip}$ of $H$, we get finally
\begin{equation*}
\Theta\colon \R(H) \hookrightarrow \text{algebraic class functions on }
  H/H^{\unip} \simeq H^{\red}.
\end{equation*}
After complexification, this map turns out to be an algebra isomorphism
\begin{equation}\label{e:characterss}
\Theta\colon \R(H)\otimes_{\mathbb Z} {\mathbb C}
\buildrel{\sim}\over\longrightarrow  \text{algebraic class functions on
  $H^{\red}$.}
\end{equation}

\end{subequations} 

\begin{definition}\label{def:eqvtcoh} Suppose $X$ is a complex
  algebraic variety and suppose that $H$ is a complex algebraic group
  acting on $X$. We are interested in the abelian category
$$\QCoh^H(X)$$
of $H$-equivariant quasicoherent sheaves on $X$, and particularly in
the full subcategory
$$\Coh^H(X)$$
of coherent sheaves. We need this theory for cases in
which $X$ is {\em not} an affine variety. In that setting the precise
definitions (found in \cite{Tho}*{1.2}) are a little difficult to make
sense of.  We will be able
to work largely with the special case when $X$ is affine; so we recall
here just the more elementary description of these equivariant sheaves
in the affine case.

Suppose therefore that $X$ is a complex {\em affine} algebraic
variety, with structure sheaf ${\mathcal O}_X$ and global ring of
functions ${\mathcal O}_X(X)$. Suppose that $H$ is an affine algebraic
group acting on $X$. This means precisely that the algebra ${\mathcal
  O}_X(X)$ is equipped with an algebraic action
$$\Ad\colon H \rightarrow \Aut({\mathcal O}_X(X))$$
by algebra automorphisms. A quasicoherent sheaf ${\mathcal M} \in \QCoh(X)$ is
completely determined by the ${\mathcal O}_X(X)$-module
$$M = {\mathcal M}(X)$$
of global sections: any ${\mathcal O}_X(X)$-module $M$
  determines a quasicoherent sheaf
$${\mathcal M} = {\mathcal O}_X \otimes_{{\mathcal O}_X(X)} M$$
of ${\mathcal O}_X$-modules. This makes an equivalence of categories
$$\QCoh(X) = {\mathcal O}_X\Mod \simeq {\mathcal O}_X(X)\Mod.$$
The sheaf ${\mathcal M}$ is coherent if and only if $M$ is finitely
generated:
$$\Coh(X) \simeq {\mathcal O}_X(X)\Modfg.$$

Write $({\mathcal O}_X(X),H)\Mod$ for the category of ${\mathcal
  O}_X(X)$-modules $M$ equipped with
an algebraic action of $H$. This means an algebraic representation
$\mu$ of $H$ on $M$, such that the module action map
$${\mathcal O}_X(X)\otimes_{\mathbb C} M \rightarrow M$$
intertwines $\Ad\otimes \mu$ with $\mu$. We have now identified
$$\begin{aligned}
\QCoh^H(X) &\simeq ({\mathcal O}_X(X),H)\Mod\\
\Coh^H(X) &\simeq ({\mathcal O}_X(X),H)\Modfg, \end{aligned}$$
always for $X$ affine.

The {\em equivariant $\K$-theory of $X$} is the Grothendieck group
$\K^H(X)$ of $\Coh^H(X)$. (Thomason and many others prefer to write
this $\K$-theory as ${\mathbf G}^H(X)$, reserving $\K$ for the theory
with vector bundles replacing coherent sheaves. Our notation follows the text
\cite{CG}.)
\end{definition}

\begin{subequations}\label{se:eqvtK}
If ${\mathcal M}$ is an $H$-equivariant coherent sheaf on $X$, we
write
\begin{equation}
  [{\mathcal M}] \in \K^H(X)
\end{equation}
for its class in $\K$-theory. Similarly, if $X$ is affine and $M$ is a
finitely generated $H$-equivariant ${\mathcal O}_X(X)$-module, we
write
\begin{equation}
  [M] \in \K^H(X).
\end{equation}

We record now some general facts about equivariant $\K$-theory of which we will
make constant use.

Suppose $X$ is a single point. Then a coherent sheaf on $X$ is the
same thing as a finite-dimensional vector space $E$; an
$H$-equivariant structure is an algebraic representation $\rho$ of $H$
on $E$. That is, there is an equivalence of categories
\begin{equation}\label{e:Cohpt}
\Coh^H(\point) \simeq \Rep(H)
\end{equation}
(see \eqref{se:repring}). Consequently
\begin{equation}\label{e:Kpt}
\K^H(\point)\simeq \K\Rep(H) \simeq \sum_{\rho \in \widehat H}
{\mathbb Z}\rho = \R(H),
\end{equation}
the representation ring of $H$.  More generally, if $(\tau,V)$ is a
finite-dimensional algebraic representation of $H$, then there is a
natural inclusion
\begin{equation}\label{e:Cohvec}
\Rep(H) \hookrightarrow (\Poly(V),H)\Modfg \simeq \Coh^H(V), \quad E
\mapsto E\otimes \Poly(V)
\end{equation}
(with $\Poly(V) = {\mathcal O}_V(V)$ the algebra of polynomial functions
on $V$). In contrast to \eqref{e:Cohpt}, this inclusion is very far
from being an equivalence of categories (unless $V=\{0\}$).
Nevertheless, just as in \eqref{e:Kpt}, it induces an isomorphism in
$K$-theory
\begin{equation}\label{e:Kvec}
  \R(H) = \K\Rep(H) \simeq \K^H(V);
\end{equation}
this is \cite{Tho}*{Theorem 4.1}.
(Roughly speaking, the image of
the map \eqref{e:Cohvec} consists of the {\em projective}
$(\Poly(V),H)$-modules; the isomorphism in $\K$-theory arises from the
existence of projective resolutions.)

Suppose now that $H\subset G$ is an inclusion of affine algebraic
groups. Then there is an equivalence of categories
\begin{equation}\label{eq:Cohhomog}
\Coh^G(G/H) \simeq \Coh^H(\point) \simeq \Rep(H);
\end{equation}
so the equivariant $\K$-theory is
\begin{equation}\label{eq:Khomog}
\K^G(G/H) \simeq \K^H(\point) \simeq \K\Rep(H) = \R(H).
\end{equation}
More generally, if $H$ acts on the variety $Y$, then we can form a
fiber product
\begin{equation*}
X = G\times_H Y
\end{equation*}
and calculate
\begin{equation}\label{e:Kfiber}\begin{aligned}
\Coh^G(G\times_H Y) &\simeq \Coh^H(Y),\\
\K^G(G \times_H Y) &\simeq \K^H(Y).
\end{aligned}\end{equation}

If $Y\subset X$ is a closed $H$-invariant subvariety, then there is a
right exact sequence
\begin{equation}\label{e:Kright}
\K^H(Y) \rightarrow \K^H(X) \rightarrow \K^H(X-Y) \rightarrow 0.
\end{equation}
This is established without the $H$ for example in
\cite{Har}*{Exercise II.6.10(c)}, and the argument there carries over to
the equivariant case.  This sequence is the end of a long
exact sequence in higher equivariant $\K$-theory (\cite{Tho}*{Theorem
  2.7}). The next term on the left is $\K_1^H(X-Y)$.
\end{subequations} 

Here are some of the consequences we want.
\begin{proposition} Suppose $H$ acts on $X$ with an open orbit
  $U\simeq H/H^u$, and $Y=X-U$ is the (closed) complement of this
  orbit. Then there is a right exact sequence
$$\K^H(Y) \rightarrow \K^H(X) \twoheadrightarrow \K^H(H/H^u) \simeq
\K^{H^u}(\point) \simeq \R(H^u).$$
Consequently the quotient
$$\K^H(X)/\IM \K^H(Y)$$
has a basis naturally indexed by the set $\widehat {H^u}$ of irreducible
representations of $H^u$.

More generally, if $U$ is the disjoint union of finitely many open
orbits $U_j \simeq H/H^{u^j}$, then $\K^H(X)/\IM \K^H(Y)$ has a basis
naturally indexed by the disjoint union of the sets $\widehat {H^{u_j}}$.
\end{proposition}

\begin{theorem}\label{thm:finorb} Suppose $H$ acts on the affine
  variety $X$ with finitely many orbits
  $$X_i = H\cdot x_i \simeq H/H_i.$$
  \begin{enumerate}
    \item The (Zariski) closure $\overline{X_i}$ is the union of $X_i$ and
      finitely many additional orbits $X_j$ with $\dim X_j < \dim X_i$.
\item For each irreducible $(\tau,V_\tau)$ of $H_i$,
  there is an $H$-equivariant coherent sheaf $\widetilde{{\mathcal
      V}_\tau}$ on $\overline{X_i}$ with the property that
  $$\widetilde{{\mathcal V}_\tau}|_{X_i} = {\mathcal V}_\tau
  =_{\text{\textnormal{def}}} G\times_{H_i} V_\tau.$$
The sheaf $\widetilde{{\mathcal V}_\tau}$ may be regarded as a
(coherent equivariant) sheaf on $X$.
\item The equivariant $\K$-theory $\K^H(X)$ is a free ${\mathbb
  Z}$-module with basis consisting of the various
  $[\widetilde{{\mathcal V}_\tau}]$, for $X_i \subset X$ and
  $\tau$ an irreducible representation of $H_i$. Each such basis
  vector is uniquely defined modulo the span of the
  $[\widetilde{{\mathcal V}_{\tau'}}]$ supported on orbits in the boundary
  of $X_i$.
\item If $Y\subset X$ is an $H$-invariant closed
  subvariety, then \eqref{e:Kright} is a short exact sequence
$$0 \rightarrow \K^H(Y) \rightarrow \K^H(X) \rightarrow \K^H(X-Y)
  \rightarrow 0.$$
  \end{enumerate}
\end{theorem}
\begin{proof}
\begin{subequations}\label{se:Kproof}
Part (1) is a standard statement about algebraic varieties. Part (2)
is established without the $H$ for example in \cite{Har}*{Exercise
  2.5.15}, and the argument there carries over to the equivariant
case.

For (3), we proceed by induction on the number of $H$-orbits on 
$X$. If the number is zero, then $X$ is empty, and $K^H(X) = 0$ is
indeed free. So suppose that the number of $H$-orbits is positive, and
that (3) is already known in the case of a smaller number of
orbits. Pick an orbit
\begin{equation}
  U=H\cdot x_{i_0} \simeq H/H_{i_0} \subset X
\end{equation}
of maximal dimension. Then necessarily $U$ is open (this follows from
(1)),  so $Y = X-U$ is closed. According to \eqref{e:Kright}, there is
an exact sequence
\begin{equation}\label{eq:exact}
  \K_1^H(U) \rightarrow \K^H(Y) \rightarrow \K^H(X)
  \rightarrow \K^H(U) \rightarrow 0.
\end{equation}
By inductive hypothesis, $\K^H(Y)$ is a free ${\mathbb Z}$-module with
basis the various $[\widetilde{{\mathcal V}_\tau'}]$ living on orbits other than
$U$. By \eqref{eq:Khomog}, $\K^H(U)$ is a free ${\mathbb Z}$-module
with basis the images of the various $[\widetilde{{\mathcal V}_\tau}]$
attached to $U$. To finish the proof of (3), we need only prove that
the map from $\K_1^H(U)$ is zero.

According to \eqref{eq:Cohhomog}, the first term of our exact sequence
is the Quillen $\K_1$ of $H/H_{i_0}$; that is, Quillen $\K_1$ of a
category of representations (of $H_{i_0}$). This category is (up to
long exact sequences) a direct sum of
categories of finite-dimensional complex vector spaces. By one of the
fundamental facts about algebraic $\K$-theory, it follows that
$\K_1^H(U) = \K_1^H(H/H_{i_0})$ is a direct sum of copies of ${\mathbb
  C}^\times$, one for each irreducible representation of $H_{i_0}$.

Any group homomorphism from the divisible group ${\mathbb C}^\times$
to ${\mathbb Z}$ must be zero; so the connecting homomorphism
$\K_1^H(U)\rightarrow \K^H(Y)$ is zero. This proves (4).
\end{subequations} 
\end{proof}

We thank Gon\c calo Tabuada for explaining to us this
proof of (4).

\section{Associated varieties for $({\mathfrak
    g},K)$-modules}\label{sec:assvar}
\setcounter{equation}{0}
\begin{subequations}\label{se:introgK}
With the structure of ${\mathcal N}^*_\theta$ from Section
\ref{sec:nilp}, and the generalities about equivariant $\K$-theory from
Section \ref{sec:Kthy}, we can now
introduce the $\K$-theory functor we will actually consider (in place
of the one from \eqref{se:introR} that we do not know how to define).
As a replacement for the moderate growth representations of
\eqref{eq:GRrep}, we will use
\begin{equation}\label{e:MfgK}
{\mathcal M}_f({\mathfrak g},K) = \text{category of finite length
  $({\mathfrak g},K)$-modules}.
\end{equation}
This category has a nice Grothendieck group $\K({\mathfrak g},K)$,
which is a free ${\mathbb Z}$-module with basis the equivalence
classes of irreducible modules. A connection with the incomplete ideas
in the introduction is provided by the Casselman-Wallach
theorem of \cite{WallachII}*{11.6.8}: passage to $K({\mathbb
  R})$-finite vectors is an equivalence of categories
\begin{equation}\label{eq:GRgK}
  {\mathcal F}_{\text{\textnormal{mod}}}(G({\mathbb R}))
  \buildrel{\sim}\over{\longrightarrow} {\mathcal M}_f({\mathfrak g},K).
\end{equation}

We write
\begin{equation}
[X]_{({\mathfrak g},K)} = \text{class of $X\in \K({\mathfrak g},K)$}.
\end{equation}
Any such $X$ admits a good filtration (far from unique), so that we
can construct
$$\gr X \in \Coh^K({\mathcal N}^*_\theta).$$
(The reason that the $S({\mathfrak g})$ module $\gr X$ is supported on
${\mathcal N}^*_\theta$ is explained in \cite{Vunip}*{Corollary 5.13}.)
The class $[\gr X] \in \K^K({\mathcal N}^*_\theta)$ is independent of
the choice of good filtration (this standard fact is proven in
\cite{Vunip}*{Proposition 2.2}) so we may write it simply as
$[X]_\theta$. In this way we get a well-defined homomorphism
\begin{equation}\label{eq:grgK}
\gr\colon \K({\mathfrak g},K) \rightarrow \K^K({\mathcal N}^*_\theta),
\qquad [X]_{({\mathfrak g},K)} \mapsto [\gr X] = [X]_\theta.
\end{equation}
This is our replacement for the map \eqref{eq:grR} that we do not know
how to define. Here are replacements for the undefined ideas in
\eqref{se:asympcycle}.
\end{subequations} 

\begin{definition}\label{def:assvar}
    \addtocounter{equation}{-1}
  \begin{subequations}\label{se:defassvar}
Fix a {\em
  nonzero} $({\mathfrak g},K)$-module
\begin{equation}
  X \in {\mathcal M}_f({\mathfrak g},K).
\end{equation}
Then $\gr X$ is a {\em nonzero} $K$-equivariant coherent sheaf on ${\mathcal
  N}^*_\theta$, and as such has a well-defined nonempty {\em support}
\begin{equation}\label{eq:supp}
  \supp(\gr X) \subset {\mathcal N}^*_\theta
\end{equation}
which is a Zariski-closed union of $K$-orbits. This
support is also called the {\em associated variety} of $X$,
\begin{equation}\label{eq:AV}\begin{aligned}
    \AV(X) &\eqdef \supp(\gr(X)) \\
    &\eqdef(\text{variety of the ideal\ }
  \Ann(\gr(X)) \subset {\mathcal N}^*_\theta.
\end{aligned}\end{equation}
(More details about the commutative
algebra definition of support are recalled for example in
\cite{Vunip}*{(1.2)}.) Theorem \ref{thm:finorb} provides a formula
\begin{equation}
  \gr(X) = \sum_{Z = K\cdot E_Z\subset \AV(X)}\sum_{\tau_Z^j\in
    \widehat{K^{Z}}} m_{\tau_Z^j}(X)[\widetilde{{\mathcal
          V}_{\tau_Z^j}}]
\end{equation}
If we write
\begin{equation}
  \{Y_1, Y_2,\ldots,Y_r\}, \qquad Y_i \simeq K/K_i
\end{equation}
for the open $K$ orbits in $\AV(X)$,
then \cite{Vunip} shows that each virtual representation
\begin{equation}
  \mu_{Y_i}(X) = \sum_j m_{\tau_{Y_i}^j(X)} \tau_{Y_i}^j \in \R(K_i)
\end{equation}
is independent of the choices defining $\widetilde{{\mathcal
    V}_{\tau_Z^j}}$. We define the {\em associated cycle of $X$} to be
\begin{equation}\label{eq:AC}
  {\mathcal A}{\mathcal C}(X) = \sum_i \mu_{Y_i}(X) Y_i.
\end{equation}
The {\em weak associated cycle of $X$} is
\begin{equation}\label{eq:WAC}
  {\mathcal A}{\mathcal C}_{\weak}(X) = \sum_i \dim \mu_{Y_i}(X) Y_i.
\end{equation}
  \end{subequations}
  \end{definition}

The following deep result is part of the connection between
what we prove here (about algebraically defined associated varieties)
and the incomplete analytic picture outlined in the introduction (involving the
analytically defined wavefront set). This theorem is not used in our
results about associated varieties.

\begin{theorem}{Schmid-Vilonen \cite{SV}*{Theorem 1.4}}\label{thm:SV}
  Suppose that $(\pi,V)$ is a nonzero ${\mathfrak Z}({\mathfrak
    g})$-finite representation of $G({\mathbb R})$ of moderate growth
  (see \eqref{eq:GRrep}). Write $X=V_{K({\mathbb R})}$ for the
  underlying Harish-Chandra module (see \eqref{eq:GRgK}). Then
  $$\WF(\pi) \longleftrightarrow \supp(\gr X)$$
  by means of the Kostant-Sekigichi identification Corollary
  \ref{cor:KS} (of $G({\mathbb R})$ orbits on ${\mathcal
    N}^*_{\mathbb R}$ with $K$ orbits on ${\mathcal N}^*_\theta$).
\end{theorem}
At this point it should be possible for the reader to rewrite the
introduction, replacing the equivariant $\K$-theory for
$G({\mathbb R})$ acting on ${\mathcal N}^*_{\mathbb R}$ (which we do
not know how to define) by Thomason's algebraic equivariant $\K$-theory for $K$
acting on ${\mathcal N}^*_\theta$; and replacing wavefront sets and
cycles by associated varieties and associated cycles. We will not do
that explicitly.

\begin{subequations}\label{se:resK}
Restriction to $K$ is a fundamental tool for us, and we pause here to
introduce a bit of useful formalism about that. Define
\begin{equation}\label{eq:MfK}
  {\mathcal M}_f(K) = \text{category of admissible algebraic
    $K$-modules;}
\end{equation}
as usual {\em admissible} means that each irreducible representation
of $K$ appears with finite multiplicity. This abelian category has a
nice Grothendieck group
\begin{equation}\label{eq:KK}
  \K(K) = \prod_{(\rho,E_\rho)\in \widehat K} {\mathbb Z}\rho,
\end{equation}
the direct {\em product} of one copy of ${\mathbb Z}$ for each
irreducible representation of $K$. A $({\mathfrak g},K)$-module $X$ of
finite length is necessarily admissible:
\begin{equation}
  X|_K = \sum_{(\rho,E_\rho)\in \widehat K} m(\rho,X)E_\rho \qquad
  (m(\rho,X)\in {\mathbb N}).
\end{equation}
Restriction to $K$ is therefore an exact functor
\begin{equation}
  \res_K\colon {\mathcal M}_f({\mathfrak g},K) \rightarrow {\mathcal
    M}_f(K).
\end{equation}
The corresponding homomorphism of Grothendieck groups is
\begin{equation}\label{eq:resgK}
  \res_K\colon \K({\mathfrak g},K) \rightarrow \K(K), \qquad
      [X]_{({\mathfrak g},K)} \mapsto \prod_{(\rho,E_\rho)\in \widehat
        K} m(\rho,X)\rho.
\end{equation}
In exactly the same way, the fact that $K$ is reductive implies that
any irreducible representation of $K$ must appear with finite
multiplicity in global sections of an equivariant coherent sheaf on a
homogeneous space for $K$. If $K$ acts on an affine variety $Z$ with
finitely many orbits, restriction to $K$ is therefore an exact functor
\begin{equation}
  \res_K\colon \Coh^K(Z) \rightarrow {\mathcal
    M}_f(K).
\end{equation}
The corresponding homomorphism of Grothendieck groups is
\begin{equation}\label{eq:resKK}
  \res_K\colon \K^K(Z) \rightarrow \K(K).
\end{equation}
The maps \eqref{eq:resgK} and \eqref{eq:resKK} fit into a commutative
diagram
\begin{equation}\label{eq:resCD}
\begin{tikzcd}
  \K({\mathfrak g},K) \arrow[rr,"{[\gr]}"] \arrow[dr,"\res_K" description] &&
  \K^K({\mathcal N}^*_\theta) \arrow[dl,"\res_K" description] \\
  & \K(K) &\\
\end{tikzcd}
\end{equation}

  \end{subequations} 

\section{The case of complex reductive groups}\label{sec:cplx}
\setcounter{equation}{0}
\begin{subequations}\label{se:cplx}
The general case of our results requires a discussion of the (rather
complicated) Langlands classification of  representations of real
reductive groups. In order to explain our new ideas, we will therefore
first consider them in the (less complicated) setting of {\em complex}
reductive groups. Suppose therefore (still using the notation of
\eqref{se:realgrps}) that there is a complex connected
reductive algebraic group $G_1$, and that
\begin{equation}
  G({\mathbb R}) \simeq G_1.
\end{equation}
Fix a compact real form
\begin{equation}\label{eq:cptform}
\sigma_1\colon G_1\rightarrow G_1, \qquad G_1({\mathbb R},\sigma_1)
=_{\text{\textnormal{def}}} G_1^{\sigma_1} = K_1\subset G_1.
\end{equation}
Then we can arrange
\begin{equation}\label{eq:cplx}\begin{aligned}
    G &= G_1 \times G_1 \\
    \sigma_0(g,g')&= (\sigma_1(g),\sigma_1(g')), \qquad G^{\sigma_0} =
    K_1\times K_1\\
    \sigma_{\mathbb R}(g,g') &= (\sigma_1(g'),\sigma_1(g)),\\
    G({\mathbb R}) &= \{(g,\sigma_1(g)) \mid g\in G_1\} \simeq G_1\\
    \theta(g,g') &= (g',g)\\
    K &= (G_1)_\Delta = \{(g,g)\mid g\in G_1\}\\
    K({\mathbb R}) &= (K_1)_\Delta = \{(k,k) \mid k\in K_1\} \simeq K_1.\\
\end{aligned}\end{equation}
(The notation $K_1$ may be a bit confusing because in general we write $K$ for
the {\em complex} group which is the complexification of the maximal
compact $K({\mathbb R})\subset G({\mathbb R})$; but here
$K_1$ is a compact (real) group. We have not found a reasonable
change of notation to address this issue.)

We fix also a (compact) maximal torus
\begin{equation}\label{eq:T1}
  T_1 \subset K_1;
\end{equation}
then automatically its complexification
\begin{equation}
  H_1 = G_1^{T_1}
\end{equation}
is a (complex) maximal torus in $G_1$. We write
\begin{equation}\label{eq:CW}
  W_1 = N_{K_1}(T_1)/T_1 \simeq N_{G_1}(H_1)/H_1
\end{equation}
for the Weyl group. Complexifications of these things are
\begin{equation}\begin{aligned}
    H &= H_1\times H_1 = \text{complexification of $H_1$}\\
    W &= W_1 \times W_1 = \text{Weyl group of $H$ in $G$.}\\
  \end{aligned}
\end{equation}
We fix also a Borel subgroup $B_1\supset H_1$ of $G_1$. Because
$\sigma_1$ preserves $T_1$, and $K_1$ is compact, the Borel subgroup
$\sigma_1(B_1)$ is necessarily equal to $B_1^{\opp}$. Therefore the
complexification of  $B_1$ is
\begin{equation}\label{eq:Bqs}
  B_{qs} = B_1 \times \sigma_1(B_1) = B_1 \times B_1^{\opp},
\end{equation}
corresponding to the real Borel subgroup $B_1$ for the quasisplit
$G({\mathbb R}) = G_1$. (The subscript $qs$ stands for ``quasisplit.'')
We are also interested in the $\theta$-stable
Borel subgroup
\begin{equation}\label{eq:Bf}
  B_f = B_1\times B_1;
\end{equation}
now the subscript $f$ stands for ``fundamental.''
\end{subequations} 

We turn next to a discussion of nilpotent orbits in the complex case.
\begin{subequations}\label{se:cplxnilp}
We write ${\mathcal N}^*_1$ for the nilpotent cone in ${\mathfrak
  g}_1^*$, and use other notation accordingly. Then
\begin{equation}\label{eq:cplxnilp} \begin{aligned}
    {\mathcal N}^* &= {\mathcal N}_1^* \times {\mathcal N}_1^* \\
    {\mathcal N}_{i{\mathbb R}}^* &= \{(E,-\sigma_1(E)) \mid E \in
    {\mathcal N}_1^*\} \simeq {\mathcal N}_1^* \\
    {\mathcal N}_\theta^* &= \{(E,-E)) \mid E \in
    {\mathcal N}_1^*\} \simeq {\mathcal N}_1^* \\
  \end{aligned}
\end{equation}
Immediately we get identifications of orbits
\begin{equation}\label{eq:cplxnilporb} \begin{aligned}
    {\mathcal N}^*/G &= {\mathcal N}_1^*/G_1 \times {\mathcal N}_1^*/G_1 \\
    {\mathcal N}_{i{\mathbb R}}^*/G({\mathbb R}) &\simeq {\mathcal N}_1^*/G_1 \\
    {\mathcal N}_\theta^*/K &\simeq {\mathcal N}_1^*/G_1 \\
  \end{aligned}
\end{equation}
and therefore ${\mathcal N}_{i{\mathbb R}}^*/G({\mathbb R}) \simeq
{\mathcal N}_\theta^*/K$. This last is the Kostant-Sekiguchi bijection
of Corollary \ref{cor:KS}.

Clearly the (antiholomorphic) automorphism $\sigma_1$ of $G_1$ acts on
the set ${\mathcal N}^*_1/G_1$ of nilpotent orbits. The
Jacobson-Morozov theorem implies that this action is trivial. Here is
why. The semisimple element $D$, whose class characterizes the orbit,
belongs to $[{\mathfrak g}_1,{\mathfrak g}_1]$ and has real
eigenvalues in the adjoint representation; so after conjugation we can
arrange
$$D \in i{\mathfrak t}_1,\qquad \sigma_1(D) = -D.$$
The Jacobson-Morozov $SL(2)$ shows that $D$ is conjugate to $-D$, so
we have shown that $\sigma_1$ preserves the conjugacy class of
$D$. Now apply Corollary \ref{cor:JM}.

It follows that $\sigma_{\mathbb R}$ acts on ${\mathcal N}_1^*/G_1
\times {\mathcal N}_1^*/G_1$ by interchanging the two factors. The
nilpotent $G$-orbits preserved by this action are exactly the diagonal
classes. What \eqref{eq:cplxnilp} shows is that each such diagonal
nilpotent class for $G$ has a unique real form (and therefore, by
Kostant-Sekiguchi, a unique $\theta$-form).
\end{subequations} 

In light of this identification of the $K=G_1$ action on ${\mathcal
  N}^*_\theta$ with the $G_1$ action on ${\mathcal N}^*_1$, we can
restate Definition \ref{def:geom} (in this complex case) as follows.

\begin{definition}\label{def:cplxgeom}
A {\em (global) geometric parameter} for a complex reductive algebraic
group $G_1$ is a nilpotent $G_1$-orbit
$Y \subset {\mathcal N}^*_1$, together with an
irreducible $G_1$-equivariant vector bundle
$${\mathcal E} \rightarrow Y.$$
Equivalently, a {\em (local) geometric parameter} is a $G_1$-conjugacy
class of pairs
$$(\xi,(\tau,E)),$$
with $\xi\in {\mathcal N}^*_1$ a nilpotent element,
and $(\tau,E)$ an irreducible (algebraic) representation of the isotropy group
$G_1^\xi$. This bijection between local and global parameters identifies
$(\xi,(\tau,E))$ with the pair
$$Y = G_1\cdot \xi \simeq G_1/G_1^\xi, \qquad {\mathcal E} \simeq
G_1\times_{G_1^\xi} E.$$
We write ${\mathcal P}_{g}(G_1)$ for the collection of
geometric parameters. Sometimes it will be convenient to write
${\mathcal E}(\tau)$ or ${\mathcal E}(\xi,\tau)$ to exhibit the
underlying local parameter.
\end{definition}

These geometric parameters are exactly what appears on the complicated
side of the Lusztig-Bezrukavnikov bijection \cite{Bezr} for
$G_1$. Lusztig's conjecture, and its proof by Bezrukavnikov, were
critical to the development of the ideas in this paper. But they are
not logically necessary to explain our results, so for brevity we are
going to omit them.

The introduction (after reformulation in terms of $K$ acting on
${\mathcal N}^*_\theta$)  outlined a connection between
the geometric parameters of Definitions \ref{def:geom} and
\ref{def:cplxgeom} and the associated varieties we seek to compute. We
conclude this section with an account of the Langlands classification,
which describes representations of $G_1$ in terms that we
will be able to relate to geometric parameters.

\begin{definition}\label{def:langcplx} In the setting of
  \eqref{se:cplx}, a {\em Langlands parameter} for
  $G_1$ (or for $(G,K) = (G_1\times G_1,(G_1)_\Delta)$) is a pair of
  linear functionals
$$(\lambda_L,\lambda_R) \in {\mathfrak h}^*_1\times {\mathfrak h}^*_1 \simeq
{\mathfrak h}^*,$$
subject to the requirement that the restriction of
$(\lambda_L,\lambda_R)$ to ${\mathfrak t}_1=({\mathfrak h}_1)_\Delta$ is
the differential of a  weight:
$$\lambda_L + \lambda_R = \gamma \in X^*(H_1).$$
Two Langlands parameters are said to be {\em equivalent} if they are
conjugate by $(W_1)_\Delta$:
$$(\lambda_L,\lambda_R) \sim (\lambda_L',\lambda_R') \iff
(\lambda_L,\lambda_R) = (w_1\cdot\lambda_L',w_1\cdot\lambda_R')$$
for some $w_1\in W_1$.
We write ${\mathcal P}_{\LL}(G,K)$ for the set of
equivalence classes of Langlands parameters. The {\em discrete part}
of the Langlands parameter is by definition
$$\gamma = \gamma(\lambda_L,\lambda_R) = \lambda_L+\lambda_R \in X^*(H_1).$$
Tbe {\em continuous part} of the parameter is
$$\nu = \nu(\lambda_L,\lambda_R) = \lambda_L-\lambda_R \in {\mathfrak h}_1^*.$$
We can recover the parameter from these two parts:
$$\lambda_L = (\gamma + \nu)/2, \qquad \lambda_R = (\gamma-\nu)/2.$$
Equivalence is easily written in terms of the discrete and continuous
parameters:
$$(\lambda_L,\lambda_R) \sim (\lambda_L',\lambda_R') \iff
(\gamma,\nu) = (w_1\cdot\gamma',w_1\cdot\nu') \quad (w_1\in W_1)$$
(with obvious notation).

The Langlands classification (due in this case to Zhelobenko) attaches
to each equivalence class of parameters a {\em standard representation}
$I(\lambda_L,\lambda_R)$ (more precisely, a
Harish-Chandra module for
$$({\mathfrak g},K) = ({\mathfrak g_1}\times{\mathfrak g}_1,(G_1)_\Delta))$$
with the following properties.
\begin{enumerate}
\item There is a unique irreducible quotient $J(\lambda_L,\lambda_R)$
  of $I(\lambda_L,\lambda_R)$.
\item Any irreducible Harish-Chandra module for $({\mathfrak g},K)$ is
  equivalent to some $J(\lambda_L,\lambda_R)$.
\item Two standard representations are isomorphic (equivalently, their
  Langlands quotients are isomorphic) if and only if their parameters
  are conjugate by $W_1$.
\item The infinitesimal character of the standard representation
  $I(\lambda_L,\lambda_R)$ is indexed by the $W=W_1\times W_1$ orbit
  of $(\lambda_L,\lambda_R)$.
\item The restriction of $I(\lambda_L,\lambda_R)$ to $K_1$ is the
  induced representation from $T_1$ to $K_1$ of $\gamma =\lambda_L +
  \lambda_R$:
  $$I(\lambda_L,\lambda_R) \simeq \Ind_{T_1}^{K_1}(\lambda_L + \lambda_R).$$
\item The restrictions to $K_1$ of two standard representations are
  isomorphic if and only if the weights $\lambda_L+\lambda_R$ and
  $\lambda_L' + \lambda_R'$ are conjugate by $W_1$.
\end{enumerate}
\end{definition}
The representation $I(\lambda_L,\lambda_R)$ is {\em tempered} (an analytic
condition due to Harish-Chandra, and central to Langlands' original
work) if and only if the continuous parameter is purely imaginary:
$$\nu = \lambda_L - \lambda_R \in iX^*(H_1)\otimes_{\mathbb Z}{\mathbb R}.$$
In this case $I(\lambda_L,\lambda_R) = J(\lambda_L,\lambda_R)$.

In the general (possibly nontempered) case, the real part of
$\nu$ controls the growth of matrix coefficients of
$J(\lambda_L,\lambda_R)$. Tempered representations have the smallest
possible growth, and larger values of $\RE\nu$ correspond to larger
growth rates.

A $K$-Langlands parameter is the same thing, but with a larger
equivalence relation. (In what follows, remember that $K$ is the
complexification of $K_1$: locally finite continuous representations
of the compact group $K_1$ are the same as algebraic representations
of the complex algebraic group $K$.)

\begin{definition}\label{def:Klangcplx}
A {\em $K$-Langlands parameter} for $(G,K)$
is any of the following equivalent things.
\begin{enumerate}
\item A Langlands parameter $(\lambda_L,\lambda_R)$. The equvalence
  relation is
$$(\lambda_L,\lambda_R) \sim_K (\lambda_L',\lambda_R') \iff
  I(\lambda_L',\lambda_R')|_K =  I(\lambda_L,\lambda_R)|_K.$$
The equivalence relation is throwing away the continuous parameter
$\nu(\lambda_L,\lambda_R)$. Another way to state this is
$$(\lambda_L,\lambda_R) \sim_K (\lambda_L',\lambda_R') \iff \lambda_L
+ \lambda_R \in W_1\cdot(\lambda_L'+ \lambda_R').$$
\item A tempered parameter $(\gamma/2,\gamma/2)$ (some
$\gamma\in X^*(H_1)$) having real infinitesimal character (see
  \cite{Vgreen}*{Definition 5.4.11}). The equivalence relation is
$$(\gamma/2,\gamma/2) \sim_K (\gamma'/2,\gamma'/2) \iff \gamma \in W_1\cdot\gamma'.$$
\item A $W_1$ orbit of weights $\gamma\in X^*(H_1)$.
\item A dominant weight $\gamma_0 \in X^*(H_1)^+$.
\end{enumerate}
The equivalence of the four conditions is standard and easy. The tempered
parameter of real infinitesimal character
$$((\lambda_L+\lambda_R)/2,(\lambda_L + \lambda_R)/2)$$
is a natural representative for the $K$-equivalence class.

If $\gamma\in X^*(T)$ is any weight (and
$\gamma_0 = w\gamma$ is its unique dominant conjugate), then the ``fixed
restriction to $K$'' in (1) is
$$ \Ind_{T_1}^{K_1}(\gamma) \simeq \Ind_{T_1}^{K_1}(\gamma_0).$$
Write ${\mathcal P}_{K\dashLL}(G,K)$ for
equivalence classes of $K$-Langlands parameters.
\end{definition}

The conjecture of Lusztig proved by Bezrukavnikov in \cite{Bezr} is a
{\em bijection} between the geometric parameters of Definition
\ref{def:cplxgeom} and the $K$-Langlands parameters of Definition
\ref{def:Klangcplx}. Bezrukavnikov proceeds by using these two sets to
index two bases of the same ${\mathbb Z}$-module $\K^K({\mathcal
  N}^*_\theta)$. He proves that his change of basis matrix is upper
triangular, and in this way establishes the bijection between the
index sets.  He does not offer a method to {\em calculate} his basis
indexed by geometric parameters.

Following Achar, we will use geometric parameters to index a different
basis of the same vector space, and we will calculate the change of
basis matrix. We are not able to prove that our basis is the same as
Bezrukavnikov's.

Here are some classical properties of $K$-Langlands parameters that we
will use in Section \ref{sec:cplxalg} to construct the basis indexed
by geometric parameters. The results are due to Zhelobenko; but his
results are spread over a number of papers and difficult to
reference. A convenient reference is \cite{DufloC}.

\begin{theorem}\label{thm:cplxresK} Suppose $(\lambda_L,\lambda_R)$
  represents a $K$-Langlands parameter for the complex group $G$
  (Definition \ref{def:Klangcplx}). Write
  $$\gamma=\lambda_L + \lambda_R \in
  X^*(H_1) \simeq X^*(T_1)$$
  for the corresponding weight of the compact torus \eqref{eq:T1}.
  \begin{enumerate}
    \item The restriction to $K_1$ of $I(\lambda_L,\lambda_R)$ contains
      the irreducible representation
      $$\mu(\lambda_L,\lambda_R) = \text{irreducible of extremal
        weight $\gamma$}$$
      with multiplicity one. The other irreducible representations of
      $K_1$ appearing have strictly larger extremal weights.
    \item The map
  $$\mu\colon {\mathcal P}_{K\dashLL}(G,K) \rightarrow \widehat{K_1},\qquad
      (\lambda_L,\lambda_R)\mapsto \mu(\lambda_L,\lambda_R)$$
      is a bijection.
     \item The classes
      $$[\res_K I(\lambda_L,\lambda_R)] \in \K(K) \qquad
      (\lambda_L,\lambda_R) \in {\mathcal P}_{K\dashLL}(G,K)$$
        (see \eqref{eq:resgK}) are linearly independent.
  \end{enumerate}
\end{theorem}

\section{Representation basis for $\K$-theory: ${\mathbb C}$
  case}\label{sec:cplxalg}
\setcounter{equation}{0}

We continue in the
setting \eqref{se:cplx}, that is, assuming that $(G,K)$ arises from a
{\em complex} reductive group.

Applying Theorem \ref{thm:finorb} to $K$ acting on the
$K$-nilpotent cone gives

\begin{corollary}\label{cor:Kspangeom}
Write $\{Y_1,\ldots,Y_r\}$ for the orbits of $K$
on ${\mathcal N}^*_\theta$ (\eqref{e:Knilcone}). For each irreducible
$K$-equivariant vector bundle ${\mathcal E}$ on some $Y_i$, fix a
$K$-equivariant (virtual) coherent sheaf $\widetilde{\mathcal E}$
supported on the closure of $Y_i$, and restricting to ${\mathcal E}$
on $Y_i$.
Then the classes $[\widetilde{\mathcal E}]$ are a ${\mathbb Z}$-basis of
$\K^K({\mathcal N}^*_\theta)$. A little more precisely, the classes
supported on any $K$-invariant closed subset $Z \subset {\mathcal
  N}_\theta$ are a basis of $\K^K(Z)$.
\end{corollary}

In order to describe the algorithm underlying ``Proposition''
\ref{prop:suppR} (precisely, Theorem \ref{thm:cplxassvar} below)
we will need an entirely different kind of basis of $\K^K({\mathcal
  N}^*_\theta)$.

\begin{definition}\label{def:cplxbasis}
  \addtocounter{equation}{-1}
  \begin{subequations}\label{se:defcplxbasis}
Suppose $\gamma \in {\mathcal P}_{K\dashLL}(G,K)$ is a
  $K$-Langlands parameter for $(G,K) = (G_1\times G_1,(G_1)_\Delta)$
  (Definition \ref{def:Klangcplx}); equivalently, a dominant weight
for $G_1$).  We
attach to $\gamma$ a $K$-equivariant coherent sheaf on the
$K$-nilpotent cone ${\mathcal N}^*_\theta$, with well-defined image
$[\gamma]_\theta \in \K^K({\mathcal N}^*_\theta)$, characterized in any of the
following equivalent ways. Note first that the cotangent bundle of
$G_1/B_1$ is
\begin{equation}
T^*(G_1/B_1) = G_1 \times_{B_1} ({\mathfrak g}_1/{\mathfrak b}_1)^*
  \buildrel{\pi_1}\over{\longrightarrow} G_1/B_1.
\end{equation}
From the total space of the tangent bundle there is the {\em moment
  map}
\begin{equation}\begin{aligned}
T^*(G_1/B_1) \buildrel{\mu_1}\over{\longrightarrow} {\mathfrak g}_1^*,
&\qquad \mu_1(g,\xi) = \Ad^*(g)(\xi) \\
&(g\in G_1,\ \xi \in ({\mathfrak g_1}/{\mathfrak b}_1)^*).
\end{aligned}\end{equation}
Of course $\pi_1$ is affine, $\mu_1$ is proper, and both maps are
$G_1$-equivariant; the image of $\mu_1$ is the nilpotent cone
${\mathcal N}^*_1$.
\begin{enumerate}
\item Fix any Langlands parameter $(\lambda_L,\lambda_R)$ restricting
  to $\gamma$; that is, elements $\lambda_L$ and $\lambda_R$ of
  ${\mathfrak h}_1^*$ such that $\gamma = \lambda_L + \lambda_R$. Fix a
  good filtration on the standard module $I(\lambda_L,\lambda_R)$, and
  define
  \begin{equation}
    [\gamma]_\theta = [\gr I(\lambda_L,\lambda_R)] =
    [I(\lambda_L,\lambda_R)]_\theta
  \end{equation}
  (see \eqref{eq:grgK}).
\item Extend $\gamma$ to a one-dimensional (algebraic) character of
  $B_1=H_1N_1$, and let ${\mathcal L_1}$ be the corresponding
  $G_1$-equivariant (algebraic) line bundle on $G_1/B_1$. The pullback
  $\pi_1^*{\mathcal L_1}$ is a $G_1$-equivariant coherent sheaf on
  $T^*(G_1/B_1)$, so each higher direct image
  $R^k\mu_{1*}(\pi_1^*{\mathcal L}_1)$ is (since $\mu_1$ is proper) a
  $G_1$-equivariant coherent sheaf on ${\mathcal N}_1^*$. We define
  \begin{equation}
    [\gamma]_{1,\theta} = \sum_k (-1)^k [R^k\mu_{1*}(\pi_1^*{\mathcal L}_1)] \in
    \K^{G_1}({\mathcal N}_1^*).
  \end{equation}
  Under \eqref{eq:cplxnilp}, $[\gamma]_{1,\theta}$
  corresponds to a class $[\gamma]_\theta \in \K^K({\mathcal N}^*_\theta)$.
\item As an algebraic representation of $K$,
  \begin{equation}
    [\gamma]_\theta = \Ind_{T}^K \gamma = \sum_{(\rho,E_\rho)\in \widehat K} \dim
    E_\rho(\gamma)\rho;
  \end{equation}

here $E_\rho(\gamma)$ denotes the $\gamma$ weight space (with respect to the
maximal torus $T$) of the $K$-representation $E_\rho$.
\end{enumerate}
The equivalence of (1) and (2) is a consequence of Zuckerman's
cohomological induction construction of $I(\lambda_L,\lambda_R)$ using
the $\theta$-stable Borel subgroup $B_f$ of \eqref{eq:Bf}. That they
have the property in (3) is a standard fact about principal series
representations of complex groups. That property (3) characterizes
$[\gamma]_\theta$ is a consequence of Corollary \ref{cor:KKbasis} below.
  \end{subequations} 
\end{definition}

\begin{theorem}\label{thm:cplxgeomexpand}
Suppose $(Y,{\mathcal E}) \in {\mathcal P}_{g}(G,K)$ is a
geometric parameter (Definition \ref{def:cplxgeom}). Then there is an
extension $\widetilde{{\mathcal E}}$ as in
Corollary \ref{cor:Kspangeom}, and a formula in $\K^K({\mathcal
  N}_\theta)$
$$[\widetilde{\mathcal E}] = \sum_{\gamma\in {\mathcal
    P}_{K\dashLL}(G,K)} m_{\widetilde{\mathcal E}}(\gamma)[\gamma]_\theta.$$
Here the sum is finite, and $m_{\widetilde{\mathcal
    E}}(\gamma) \in {\mathbb Z}$.

Suppose $\widetilde{\mathcal E}'$ is another extension of ${\mathcal
  E}$ to $\overline Y$. Then
$$[\widetilde{\mathcal E}'] - [\widetilde{\mathcal E}] =
\sum_{\substack{(Z,{\mathcal F}) \in {\mathcal P}_{g}(G,K)\\[.2ex]
    Z \subset \partial Y}} n_{\mathcal F} [{\widetilde {\mathcal F}}].$$
Only finitely many terms appear in the sum, and $n_{\mathcal
  F}\in {\mathbb Z}$.
\end{theorem}

We will give a proof in Corollary \ref{cor:JMres}. 

\begin{corollary}\label{cor:KKbasis}
  In the setting of \eqref{se:cplx}, the classes
  $$\{[\gamma]_\theta \in K^K({\mathcal N}^*_\theta)\mid \gamma \in
  {\mathcal P}_{K\dashLL}(G,K)\}$$
  are a ${\mathbb Z}$-basis of $\K^K({\mathcal N}^*_\theta)$. The
  restriction to $K$ map
  $$\res_K\colon \K^K({\mathcal N}^*_\theta) \rightarrow \K(K)$$
 of \eqref{eq:resKK} is injective.
\end{corollary}
\begin{proof}
That these classes span is a consequence of Corollary
\ref{cor:Kspangeom} and Theorem \ref{thm:cplxgeomexpand}. That
they are linearly independent is a consequence of Theorem
\ref{thm:cplxresK}(3); the argument proves injectivity of the
restriction at the same time.
\end{proof}

\section{Geometric basis for $\K$-theory: ${\mathbb C}$ case}
\label{sec:cplxalgweights}
\setcounter{equation}{0}

In this section we consider how to relate the representation-theoretic
basis of Corollary \ref{cor:KKbasis} to the geometry of $K$-orbits on
${\mathcal N}^*_\theta$.
We will proceed in the aesthetically undesirable way of
using the Jacobson-Morozov theorem (and so discussing not the
nilpotent elements in ${\mathfrak g}^*$ that we care about, but rather
the nilpotent elements in ${\mathfrak g}$).
\begin{subequations}\label{se:JM}
We begin therefore with an arbitrary nilpotent element $X_1\in
{\mathcal N}_1$ (see \eqref{e:cplxnilcone}). The Jacobson-Morozov theorem
finds elements $Y_1$ and $D_1$ in ${\mathfrak g}_1$ so that
\begin{equation}\label{e:tds}
[D_1,X_1] = 2X_1,\quad [D_1,Y_1] = -2Y_1,\quad [X_1,Y_1] = D_1.
\end{equation}
We use the eigenspaces of $\ad(D_1)$ to define a ${\mathbb Z}$-grading
of ${\mathfrak g}_1$
\begin{equation}\begin{aligned}
{\mathfrak g}_1[k] &\eqdef \{Z\in {\mathfrak g}_1 \mid [D_1,Z] = kZ\},
\quad {\mathfrak g}_1[{\scriptstyle\ge}\, k] \eqdef \sum_{j\ge k}
      {\mathfrak g}_1[j],\\
      {\mathfrak q}_1 &\eqdef {\mathfrak g}_1[{\scriptstyle\ge}\,
        0],\qquad\qquad \qquad{\mathfrak  u}_1 \eqdef {\mathfrak
        g}_1[{\scriptstyle\ge}\, 1].
\end{aligned}\end{equation}
Then ${\mathfrak q}_1$ is the Lie algebra of a parabolic subgroup
$Q_1=L_1U_1$ of $G_1$, with Levi factor $L_1=G_1^{D_1}$. After
conjugating $X$ by $G_1$, we may assume that
\begin{equation}
 D_1\in {\mathfrak t}_1, \qquad T_1\subset L_1.
\end{equation}
We will be concerned with the equivariant vector bundle
\begin{equation}
 {\mathcal R}_1 =_{\text{\textnormal{def}}} G_1\times_{Q_1} {\mathfrak g}_1[{\scriptstyle\ge}\,2]
  \buildrel{\pi}\over{\longrightarrow} G_1/Q_1.
\end{equation}
(The reason ${\mathcal R}_1$ is of interest is that Corollary
\ref{cor:JMres} below says that it is a $G_1$-equivariant resolution
of singularities of the nilpotent orbit closure $\overline{G_1\cdot
  X_1}$. The ${\mathcal R}$ is meant to stand for {\em resolution}.)
According to \eqref{e:Kfiber} and \eqref{e:Kvec},
\begin{equation}
  \K^{G_1}({\mathcal R}_1) \simeq
  \R(Q_1) \simeq \R(L_1);
\end{equation}
if $(\sigma,S)$ is an irreducible representation of $L_1 \simeq
Q_1/U_1$, we write
\begin{equation}
{\mathcal S}_1(\sigma) = G_1\times_{Q_1} S
\end{equation}
for the induced vector bundle on $G_1/Q_1$. The corresponding basis
element of the equivariant
$\K$-theory is represented by the equivariant vector bundle
\begin{equation}\label{e:KR1rep}
  {\mathcal S}(\sigma) = \pi^*({\mathcal S}_1(\sigma)) =
  G_1\times_{Q_1} \left({\mathfrak
    g}_1[{\scriptstyle \ge}\,2] \times S \right) \rightarrow {\mathcal
    R}_1.
\end{equation}

\end{subequations}

We are in the setting of Theorem \ref{thm:JM}. As
a consequence of that Theorem, we have

\begin{corollary}\label{cor:JMbdle}
  Suppose we are in the setting \eqref{se:JM}.
\begin{enumerate}
\item The natural map
$$\mu\colon {\mathcal R}_1 \rightarrow
{\mathcal N}_1,\quad (g,Z)\mapsto \Ad(g)Z$$
is a proper birational map onto $\overline{G_1\cdot X_1}$. We may
therefore {\em identify} $G_1\cdot X_1$ with its preimage $U$:
$$ G_1/G_1^{X_1} \simeq G_1\cdot X_1 \simeq U \subset {\mathcal R}_1.$$
Because $G_1\cdot X_1$ is open in $\overline{G_1\cdot X_1}$, $U = \mu^{-1}(G_1\cdot
X_1)$ is open in ${\mathcal R}_1$.
\item The classes
  $$\{[{\mathcal S}(\sigma)] \mid \sigma \in \widehat{L_1} =
  \widehat{Q_1}\}$$
  of \eqref{e:KR1rep} are a basis of the equivariant $\K$-theory
  $\K^{G_1}({\mathcal R}_1)$.
\item Since $\mu$ is proper, higher direct images of coherent sheaves
  are always coherent. Therefore
$$[\mu_*({\mathcal S})] =_{\text{\textnormal{def}}} \sum_i (-1)^i
[R^i\mu_*{\mathcal S}] \in \K^{G_1}({\mathcal N}_1)$$
is a well-defined virtual coherent sheaf. This defines a map in
equivariant $\K$-theory
$$\mu_*\colon \K^{G_1}({\mathcal R}_1) \rightarrow \K^{G_1}({\mathcal
  N}_1). $$
Restriction to the open set $U\simeq G_1\cdot X_1$ commutes with $\mu_*$.
\item Suppose $\sigma$ is a representation of $L_1$, inducing vector
  bundles ${\mathcal S}_1(\sigma)$ on $G_1/Q_1$ and ${\mathcal S}(\sigma)$ on
  ${\mathcal R}_1$ as in \eqref{e:tds}. As a representation of $G_1$,
  (that is, in the Grothendieck group $\K(G_1)$; see \eqref{eq:MfK})
{\small \begin{align*}
[\mu_*(&{\mathcal S}(\sigma))] = \Ind_{L_1}^{G_1}\bigg(\sigma\otimes\Big( \sum_j
  (-1)^j \textstyle{\bigwedge^j} {\mathfrak g}_1[1]^*\Big)\bigg)\\[1.5ex]
&= \Ind_{T_1}^{G_1} \bigg(\sum_\phi m_\sigma(\phi) \hskip -3ex
\sum_{\substack{A\subset \Delta({\mathfrak g}_1[1],{\mathfrak t}_1)
    \\ w\in W(L_1)}} \hskip -3ex 
(-1)^{|A| + \ell(w)}[\phi - 2\rho(A) + (\rho_{L_1} - w\rho_{L_1})]\bigg)\\[1.5ex]
&= \Ind_{T_1}^{G_1} \bigg(\sum_\phi m_\sigma(\phi)\hskip -3ex
\sum_{\substack{A\subset
  \Delta({\mathfrak g}_1[1],{\mathfrak t}_1)\\B \subset
  \Delta^+({\mathfrak l}_1,{\mathfrak t}_1) }} \hskip -3ex
 (-1)^{|A|+|B|} [\phi - 2\rho(A) + 2\rho(B)]\bigg).
\end{align*}}
Here the outer sum is over the highest weights $\phi$ of the
virtual representation $\sigma$ of $L_1$, and the integers
$m_\sigma(\phi)$ are their multiplicities. The notation $2\rho(X)$
stands for the sum of a set $X$ of roots.
\item If every weight $\phi-2\rho(A)+2\rho(B)$ in (4) is replaced by
  its unique dominant conjugate, we get a computable formula
  $$[\mu_*({\mathcal S}(\sigma))] = \sum_{\gamma\in {\mathcal P}_{K\dashLL}(G,K)}
    m_\sigma(\gamma) [\gamma]_\theta.$$
\end{enumerate}\end{corollary}

\begin{proof}
\begin{subequations}\label{se:acharcomplexproof}
For (1), that $\mu$ is proper is immediate from the fact that $G_1/Q_1$ is
projective. That $G_1\cdot X_1$ is open in the image follows from
Theorem \ref{thm:JM}(6). Since $\mu$ is proper and the domain is
irreducible, the image must be the closure of $G_1\cdot X_1$. The
fiber over $X_1$ is evidently $G^{X_1}/(G^{X_1}\cap Q_1)$, which is a
single point by Theorem \ref{thm:JM}(4). So the map is birational.

Part(2) is \eqref{e:Kvec} and \eqref{e:Kfiber}.

Part (3) is standard algebraic geometry.

For (4), the Leray spectral sequences for the maps $\mu$ and $\pi$ say that
\begin{small}\begin{equation}\label{e:Leray}\begin{aligned}
 \sum_i (-1)^i &H^0(\overline{G_1\cdot X_1},R^i\mu_*{\mathcal S}(\sigma)) =
 \sum_i (-1)^i H^i\left({\mathcal R}_1,{\mathcal S}(\sigma)\right)\\
  &=  \sum_i (-1)^i H^i\left(G_1/Q_1,G_1\times_{Q_1} (\sigma \otimes S^\bullet({\mathfrak
    g}_1[{\scriptstyle \ge 2}\,])^*)\right)\\
 &= \sum_{i,j} (-1)^{i+j} H^i\left(G_1/Q_1,G_1\times_{Q_1} \sigma \otimes
  {\textstyle\bigwedge^j} {\mathfrak g}_1[1]^* \otimes S^\bullet({\mathfrak
    g}_1[{\scriptstyle \ge 1}\,])^*\right)\\
  &= \sum_j (-1)^j \Ind_{L_1}^{G_1} \left( \sigma \otimes
  {\textstyle\bigwedge^j} {\mathfrak g}_1[1]^*\right).
  \end{aligned}
\end{equation}\end{small}
For the first equality, we use the fact that $\overline{G_1\cdot X_1}$
is affine, so that higher cohomology vanishes. For the second, we use
the fact that $\pi$ is affine, so the higher direct images vanish. For
the third, we use the deRham complex identity (valid in $\K {\mathcal
    M}_f(G_1)$) (see \eqref{eq:MfK}) for any $G_1$-representation $V$ such that
$S^\bullet(V)$ decomposes into irreducibles with finite multiplicities)
$$[{\mathbb C}] = \sum_j (-1)^j [S^\bullet(V) \otimes
  {\textstyle\bigwedge^j} V],$$
applied to $V= {\mathfrak g}_1[1]^*$. The fourth equality more or less
identifies functions (or sections of a vector bundle) on $T^*G_1/Q_1$
with functions (or sections of a vector bundle) on $G_1/L_1$. Now
\eqref{e:Leray} is the first equality in (4).

The second equality in (4) uses first of all the fact that
if $\sigma_0$ is an irreducible of $L_1$ of highest weight $\phi_0$, then
\begin{equation}\label{e:Weyl}
  \sigma_0 = \sum_{w\in W(L_1)} (-1)^{\ell(w)}
  \Ind_{T_1}^{L_1}(\phi_0+\rho_{L_1} - w\rho_{L_1}).
\end{equation}
(This is a version of the Weyl character formula.) Next, it uses the
fact that if $E$ is any representation of $L_1$ (in
this case an exterior power of ${\mathfrak g}_1[1]^*$), then
\begin{equation}
  \Ind_{T_1}^{L_1}(\psi)\otimes E = \sum_{\gamma\in \Delta(E)}
  \Ind_{T_1}^{L_1}(\psi+\gamma).
\end{equation}
Here the sum runs over the weights of $T_1$ on $E$.

The third equality in (4) follows for example from the Bott-Kostant
fact that (if ${\mathfrak n}_1(L_1)$ corresponds to a set of positive
roots for $T_1$ in ${\mathfrak l}_1$) the weights of $T_1$ on
$H^*({\mathfrak n}_1(L_1),{\mathbb C})$ are the various $\rho_{L_1} -
w\rho_{L_1}$, appearing in degree $\ell(w)$. The sum on the left side
runs over the weights of cohomology (indexed by $w\in W(L_1)$), and on
the right over the weights of the complex $\bigwedge^\bullet {\mathfrak
  n}_1(L_1)^*$ (indexed by subsets $B$ of positive roots for $L_1$).
\end{subequations} 
\end{proof}

\begin{corollary}\label{cor:JMres}
 We continue in the setting \eqref{se:JM}.
\begin{enumerate}
\item The restriction map in equivariant $\K$-theory (see \eqref{e:Kright})
  $$\R(L_1) \simeq \R(Q_1) \simeq \K^{G_1}({\mathcal R}_1) \rightarrow
  K^{G_1}(U) \simeq \R(G_1^{X_1}) = \R(Q_1^{X_1})$$
  sends a (virtual) representation $[\sigma]$ of $Q_1$ to
  $[\sigma|_{Q_1^{X_1}}]$.
\item Any virtual (algebraic) representation $\tau$ of
  $G_1^{X_1}=Q_1^{X_1}$ can be extended to a virtual (algebraic)
  representation $\sigma$ of $Q_1$. That is, the restriction map of
  representation rings
$$\R(L_1) \simeq \R(Q_1) \twoheadrightarrow \R(Q_1^{X_1}) \simeq
\R(L_1^{X_1})$$
is surjective.
\item Suppose $[\tau]$ is a virtual algebraic representation of $G_1^{X_1}$,
  corresponding to a virtual equivariant coherent sheaf ${\mathcal T}$
  on $G_1\cdot X_1$. Choose a virtual algebraic
  representation $[\sigma]$ of $Q_1$ extending $\tau$. Then the virtual
  coherent sheaf
  $$[\mu_*({\mathcal S}(\sigma))] \eqdef [\widetilde {\mathcal T}]$$
  is a virtual extension (Corollary \ref{cor:Kspangeom}) of
  $[{\mathcal T}]$. We have a formula
    $$[\widetilde{\mathcal T}] = \sum_{\gamma\in {\mathcal P}_{K\dashLL}(G,K)}
    m_{\widetilde{\mathcal T}}(\gamma) [\gamma]_\theta.$$
\end{enumerate}
{\em Computability} of the extension $\sigma$ in (3) is a problem in
finite-dimensional representation theory of reductive algebraic
groups, for which we do {\em not} offer a general solution.
\end{corollary}

\begin{proof}
For (1), that the restriction in $\K$-theory corresponds to
restriction of representations is clear from the description of
representatives for the classes on ${\mathcal R}_1$ in \eqref{e:KR1rep}.

For (2), restriction to an open subset in equivariant $\K$-theory is
always surjective; this is the right exactness of \eqref{e:Kright}. So
(2) follows from (1).

Part (3) follows from the last assertion of Corollary
\ref{cor:JMbdle}(3), and Corollary \ref{cor:JMbdle}(5).
\end{proof}

The last formula in Corollary \ref{cor:JMres} relates the geometric
basis of Theorem \ref{thm:finorb} to the representation-theoretic
basis of Corollary \ref{cor:KKbasis}. The difficulty, as mentioned in the
Corollary above, is that computing this formula requires (for each
irreducible representation $\tau$ of $L_1^{X_1}$) a
{\em computable} virtual representation $\sigma$ of $L_1$ with
$$\sigma|_{L_1^X} = \tau.$$
Finding such a $\sigma$ does not seem to be an intractable
problem. For $GL(n)$, Achar addresses it in his thesis \cite{acharTH}. For
other classical $G_1$, a fairly typical example (arising for
the nilpotent element in $Sp(2n)$ corresponding to the partition
$2^n$ of $2n$) has
$$L_1 = GL(n,{\mathbb C}),\qquad L_1^{X_1} = O(n,{\mathbb C}).$$ But
we are not going to address this branching problem. Instead, we will
calculate not individual basis vectors $\widetilde{\mathcal
  E}^{\alg}(\tau)$, but rather a basis of the {\em span} of all these
vectors as $\tau$ varies over $\widehat{G_1^{X_1}}$ (always for a
fixed $X_1$).

For our application to calculating associated varieties, the price is
that we can calculate the components of the associated variety and
their multiplicities, but not the virtual representations of isotropy
groups giving rise to those multiplicities.

What we gain for this price is an algorithm, which we have implemented
in the {\tt atlas} software (see \cite{atlas}).

\begin{algorithm}[A geometric basis for equivariant
    $\K$-theory]\label{alg:acharOrbit}
  \addtocounter{equation}{-1}
  \begin{subequations}\label{se:acharOrbit}

  We begin in the setting \eqref{se:JM} with a nilpotent orbit
  \begin{equation}\label{e:acharNotationA}
    Y = G_1\cdot X_1 \subset {\mathcal N}_1 \simeq {\mathcal N}_1^*
    \simeq {\mathcal N}^*_\theta.
  \end{equation}
(Here we use the identifications of \eqref{eq:cplxnilporb} and
  \eqref{se:gg*}.) The goal is to produce a collection of explicit elements
  \begin{equation}\label{e:acharNotationB}
    {\mathcal E}^{\orbalg}_j(Y) = \sum_{\gamma\in {\mathcal
    P}_{K\dashLL}(G,K)} m_{{\mathcal E}^{\orbalg}_j(Y)}(\gamma)[\gamma]_\theta
      \in K^K(\overline Y) 
  \end{equation}
  which are a {\em basis} of $\K^K(\overline Y)/\K^K(\partial\overline
  Y)$. (The superscript ``$\orbalg$'' stands for ``orbital
  algorithm.'' The subscript $j$ is just an indexing parameter for the
  basis vectors we compute, running over
  $$\{0,1,\ldots M-1\}\qquad \text{or}\qquad {\mathbb N};$$
  it has no particular
  meaning. It replaces the parameter $\tau$ (corresponding to the
  coherent sheaf ${\mathcal E}$) in the basis of Theorem \ref{thm:finorb}.
 The algorithm proceeds by induction on $\dim Y$; so we assume that
 such a basis is available for every boundary orbit $Y' \subset
 \partial\overline Y$.

 Given an arbitrary (say irreducible) representation
 $\sigma$ of $L_1$, \ref{cor:JMbdle}(5) provides an expression
 \begin{equation}\label{e:acharSpan}
   [\mu_*{\mathcal S}(\sigma)] = \sum_{\gamma \in {\mathcal P}_{K\dashLL}(G,K)}
   m_{\sigma}(\gamma)[\gamma]_\theta.
 \end{equation}
 The ``sheaf'' $\mu_*{\mathcal S}(\sigma)$ (actually it is a formal
 alternating sum of higher direct image sheaves, but the higher terms
 are supported on the boundary) is a vector bundle over $Y$, of rank
 \begin{equation}\label{e:acharRankA}
\rank([\mu_*{\mathcal S}(\sigma)]|_{G_1\cdot X_1}) = \dim(\sigma).
 \end{equation}
This dimension (of an irreducible of $L_1$) is easy to compute.

 According to Corollary \ref{cor:JMres}, the classes $\{[\mu_*{\mathcal
   S}(\sigma)] \mid \sigma \in \widehat{L_1}\}$, after restriction to
 $\K^{G_1}(Y)$, are a spanning set. Furthermore the kernel of the
 restriction map has as basis the (already computed) set
   \begin{equation}\label{e:kerRes}
     \bigcup_{Z\subset {\partial Y}} \{{\mathcal E}_k^{\orbalg}(Z)\}.
    \end{equation}
   Now extracting a subset
   \begin{equation}
{\mathcal E}_j^{\orbalg}(Y) = \sum_{\sigma\in \widehat{L_1}}
n_j(\sigma) [\mu_*{\mathcal S}(\sigma)]
   \end{equation}
of the span of the $[\mu_*{\mathcal S}(\sigma)]$ restricting to a basis of
the image of the restriction is a linear algebra problem. Because the
rank (the virtual dimension of fibers over $G_1\cdot X_1$) is additive
in the Grothendieck group, we can compute each integer
 \begin{equation}\label{e:acharRankB}
\rank([{\mathcal E}_j^{\orbalg}]) = \sum_{\sigma\in \widehat{L_1}}
n_j(\sigma)\dim(\sigma).
 \end{equation}
   \end{subequations} 
\end{algorithm}

In this description we have swept under the rug the issue of doing
finite calculations. We will now address this. Recall from
\eqref{eq:form} the invariant bilinear form $\langle,\rangle$ on
${\mathfrak g}$, and from Proposition \ref{prop:formpos} the fact that
this form defines a positive definite form on any character lattice
$X^*(H)$, with $H\subset G$ a maximal torus. The same proof applies to
non-maximal tori; so we get a positive definite form on highest
weights for any reductive subgroup of $G$.

\begin{lemma}\label{lemma:branchnorm}  Suppose $E\subset F$ are
  algebraic subgroups of $G$ (not necessarily connected), and that
  $\tau$ and $\sigma$ are
  irreducible representations of $E$ and $F$ respectively. Define
  $$\| \tau\| = \text{length of a highest weight of $\tau$}$$
  and similarly for $\sigma$. If $\tau$ appears in $\sigma|_E$, then
  necessarily
  $$\|\tau\| \le \|\sigma\|.$$
\end{lemma}
Because an irreducible algebraic representation must be trivial on the
unipotent radical, we may assume that $E$ and $F$ are reductive, so
that the notion of ``highest weight'' makes sense. The lemma reduces
immediately to the case when $E$ and $F$ are tori, and in that case is
obvious.

If now $\mathcal E$ is a geometric parameter corresponding to an
irreducible representation $\tau$ of $G_1^\xi$, we define
\begin{equation}\label{eq:geomnorm}
  \|{\mathcal E}\| = \|\tau\|.
\end{equation}
If $\gamma$ is a dominant weight thought of as a Langlands parameter,
we define
\begin{equation}\label{eq:KLnorm}
  \|\gamma\| = \text{length of $\gamma$ as a weight}.
\end{equation}

\begin{proposition} \label{prop:LBnorm}
Suppose we are in the setting of Corollary \ref{cor:JMres}.
  \begin{enumerate}
    \item Any formula for an extension of $(Y,{\mathcal E})$ must
      include a term $\gamma$ for which the restriction of the
      $G_1$-representation of extremal weight $\gamma$ contains the
      $G_1^{X_1}$ representation $\tau$ defining ${\mathcal E}$; and
      therefore
      $$\|\gamma\| \ge \|\tau\|.$$
      \item There is a constant $C$, depending only on $G$, so that
        there is a virtual extension of $\tau$  to $L_1$ in which
        every $L_1$ highest weight $\gamma_1$ appearing satisfies
           $$\| \gamma_1\| \le \|\tau\| + C.$$
      \item There is a constant $C$ depending only on $G$ so that, in
        the formula for the virtual extension of ${\mathcal E}$ given
        in Corollary \ref{cor:JMres}, every $K$-Langlands parameter
        $\gamma$ appearing satisfies
        $$\|\gamma\| \le \|\tau\| + 2C.$$
  \end{enumerate}
\end{proposition}
\begin{proof} The main assertion in Part (1) is elementary, and then
  the last inequality is Lemma \ref{lemma:branchnorm}.  Part (2) is
  elementary but tedious; we omit the argument. Part (3) is clear by
  inspection of Corollary \ref{cor:JMbdle}(4); the constant $C$ in this
  case is a bound for the sizes of the various root sums appearing.
\end{proof}

Here now is how to make Algorithm \ref{alg:acharOrbit} into a finite
calculation. We fix some bound $N$, and at every stage consider
only the (finitely many) irreducible representations of $L_1$ of
highest weights bounded in size by $N + C$. When this is done, all the
linear algebra mentioned in the algorithm will take place in the
finite-rank ${\mathbb Z}$ module spanned by $K$-Langlands parameters
of size bounded by $N+2C$. Instead of surjectivity for the restriction
from $\R(L_1)$ to $\R(L_1^X)$, what we will know is that
\begin{equation}\label{e:finiteCalcA}
  \text{\parbox{.65\textwidth}{the image contains all irreducible
      representations of $L_1^X$ of highest weight size bounded by $N$.}}
\end{equation}
The conclusion about the algorithm is that
\begin{equation}\label{e:finiteCalcB}
  \text{\parbox{.82\textwidth}{proposed basis vectors with
$K$-Langlands parameters of size bounded by $N$ are linearly
      independent in $\K^K(\overline Y)/\K^K(\partial \overline Y)$.}}
\end{equation}
The proposed basis vectors involving parameters of size between $N$
and $N+2C$ will indeed live in $\K^K(\overline Y)/\K^K(\partial
\overline Y)$, but they may not be linearly independent.

\section{Associated varieties for complex
  groups}\label{sec:compassvar}
\setcounter{equation}{0}
In the setting of \eqref{se:cplx}, suppose $X$ is a $({\mathfrak
  g},K)$-module for the complex group $G_1$ (regarded as a real
group).
\begin{subequations}\label{se:assvar}
Kazhdan-Lusztig theory allows us (if $X$ is specified as a sum
of irreducibles in the Langlands classification) to find an explicit
formula (in the Grothendieck group of finite length Harish-Chandra modules)
\begin{equation}\label{e:cplxKL}
X = \sum_{(\lambda_L,\lambda_R)\in {\mathcal P}_{\LL}(G,K)}
m_X(\lambda_L,\lambda_R) I(\lambda_L,\lambda_R).
\end{equation}
Fix a $K$-invariant good filtration of the Harish-Chandra module $X$,
so that $\gr X$ is a finitely generated $S({\mathfrak g}/{\mathfrak
  k})$-module supported on ${\mathcal N}^*_\theta$. The class in
equivariant $\K$-theory
\begin{equation}
[\gr X] \in \K^K({\mathcal N}^*_\theta)
\end{equation}
is independent of the choice of good filtration. Because of the
characterizations in Definition \ref{def:cplxbasis} of the basis
$\{[\gamma]_\theta\}$ of this $\K$-theory, we find
\begin{equation}\label{e:cplxKformula}\begin{aligned}
\ [\gr X] &= \sum_{\gamma \in {\mathcal P}_{K\dashLL}(G)}
  \left(\sum_{\substack{(\lambda_L,\lambda_R)\in {\mathcal
        P}_{\LL}(G,K)\\ (\lambda_L,\lambda_R) \sim_K \gamma}}
  m_X(\lambda_L,\lambda_R)\right)[\gamma]_\theta\\
&= \sum_{\gamma \in {\mathcal P}_{K\dashLL}(G)} m_X(\gamma)[\gamma]_\theta.
\end{aligned}\end{equation}
Here the equivalence $\sim_K$ in the first inner sum is that of Definition
\ref{def:Klangcplx}. If we think of $\gamma$ as a dominant weight, then
\begin{equation}
(\lambda_L,\lambda_R)\sim_K \gamma \iff \gamma \in W(K,T)\cdot(\lambda_L
  +\lambda_R).
\end{equation}
Recall now the classes $[{\mathcal E}_k^{\orbalg}(Z)]$
constructed in Achar's Algorithm \ref{alg:acharOrbit}. Comparing their known
formulas with \eqref{e:cplxKformula}, we can do an (upper
triangular) change of basis calculation, and get an explicit formula
\begin{equation}\label{e:cplxgeomKform}
[\gr X] = \sum_{{\mathcal E}_k^{\orbalg}(Z)} n_X({\mathcal
  E}_k^{\orbalg}(Z))[{\mathcal E}_k^{\orbalg}(Z)],
\end{equation}
with computable integers $n_X({\mathcal E}_k^{\orbalg}(Z))$.

Here is how to make this calculation finite. After using
Kazhdan-Lusztig theory to calculate the formula
\eqref{e:cplxKformula}, write $N$ for the size of the largest
highest weight appearing. (If $X$ is irreducible, then $N$ is just the
length of the highest weight of the lowest $K$-type of $X$; no
Kazhdan-Lusztig theory arises.) Then run the algorithm as described in
\eqref{e:finiteCalcA} and \eqref{e:finiteCalcB}, using always
representations of $L_1$ of highest weights bounded by $N+C$.
\end{subequations} 

\begin{theorem}\label{thm:cplxassvar} Suppose $X$ is a $({\mathfrak
  g},K)$-module for the complex group $G_1$ (regarded as a real
group). Use the notation of \eqref{se:assvar}.
\begin{enumerate}
\item The associated variety of $X$ (Definition \ref{def:assvar}) is
  the union of the closures of
  the maximal $K$-orbits $Z\subset {\mathcal N}^*_\theta$ with some
  $n_X({\mathcal E}_k^{\orbalg}(Z)) \ne 0$.
\item The multiplicity of a maximal orbit $Z$ in the associated
cycle of $X$ is
$$\sum_{{\mathcal E}_k^{\orbalg}(Z)} n_X({\mathcal
    E}_k^{\orbalg}(Z))\rank({\mathcal E}_k^{\orbalg}(Z)).$$
\end{enumerate}
\end{theorem}
\begin{proof} Theorem \ref{thm:finorb}(3) provides a formula in
  equivariant $\K$-theory
\begin{equation}\label{e:cplxgeomvarform}
[\gr X] = \sum_{\substack{(Z,{\mathcal F}) \in {\mathcal P}_g(G,K)\\
Z\subset \supp(\gr X)}} m_X(Z,{\mathcal F})[\widetilde{\mathcal
F}^{\alg}],
\end{equation}
and implies that the coefficients of the terms on maximal orbits are
independent of choices. The definition \eqref{eq:AC}, and the
definition of the basis vectors ${\mathcal E}_k^{\orbalg}(Z)$, relate
the coefficients in \eqref{e:cplxgeomvarform}  to the
multiplicities in the associated variety.  We know that the basis
$\{\widetilde{\mathcal F}\}$ can be expressed in terms of the basis
$\{{\mathcal E}_k^{\orbalg}\}$, and that the change of basis is weakly
upper triangular with respect to the ordering of
orbits by closure.  From  these facts the theorem follows.
\end{proof}

\section{Representation basis for $\K$-theory: ${\mathbb R}$
  case}\label{sec:realalg}
\setcounter{equation}{0}
Everything has been said so as to carry over to real (linear
algebraic) groups with minimal changes. We return therefore to the
general setting of \eqref{se:realgrps}. Just as in the complex case,
Theorem \ref{thm:finorb} in the present setting gives

\begin{corollary}\label{cor:Kspangeomreal}
Write $\{Y_1,\ldots,Y_r\}$ for the orbits of $K$
on ${\mathcal N}^*_\theta$ (\eqref{e:Knilcone}). For each irreducible
$K$-equivariant vector bundle ${\mathcal E}$ on some $Y_i$, fix a
$K$-equivariant (virtual) coherent sheaf $\widetilde{\mathcal E}$
supported on the closure of $Y_i$, and restricting to ${\mathcal E}$
on $Y_i$.
Then the classes $[\widetilde{\mathcal E}]$ are a basis of
$\K^K({\mathcal N}^*_\theta)$. A little more precisely, the classes
supported on any $K$-invariant closed subset $Z \subset {\mathcal
  N}_\theta$ are a basis of $\K^K(Z)$.
\end{corollary}

As stated at the beginning of Section \ref{sec:cplx}, the
distinguishing complication in the general real case is the
formulation of the Langlands classification. We now begin to explain
the details we need. The main point is that Harish-Chandra
parametrized the discrete series of $G({\mathbb R})$ using characters
of a compact maximal torus; but this ``family'' of representations
changes drastically as the character moves from one Weyl chamber to
another. The discrete series characters constitute a nice family only
if we restrict the Harish-Chandra parameter to vary just  over
characters within a single Weyl chamber.

Once that is done, we have the problem that each nice family of
representations is inconveniently small: it is indexed not by {\em
  all} characters of a maximal torus, but only by appropriately {\em
  dominant} characters. We address this (following a fundamental idea of
Hecht and Schmid from the 1970s) by enlarging the family to depend on
arbitrary (not necessarily dominant) characters.

The resulting families of (virtual) representations are convenient for our
calculations, but too large to index irreducible representations. The
positivity notion of {\em weak} in Definition \ref{def:langreal}
singles out those representations that have some chance to be part of
the classification.

There remain two smaller issues. First, when the character is dominant
but singular, it may happen that the corresponding representation of
$G({\mathbb R})$ is zero. This possibility is ruled out by the
condition {\em nonzero} in Definition \ref{def:langreal}. Second
(again for singular characters) it may happen that the same
representation is attached to characters on two nonconjugate maximal
tori. In this case it turns out that (among these various
realizations) there is a unique one on a most compact torus; this is
the one identified by the condition {\em final}.

\begin{definition}\label{def:langreal} (See for example
  \cite{ALTV}*{Section 6} for details.) A {\em continued Langlands
    parameter} for
  $(G,K)$ is a triple $\Gamma = (H,\gamma,\Psi)$ such that
\begin{enumerate}
\item $H$ is a $\theta$-stable maximal torus in $G$;
\item $\gamma$ is  a one-dimensional $({\mathfrak
    h},[H^{\theta}]^{\rho_{\abs}})$-module in which the kernel of the
  two to one covering map $[H^\theta]^{\rho_{\abs}} \rightarrow
  H^\theta$ acts nontrivially; and
\item $\Psi$ is a system of positive imaginary (that is,
  $\theta$-fixed) roots for $H$ in $G$.
\end{enumerate}
Two continued Langlands parameters are {\em equivalent} if they are
conjugate by $K$.  A continued Langlands parameter is called {\em
  weak} if in addition
\begin{enumerate}[resume]
\item $d\gamma\in {\mathfrak h}^*$ is weakly dominant with respect to
  $\Psi$.
\end{enumerate}

The weak Langlands parameter is called {\em nonzero} if in addition
\begin{enumerate}[resume]
\item whenever $\alpha\in \Psi$ is simple and compact, $\langle
  d\gamma,\alpha^\vee\rangle \ne 0$.
\end{enumerate}
Here $d\gamma\in {\mathfrak h}^*$ means the weight by which the Lie
algebra ${\mathfrak h}$ acts in $\gamma$.
The nonzero weak Langlands parameter is called {\em final} if in
addition
\begin{enumerate}[resume]
\item whenever $\beta$ is a real root of $H$ such that
$\langle d\gamma,\beta^\vee\rangle =0$, then $\gamma_{\mathfrak
  q}(m_\beta) = 1$.
\end{enumerate}
(Here $\gamma_{\mathfrak q}$ is a $\rho$-shift of
$\gamma$ defined in \cite{ALTV}*{(9.3g)}.)

We write
$${\mathcal P}_{\LL}(G,K) = \{\text{equivalence classes of final
  Langlands parameters}\}.$$

The Langlands classification attaches 
to the equivalence class of a continued parameter $\Gamma$ a {\em
  continued standard representation} $[I(\Gamma)]$ (more precisely, a virtual
$({\mathfrak g},K)$-module in the Grothendieck group $\K({\mathfrak
  g},K)$ defined in \eqref{se:introgK}), with the following
properties.
\begin{enumerate}
\item The infinitesimal character of $[I(\Gamma)]$ is
  the $W(G,H)$ orbit of $d\gamma$.
\item The restriction of $[I(\Gamma)]$ to $K$ depends only on
$$\Gamma_K = (H,\gamma|_{H^\theta},\Psi).$$
More precisely,
\item the class in equivariant $\K$-theory
$$[\gr I(\Gamma)] \eqdef [\Gamma]_\theta \in \K^K({\mathcal N}^*_\theta)$$
is independent of $\gamma|_{{\mathfrak h}^{-\theta}}$.
\item If $\Gamma$ is weak, then $[I(\Gamma)]$ is represented by a
  $({\mathfrak g},K)$-module $I(\Gamma)$, a {\em weak standard
    representation}.
\item The weak standard representation $I(\Gamma)$ is nonzero if and
  only if the parameter $\Gamma$ is nonzero (as defined above).
\item If $\Gamma$ is weak and nonzero, then $I(\Gamma)= \oplus_{i=1}^r
  I(\Gamma_i)$; here $\{\Gamma_i\}$ is
  a computable finite set of final parameters attached to a
  single (more compact) $\theta$-stable maximal torus $H'$.
\item If $\Gamma$ is final, there is a unique irreducible quotient
  $J(\Gamma)$  of $I(\Gamma)$. 
\item Any irreducible $({\mathfrak g},K)$-module is equivalent to some
  $J(\Gamma)$ with $\Gamma$ final.
\item Two final standard representations are isomorphic
  (equivalently, their Langlands quotients are isomorphic) if and only
  if their parameters  are conjugate by $K$.
\end{enumerate}

Missing from these properties is an explicit description of the
equivariant $\K$-theory class $[\Gamma]_\theta$ like Definition
\ref{def:langcplx}(5) in the complex case. We will return to this
point in Section \ref{sec:stdK}.

Just as for complex groups, the final standard representation
$I(\Gamma)$ is tempered if and only if the character $\gamma$ of
$H({\mathbb R})^{\rho_{\abs}}$ is unitary; equivalently, if and only if
  $d\gamma|_{{\mathfrak  h}^{-\theta}} \in i{\mathfrak h}^*$.  In this
case $I(\Gamma) = J(\Gamma)$.

In the general (possibly nontempered) case, the real part of
$d\gamma|_{{\mathfrak h}^{-\theta}}$ controls the growth of matrix
  coefficients of $J(\Gamma)$. When $\RE d\gamma|_{{\mathfrak
      h}^{-\theta}}$ is larger, the matrix coefficients grow faster.

Partly because of the notion of tempered, it is useful to define the
{\em $K$-norm} of a continued parameter $\Gamma$:
$$\|\Gamma\|^2_K \eqdef \langle d\gamma|_{{\mathfrak
    h}^\theta}, d\gamma|_{{\mathfrak h}^\theta} \rangle.$$
In the setting of property (6) above, $\|\Gamma\|_K = \|\Gamma_i\|_K$.

The $K$-norm is evidently bounded by the canonical real part of the
infinitesimal character:
$$\|\Gamma\|^2_K = \langle \RE d\gamma,\RE d\gamma\rangle - \langle
\RE d\gamma|_{{\mathfrak h}^{-\theta}}, \RE d\gamma|_{{\mathfrak
      h}^{-\theta}} \rangle \le \langle \RE d\gamma,\RE d\gamma\rangle,$$
with equality if and only if $\gamma$ is unitary.
\end{definition}

We pause here to mention the real groups formulation of the Langlands
classification, to which we alluded in the introduction.
\begin{subequations}\label{se:realLangReal}
  As usual we use the notation of \eqref{se:realgrps}. There are
  natural bijections
  \begin{equation}\label{e:tori}\begin{aligned}
      &\left\{\text{$\theta$-stable maximal tori
        $H_1 \subset G$}\right\}/\text{$K$-conjugacy}\\
      \longleftrightarrow \ & \left\{\text{$\theta$-stable real maximal
        tori $H_2 \subset G$}\right\}/\text{$K({\mathbb
          R})$-conjugacy}\\
            \longleftrightarrow\  &\left\{\text{real maximal tori $H_3 \subset
              G$}\right\}/\text{$G({\mathbb R})$-conjugacy}
  \end{aligned}\end{equation}
    A $\theta$-stable torus $H_i$ contains an algebraic subgroup
    \begin{equation}\label{e:tortheta}
      H_i^\theta \qquad (i = 1, 2).
    \end{equation}
    A real torus $H_j$ has a real form
    \begin{equation}\label{e:torreal}
      H_j({\mathbb R}) \qquad (j = 2, 3)
    \end{equation}
    which in turn has a natural maximal split subtorus
    \begin{equation}\label{e:torrealsplit}
      H_j({\mathbb R}) \supset A_j({\mathbb R}) \simeq ({\mathbb
        R}^\times)^d\qquad (j = 2, 3)
    \end{equation}
     It is the (topological) {\em identity component}
   \begin{equation}\label{e:torrealA}
      A_j \eqdef A_j({\mathbb R})_0 \simeq ({\mathbb
        R}_+^\times)^d\qquad (j = 2, 3)
    \end{equation}
that typically appears in discussions of structure theory for real
reductive groups. Just as for a general reductive group, the (unique)
maximal compact subgroup of $H_j({\mathbb R})$ is the group of real
points of a (unique) algebraic subgroup $T_j \subset H_j$:
\begin{equation}\label{e:torrealcpt}
  H_j({\mathbb R}) \supset T_j({\mathbb R}) = \text{maximal compact
    subgroup.}
\end{equation}
The Cartan decomposition is the direct product decomposition
\begin{equation}
  H_j({\mathbb R}) = T_j({\mathbb R}) \times A_j \qquad (j=2,3).
\end{equation}
Consequently the continuous characters of $H_j({\mathbb R})$ may be
described as
\begin{equation}\label{e:torrealchar}\begin{aligned}
  \widehat{H_j({\mathbb R})} &\simeq \widehat{T_j({\mathbb R})} \times
  \widehat A_j\\
  &\simeq \widehat T_j \times {\mathfrak a}_j^*;\end{aligned}
  \end{equation}
the last equality is because $T_j({\mathbb R})$ is a compact form of
the algebraic group $T_j$, and $A_j$ is an abelian vector group.

Since $H_2$ is both real and $\theta$-stable, we find
\begin{equation}\label{eq:realtortheta}
  T_2 = H_2^\theta,\qquad \widehat{T_2({\mathbb R})} \simeq \widehat
  H_2^\theta.
\end{equation}
\end{subequations} 

\begin{definition}\label{def:langrealReal} (See for example
  \cite{ALTV}*{Section 6} for details.) A {\em continued Langlands
    parameter} for $G({\mathbb R})$ is a triple $\Gamma = (H({\mathbb
    R}),\gamma,\Psi)$ such that
\begin{enumerate}
\item $H({\mathbb R})$ is a maximal torus in $G({\mathbb R})$;
\item $\gamma$ is a level one character of the $\rho_{\abs}$ double
 cover of $H({\mathbb R})$, and
\item $\Psi$ is a system of positive imaginary (that is,
  $\sigma_{\mathbb R}(\alpha) = -\alpha$) roots for $H$ in $G$.
\end{enumerate}
Two continued Langlands parameters are {\em equivalent} if they are
conjugate by $G({\mathbb R})$.
\end{definition}
This definition can be continued in a way precisely parallel to
Definition \ref{def:langreal}, defining in the end the set of
Langlands parameters
\begin{equation}\label{e:realparams}
  {\mathcal P}_{\LL}(G({\mathbb R})) \simeq {\mathcal P}_{\LL}(G,K).
\end{equation}
The bijection with Langlands parameters for $(G,K)$ is an easy
consequence of \eqref{se:realLangReal}. This entire digression is just
another instance of Harish-Chandra's idea that analytic questions
about representations of $G({\mathbb R})$ can often be phrased
precisely as algebraic questions about $({\mathfrak g},K)$-modules.

We now return to that algebraic setting.

\begin{definition}\label{def:Klangreal} A {\em $K$-Langlands continued
    parameter} for
  $(G,K)$ is a triple $\Gamma_K = (H,\gamma_K,\Psi)$ such that
\begin{enumerate}
\item $H$ is a $\theta$-stable maximal torus in $G$;
\item $\gamma_K$ is a level one character of the  $\rho_{\abs}$ double
  cover of $H^\theta$; and
\item $\Psi$ is a system of positive imaginary (that is,
  $\theta$-fixed) roots for $H$ in $G$.
\end{enumerate}
Two continued $K$-Langlands parameters are {\em equivalent} if they are
conjugate by $K$.  A continued $K$-Langlands parameter is called {\em
  weak} if in addition
\begin{enumerate}[resume]
\item $d\gamma_K\in ({\mathfrak h}^\theta)^*$ is weakly dominant with
  respect to $\Psi$.
\end{enumerate}
The weak $K$-Langlands parameter is called {\em nonzero} if in addition
\begin{enumerate}[resume]
\item whenever $\alpha\in \Psi$ is simple and compact, $\langle
  d\gamma_K,\alpha^\vee\rangle \ne 0$.
\end{enumerate}
The nonzero $K$-Langlands parameter is called {\em final} if in addition
\begin{enumerate}[resume]
\item for any real root $\beta$, $\gamma_{K,{\mathfrak q}}(m_\beta) = 1$.
\end{enumerate}
The set of equivalence classes of final $K$-Langlands parameters is
written ${\mathcal P}_{K\dashLL}(G,K)$.

In the complex case, we get a $K$-Langlands parameter from a Langlands
parameter just by discarding a bit of information (the restriction to
${\mathfrak h}^{-\theta}$). In the general real case, matters are
more subtle. The difference between final $K$-Langlands parameters and final
Langlands parameters is first, that there is no character on (the
split torus) ${\mathfrak h}^{-\theta}$ (or, equivalently, that
$d\gamma_K$ is assumed to be zero there); and second, that the finality
condition is assumed for {\em all} the real roots, rather than just
those on which $d\gamma$ vanishes.

We will make use of the {\em $K$-norm} of a $K$-Langlands parameter,
defined exactly as for Langlands parameters by
$$\|\Gamma_K\|^2_K \eqdef \langle d\gamma_K, d\gamma_K
\rangle;$$
the weight whose length we are taking belongs to $({\mathfrak
  h}^\theta)^*$.
\end{definition}

\begin{theorem}[\cite{Vres}*{Theorem 11.9}]\label{thm:Klangreal} The
    set ${\mathcal P}_{K\dashLL}(G,K)$ of equivalence classes
of final $K$-Langlands parameters (Definition
\ref{def:Klangreal}) is in one-to-one correspondence with
\begin{enumerate}
\item (final Langlands parameters for) tempered
representations of real infinitesimal character (extending
$\Gamma_K$ by zero on ${\mathfrak h}^{-\theta}$); or
\item irreducible representations of $K$ (by taking lowest $K$-type).
\end{enumerate}
\end{theorem}

Once we have in hand the $K$-Langlands parameters, there is an obvious
extension of Lusztig's conjecture (what is proven by Bezrukavnikov in
\cite{Bezr}) to real groups. But this extension is not true
for $SL(2,{\mathbb R})$ (see Example \ref{ex:SL2R} below).

We can now begin to extend to real groups the ideas in Section
\ref{sec:cplxalg}.

\begin{theorem}\label{thm:realgeomexpand}
Suppose $(Y,{\mathcal E}) \in {\mathcal P}_{g}(G,K)$ is a
geometric parameter (Definition \ref{def:geom}); fix an extension
$\widetilde{{\mathcal E}}$ as in Corollary
\ref{cor:Kspangeomreal}. Then there is a formula in 
$\K^K({\mathcal N}_\theta)$
$$[\widetilde{\mathcal E}] = \sum_{\Gamma_K\in {\mathcal
    P}_{K\dashLL}(G,K)} m_{\widetilde{\mathcal
    E}}(\Gamma_K)[\Gamma_K]_\theta.$$

Here $m_{\widetilde{\mathcal E}}(\Gamma_K) \in {\mathbb Z}$, and the
sum is finite.

Suppose $\widetilde{\mathcal E}'$ is another extension of ${\mathcal
  E}$ to $\overline Y$. Then
$$[\widetilde{\mathcal E}'] - [\widetilde{\mathcal E}] =
\sum_{\substack{(Z,{\mathcal F}) \in {\mathcal P_{g}(G,K)}\\[.2ex]
    Z \subset \partial Y}} n_{\mathcal F} [{\widetilde {\mathcal F}}].$$
Here $n_{\mathcal F} \in {\mathbb Z}$, and the sum is finite.
\end{theorem}

We will prove this in Corollary \ref{cor:KRres} below.

In the complex case, Bezrukavnikov's proof of the
Lusztig-Bezrukavnikov conjecture guarantees the existence of an
extension $\widetilde{\mathcal E}$ with a {\em single} leading term,
and in this way finds a {\em bijection} between geometric parameters
and $K$-Langlands parameters. In the real case there will sometimes be
no reasonable way to arrange a {\em single} leading term, and
accordingly no such bijection. Fortunately computers are better able
than humans to do linear algebra with matrices that are not upper
triangular.

\begin{corollary}\label{cor:KKrealbasis}
  In the setting of \eqref{se:realgrps}, the classes
  $$\left\{[\Lambda_K]_\theta \in \K^K({\mathcal N}^*_\theta)\mid \Lambda_K
  \in {\mathcal P}_{K\dashLL}(G,K)\right\}$$
  are a ${\mathbb Z}$-basis of $\K^K({\mathcal N}^*_\theta)$. The
  restriction to $K$ map
  $$\res_K\colon \K^K({\mathcal N}^*_\theta) \rightarrow \K(K)$$
 of \eqref{eq:resKK} is injective.
\end{corollary}
\begin{proof}
That these classes span is a consequence of Theorem
\ref{thm:finorb} and Theorem \ref{thm:realgeomexpand}. That
they are linearly independent is a consequence of Theorem
\ref{thm:Klangreal}; the argument proves injectivity of the
restriction at the same time.
\end{proof}

\begin{example}\label{ex:SL2R}
\addtocounter{equation}{-1}
\begin{subequations}
Let us take $G=SL(2,{\mathbb C})$,
$$D=\begin{pmatrix} 1&0 \\
  0&-1\end{pmatrix} \qquad \theta(g) = DgD^{-1},$$
so that $K=H_c$ is the diagonal torus, and $G$ is the complexification
of $SU(1,1)$. We have naturally
\begin{equation}
  K\simeq {\mathbb C}^\times, \qquad \widehat K \simeq {\mathbb Z};
\end{equation}
we will write an irreducible representation of $K$ just as an
unadorned integer. The $K$-nilpotent cone is
$${\mathcal N}^*_\theta \simeq \left\{\begin{pmatrix} 0&a\\
    b&0\end{pmatrix} \mid ab=0 \right\}.$$
There are two nonzero orbits of $K$ on ${\mathcal N}_\theta^*$:
$$Y^+ = \left\{\begin{pmatrix} 0&a\\   0&0\end{pmatrix} \mid a\ne 0
\right\}, \quad  Y^- = \left\{\begin{pmatrix} 0&0\\   b&0\end{pmatrix}
  \mid b\ne 0 \right\},$$
each isomorphic to $K/\{\pm I\}$; and the zero orbit $Y^0 \simeq K/K$. The
geometric parameters are therefore
\begin{equation}\label{e:SL2geom}
{\mathcal P}_g(G,K) = \{(Y^\pm,{\mathcal E}^\pm_{\triv}),\ (Y^\pm,{\mathcal
  E}^{\pm}_{\sgn}),\ (Y^0,{\mathcal E}^0_n)\mid n\in {\mathbb Z} \}.
\end{equation}
(In each case the superscript $0$ or $\pm$ on the vector bundle
identifies the underlying orbit.)

On the $\theta$-stable maximal torus $H_c$, the Cartan involution
$\theta$ acts trivially. Consequently every root is imaginary, and
there are two systems of positive imaginary roots: $\Psi^+$
(corresponding to upper triangular matrices), and $\Psi^-$. Attached
to each non-negative integer $n$ there are two final $K$-Langlands
parameters $\Gamma_K^+(n)$ (corresponding to $\Psi^+$, with the
differential of the character identified with $n$) and
$\Gamma_K^-(n)$. These are discrete series and limits of discrete
series:
\begin{equation}\begin{aligned}
  \ [\Gamma_K^+(n)]|_K &= n+1,\ n+3,\ n+5,\ldots\\
  \  [\Gamma_K^-(n)]|_K &=-n-1,\ -n-3,\ -n-5,\ldots.
\end{aligned}\end{equation}

A representative of the other $K$-conjugacy class of $\theta$-stable maximal
tori is
$$H_s = \left\{ \pm\begin{pmatrix} \cosh z & \sinh z \\ \sinh z & \cosh
    z\end{pmatrix}\right\},$$
with Lie algebra
$${\mathfrak h}_s = \left\{\begin{pmatrix} 0&z\\ z&0\end{pmatrix} \mid
z\in {\mathbb C}\right\}.$$
Here $\theta$ acts by inversion, so there are no
imaginary roots. There is exactly one final $K$-Langlands parameter
$\Gamma_K^0$, corresponding to the spherical principal series:
\begin{equation}
  [\Gamma_K^0]|_K =0, \pm 2, \pm 4,\ldots
  \end{equation}
and so on. A reasonable partial order on these parameters is
\begin{equation}\begin{aligned}
\Gamma_K^0 &\prec \Gamma_K^+(1) \prec \Gamma_K^+(3) \prec
\Gamma_K^+(5) \prec \cdots \\
\Gamma_K^0 &\prec \Gamma_K^-(1) \prec \Gamma_K^-(3) \prec
\Gamma_K^-(5) \prec \cdots \\
\Gamma_K^+(0) &\prec \Gamma_K^+(2) \prec \Gamma_K^+(4) \prec
\Gamma_K^+(6) \prec \cdots \\
\Gamma_K^-(0) &\prec \Gamma_K^-(2) \prec \Gamma_K^-(4) \prec
\Gamma_K^-(6) \prec \cdots
\end{aligned}
\end{equation}
Here are some reasonable choices of extensions:
\begin{equation}\label{e:muSL(2)a}
\widetilde{{\mathcal E}_n^0} = [\Gamma_K^{\sgn(n)}(|n|-1)]_\theta -
          [\Gamma_K^{\sgn(n)}(|n|+1)]_\theta\quad (n\ne 0),
\end{equation}
\begin{equation}\label{e:muSL(2)b}
\widetilde{{\mathcal E}_0^0} = [\Gamma_K^0]_\theta - [\Gamma_K^+(1)]_\theta -
          [\Gamma_K^-(1)]_\theta,
  \end{equation}
\begin{equation}\label{e:muSL(2)c}
\widetilde{{\mathcal E}_{\sgn}^\pm} = [\Gamma_K^\pm(0)]_\theta;
\end{equation}
and
\begin{equation}\label{e:muSL(2)d}
\widetilde{{\mathcal E}_{\triv}^\pm} = [\Gamma_K^\pm(1)]_\theta.
\end{equation}

But in the last case, there is another reasonable choice of extension:
\begin{equation}\label{e:muSL(2)dprime}
\widetilde{{\mathcal E}_{\triv}^\pm}' = [\Gamma_K^0]_\theta -
[\Gamma_K^\mp(1)]_\theta.
\end{equation}

(Here we make the natural choice of extending the structure sheaf on
the open orbit $Y^\pm$ to the structure sheaf on its closure.)
So here is what we have in the direction of a Lusztig-Bezrukavnikov
bijection for $SL(2,{\mathbb R})$:
\begin{equation}\label{e:LBSL2R}\begin{aligned}
{\mathcal E}_n^0 &\longleftrightarrow \Gamma_K^{\sgn(n)}(|n|+1) \qquad
(n\ne 0)\\
{\mathcal E}_{\sgn}^\pm &\longleftrightarrow \Gamma_K^\pm(0)\\
{\mathcal E}_0^0,\ {\mathcal E}_{\triv}^+,\ {\mathcal E}_{\triv}^-
&\longleftrightarrow \Gamma_K^+(1),\ \Gamma_K^-(1);
\end{aligned}
\end{equation}
the map from left to right is taking some kind of ``leading terms'' of
some natural extension. One might like to include on the right in the
last case the $K$-Langlands parameter $\Gamma_K^0$; it is not a
leading term, but the result is that there is something like an
``almost bijection,'' with the last three ``smallest'' geometric
parameters corresponding (as a set) to the three ``smallest''
$K$-Langlands parameters.
\end{subequations}
\end{example}

We conclude this section by recording the (known) information we will
need about continued standard parameters.

\begin{theorem}\label{thm:normalize} Use the notation of Definitions
  \ref{def:langreal} and \ref{def:Klangreal}.
  \begin{enumerate}
  \item The equivalence classes (that is, orbits of $K$)
    $$\{[I(\Gamma) \mid \Gamma \in {\mathcal P}_{\LL}\}$$
 of final Langlands parameters are a ${\mathbb Z}$-basis of the
 Grothendieck group ${\mathcal M}_f({\mathfrak g},K)$ of finite length
 Harish-Chandra modules.
    \item For any final parameter $\Gamma$, Kazhdan-Lusztig theory
      computes
      $$[J(\Gamma)] = \sum_{\Lambda \in {\mathcal P}_{\LL}} m_\Gamma(\Lambda)
      [I(\Lambda)];$$
      here $m_\Gamma(\Lambda) \in {\mathbb Z}$, and the sum is
      finite. We have $m_\Gamma(\Gamma) = 1$, and the other nonzero
      terms all satisfy
      $$\|\Lambda\|_K  > \|\Gamma\|_K.$$
    \item For any continued parameter $\Gamma'$, the unique formula
      $$[I(\Gamma')] = \sum_{\Lambda \in {\mathcal P}_{\LL}}
      p_{\Gamma'}(\Lambda)[I(\Lambda]$$
      can be computed using classical results of Hecht and Schmid.
  \item The equivalence classes (that is, orbits of $K$)
    $$\{[I(\Gamma_K)] \mid \Gamma_K\in
    {\mathcal P}_{K\dashLL}\}$$
 are a ${\mathbb Z}$-basis of the Grothendieck group of finite-length
 Harish-Chandra modules restricted to $K$.
    \item For any final parameter $\Gamma\in {\mathcal P}_{\LL}$,
      there is an elementary computation of the unique formula
      $$[I(\Gamma)|_K] = \sum_{\Lambda_K\in {\mathcal P}_{K\dashLL}}
      m_\Gamma(\Lambda_K) I(\Lambda_K).$$
      All the parameters $\Lambda_K$ appearing live on the same (more
      compact) maximal torus, satisfy
$$\|\Lambda_K\|_K = \|\Gamma\|_K,$$
      and have  $m_\Gamma I(\Lambda_K) = 1$; they correspond to the
      lowest $K$-types of $I(\Gamma)$ or $J(\Gamma)$.
    \item For any continued parameter $\Gamma'$, the unique formula
      $$[I(\Gamma')|_K] = \sum_{\Lambda_K \in {\mathcal P}_{K\dashLL}}
      q_{\Gamma'}(\Lambda_K) [I(\Lambda_K)]$$
can be computed explicitly.
      \item For any final parameter $\Gamma \in {\mathcal P}_{\LL}$,
        there is computable formula
        $$J(\Gamma)|_K = \sum_{\Lambda_K\in {\mathcal P}_{K\dashLL}}
        n_{J(\Gamma)} I(\Lambda_K)|_K,$$
        or equivalently
        $$[\gr J(\Gamma)] = \sum_{\Lambda_K\in {\mathcal P}_{K\dashLL}}
        n_{J(\Gamma)} [\Lambda_K]_\theta \in K^K({\mathcal N}^*_\theta).$$
  \end{enumerate}
\end{theorem}
This theorem corresponds to the preparations made in
\eqref{se:assvar} in the complex case.

\section{Standard representations restricted to $K$}\label{sec:stdK}
\setcounter{equation}{0}
In this section we will recall how to compute the restrictions to $K$
of the (continued) standard $({\mathfrak g},K)$-modules described in
Definition \ref{def:langreal}. This will be critical for the
description in Section \ref{sec:realalgweights} of how explicitly to
write the formulas of Theorem \ref{thm:realgeomexpand}.

\begin{subequations}\label{se:cohind}
Always we work in the setting \eqref{se:realgrps}. Suppose to begin that
we have also a $\theta$-stable parabolic subgroup with $\theta$-stable
Levi decomposition
\begin{equation}
Q = LU, \qquad \theta Q = Q,\qquad \theta L = L.
\end{equation}
It is not difficult to show that
\begin{equation}
K/(Q\cap K) \hookrightarrow G/Q
\end{equation}
is a closed embedding, so that $K/(Q\cap K)$ is projective, and
therefore $Q\cap K$ is parabolic in $K$. (Since $K$ may be
disconnected, there is a question about the meaning of ``parabolic
subgroup.'' We will say that $P\subset K$ is {\em parabolic} if $K/P$ is
projective; equivalently, if $P\cap K_0 = P_0$ is parabolic in $K_0$,
or if $P$ contains a (connected) Borel subgroup of $K_0$.) We may in
particular fix a torus
\begin{equation}
T \subset L\cap K
\end{equation}
that is a maximal torus in $K_0$.
\end{subequations} 

In the next proposition the disconnectedness of $K$ complicates
matters slightly, and is the reason we need not get {\em irreducible}
representations of $K$ from irreducibles of $L\cap K$.

\begin{theorem}[Bott-Borel-Weil]\label{thm:BW} In the setting
  \eqref{se:cohind}, suppose $(\sigma,S)$ is an algebraic
  representation of $L\cap K$ (or even of $Q\cap K$). Then we get an
  equivariant algebraic vector bundle
$$
{\mathcal S} = K\times_{Q\cap K} S \rightarrow K/(Q\cap K).
$$
\begin{enumerate}
\item Each cohomology space $H^i(K/(Q\cap K),{\mathcal S})$ is a
  finite-dimensional algebraic representation of $K$.
\item The virtual representation
$$\sum_i (-1)^i [H^i(K/(Q\cap
  K),{\mathcal S})] \in \R(K)$$
depends only on the class $[(\sigma,S)] \in \R(L\cap K)$.
\item Suppose $(\sigma,S)$ is irreducible, and that its infinitesimal
  character is represented by
$$\xi_{L\cap K} \in X^*(T) - \rho_{L\cap K} \subset {\mathfrak t}^*.$$
Write
$$\xi_K = \xi_{L\cap K} - \rho({\mathfrak u}\cap {\mathfrak k}) \in
X^*(T) - \rho_K.$$
Then {\em either}
$$H^i(K/(Q\cap K),{\mathcal S}) = 0, \qquad (\text{all $i$})$$
(if $\xi_K$ vanishes on some coroot of $T$ in $K$); {\em or}
$$H^i(K/(Q\cap K),{\mathcal S}) = \begin{cases}
  \text{\parbox{.339\textwidth}{nonzero representation of infinitesimal
      character $\xi_K$}} & (i = i(\xi_K))\\[.5cm]
0 & (i\ne i(\xi_K)).
\end{cases}$$
\end{enumerate}
\end{theorem}

Our first tool for computing cohomological induction is the operation
\begin{equation}
\begin{aligned}\label{e:Kcohind}
\Ind_{Q\cap K}^K &\colon \R(L\cap K) \rightarrow \R(K), \\
[(\sigma,S)] &\mapsto \sum_i (-1)^i [H^i(K/(Q\cap K),{\mathcal S})] \in \R(K).
\end{aligned}
\end{equation}
(Recall that $\R(K)$ is the representation ring of virtual
representations of $K$. In the language of $\K$-theory,
\begin{equation}\label{e:KcohindK}
  \Ind_{Q\cap K}^K \colon \K^{L\cap K}(\point) \rightarrow
  \K^K(\point).
\end{equation}
One should think of the case when $(\sigma,S)$ is a $K$-dominant
irreducible representation of $L\cap K$; then $\Ind_{Q\cap
  K}^K(\sigma)$ lives only in the highest degree $\dim K/(Q\cap K)$,
and there is essentially an irreducible representation of $K$ of
highest weight $\sigma - 2\rho({\mathfrak u}\cap {\mathfrak k})$. (For
disconnected $K$ it may happen that there are several irreducible
representations of $K$ of highest weight $\sigma - 2\rho({\mathfrak
  u}\cap {\mathfrak k})$; this is how reducible representations of $K$
arise in Theorem \ref{thm:BW}.)

\begin{proposition}[Zuckerman] In the setting \eqref{se:cohind}, there
  are cohomological induction functors
$${\mathcal R}^i\colon {\mathcal M}_f({\mathfrak l},L\cap K)
\rightarrow {\mathcal M}_f({\mathfrak g}, K) \qquad (0\le i \le
  \dim{\mathfrak u}\cap {\mathfrak k})$$
(notation \eqref{e:MfgK}) with the following properties.
\begin{enumerate}
\item The class
$$[{\mathcal R}(Z)] \eqdef \sum_i (-1)^i [{\mathcal R}^i(Z)] \in
\K({\mathfrak g}, K)$$
is well-defined, depending only on $[Z] \in \K({\mathfrak l}, L\cap K)$.
\item The class
$$[\gr {\mathcal R}(Z)] \eqdef \sum_i (-1)^i [\gr {\mathcal R}^i(Z)]
\in \K^K({\mathcal N}^*_\theta)$$
is well-defined, depending only on $[\gr Z] \in \K^{L\cap K}({\mathcal
  N}^*_{L,\theta})$.
\item Suppose $H\subset L$ is a $\theta$-stable maximal torus, and
  $\Gamma_L = (H,\gamma_L,\Psi_L)$ is a continued Langlands parameter
  for $(L,L\cap K)$. Define
$$\gamma_G = \gamma_L \otimes \rho({\mathfrak u})^*$$
$$\Psi_G = \Psi_L \cup \{\text{imaginary roots of $H$ in ${\mathfrak
    u}$}\},$$
so that $\Gamma_G = (H,\gamma_G,\Psi_G)$ is a continued Langlands
parameter for $(G,K)$. Then
$$[{\mathcal R}I(\Gamma_L))] = [I(\Gamma_G)]. $$
\end{enumerate}
\end{proposition}

This proposition describes (or at least says that it is possible to
describe) how to construct standard modules using the
\begin{subequations}\label{se:indKnil}
geometry of \eqref{se:cohind}. We are interested in computing $\gr$
(an image in the equivariant $\K$-theory of the nilpotent cone) of
standard modules; so we need to relate that geometry to
\eqref{se:cohind}. Often the best way to think of $G/Q$ is as a
variety of parabolic subgroups:
\begin{equation}
G/Q \simeq {\mathcal Q} \eqdef \text{variety of parabolic
  subalgebras conjugate to ${\mathfrak q}$}.
\end{equation}
To think about nilpotent elements in ${\mathfrak g}^*$, it may be
helpful to recall that in any identification ${\mathfrak g} \simeq
{\mathfrak g}^*$ from an invariant bilinear form, we have
$${\mathfrak q} \simeq ({\mathfrak g}/{\mathfrak u})^*.$$
The natural projection ${\mathfrak q} \rightarrow {\mathfrak l}$
corresponds to restriction of linear functionals
\begin{equation}
\pi_{\mathfrak q}\colon ({\mathfrak g}/{\mathfrak u})^* \rightarrow ({\mathfrak
  q}/{\mathfrak u})^*.
\end{equation}
An element of ${\mathfrak q}$ is nilpotent if and only if its image in
${\mathfrak l}$ is nilpotent. We therefore write
\begin{equation}\begin{aligned}
{\mathcal N}^*_{\mathfrak l} &= \text{nilpotent cone in ${\mathfrak
    l}^*$}, \\
{\mathcal N}^*_{\mathfrak q} &=  \pi_{\mathfrak q}^{-1}({\mathcal
  N}^*_{\mathfrak l});
\end{aligned}
\end{equation}
this is an ``affine space bundle'' over the nilpotent cone for $L$
(roughly, a vector bundle without chosen zero section) corresponding
to the vector space $({\mathfrak g}/{\mathfrak q})^* \simeq {\mathfrak
  u}$.  If we use the identification ${\mathfrak g}^*\simeq {\mathfrak
  g}$,
$${\mathcal N}^*_{\mathfrak q} \simeq {\mathcal N}_{\mathfrak l} +
{\mathfrak u}.$$

For the $K$-nilpotent cone,
\begin{equation}
\pi_{{\mathfrak q},\theta} \colon ({\mathfrak g}/({\mathfrak
  u}+{\mathfrak k}))^* \rightarrow ({\mathfrak
  q}/({\mathfrak u}+({\mathfrak q}\cap{\mathfrak k})))^* \simeq
({\mathfrak l}/({\mathfrak l}\cap{\mathfrak k}))^*.
\end{equation}
\begin{equation}\begin{aligned}
{\mathcal N}^*_{{\mathfrak l},\theta} &= \text{$(L\cap K)$-nilpotent
  cone in ${\mathfrak l}^*$}\\
&= {\mathcal N}^*_{\mathfrak l} \cap
({\mathfrak l}/({\mathfrak l}\cap {\mathfrak k}))^* \\
{\mathcal N}^*_{{\mathfrak q},\theta} &=  \pi_{{\mathfrak
    q},\theta}^{-1}({\mathcal N}^*_{{\mathfrak l},\theta})\\
{\mathcal N}^*_{{\mathfrak q},\theta} &\simeq {\mathcal
  N}_{{\mathfrak l},\theta} + ({\mathfrak u} \cap {\mathfrak s}).
\end{aligned}
\end{equation}

The basic Grothendieck-Springer method to study nilpotent elements is
to consider
\begin{equation}\begin{aligned}
{\mathcal N}^*_{\mathcal Q} &\eqdef \{(\xi',{\mathfrak q}')
\mid {\mathfrak  q}' \in {\mathcal Q},\  \xi' \in {\mathcal N}^*_{{\mathfrak
  q}'} \}\\
&\simeq G\times_Q {\mathcal N}^*_{\mathfrak q}.
\end{aligned}\end{equation}
Points here are nilpotent linear functionals $\xi'$ on ${\mathfrak g}$ with
the extra information of a chosen parabolic ${\mathfrak q}'$
(conjugate to ${\mathfrak q}$) so that $\xi'$ vanishes on the nil
radical ${\mathfrak u}'$ of ${\mathfrak q}'$.  Such parabolics exist
for any $\xi'$, so the {\em moment map}
\begin{equation}
\mu_{\mathcal Q}\colon {\mathcal N}^*_{\mathcal Q} \rightarrow
{\mathcal N}^*, \qquad (\xi',{\mathfrak q}') \mapsto \xi'
\end{equation}
is (projective and) surjective.  In the same way, the projection
\begin{equation}
\pi_{\mathcal Q}\colon {\mathcal N}^*_{\mathcal Q} \rightarrow
{\mathcal Q}, \qquad (\xi',{\mathfrak q}') \mapsto {\mathfrak q}'
\end{equation}
is an affine morphism (even a bundle) with fiber ${\mathcal
  N}^*_{\mathfrak q}$.

The subvariety $K/(Q\cap K)$ is
\begin{equation}\begin{aligned}
K/(Q\cap K) \simeq {\mathcal Q}_K \eqdef\ &\text{variety of
  $\theta$-stable parabolic}\\
&\text{subalgebras conjugate by $K$ to ${\mathfrak q}$},
\end{aligned}
\end{equation}
a single closed orbit of $K$ on ${\mathcal Q}$. Over this orbit we are
interested in a subbundle of ${\mathcal N}^*_{\mathcal Q}$
\begin{equation}\begin{aligned}
{\mathcal N}^*_{{\mathcal Q},\theta} &\eqdef \{(\xi',{\mathfrak q}')
\mid {\mathfrak  q}' \in {\mathcal Q}_K,\  \xi' \in {\mathcal N}^*_{{\mathfrak
  q}',\theta} \}\\
&\simeq K\times_{Q\cap K} {\mathcal N}^*_{{\mathfrak q},\theta}.
\end{aligned}\end{equation}
Points here are nilpotent linear functionals $\xi'$ on ${\mathfrak
  g}/{\mathfrak k}$ with
the extra information of a chosen $\theta$-stable parabolic ${\mathfrak q}'$
(conjugate by $K$ to ${\mathfrak q}$) so that $\xi'$ vanishes on the nil
radical ${\mathfrak u}'$ of ${\mathfrak q}'$.  Such parabolics may
{\em not} exist
for some $\xi'$, so the {\em moment map}
\begin{equation}
\mu_{{\mathcal Q},\theta}\colon {\mathcal N}^*_{{\mathcal Q},\theta}
\rightarrow
{\mathcal N}^*_\theta, \qquad (\xi',{\mathfrak q}') \mapsto \xi'
\end{equation}
is projective but not necessarily surjective.  In the same way, the projection
\begin{equation}
\pi_{{\mathcal Q},\theta}\colon {\mathcal N}^*_{{\mathcal Q},\theta}
\rightarrow
{\mathcal Q}_K, \qquad (\xi',{\mathfrak q}') \mapsto {\mathfrak q}'
\end{equation}
is an affine morphism (even a bundle) with fiber ${\mathcal
  N}^*_{{\mathfrak q},\theta}$.

Suppose now that
\begin{equation}
{\mathcal E}_L \in \Coh^{L\cap K}({\mathcal N}^*_{{\mathfrak l},\theta});
\end{equation}
that is, that ${\mathcal E}_L$ is a finitely generated module for
$S({\mathfrak l})$, with ${\mathfrak l}\cap {\mathfrak k}$ and the
$L$-invariants of positive degree acting by zero, and endowed with a
compatible action of $L\cap K$. The pullback
\begin{equation}
{\mathcal E}_Q \eqdef \pi^*_{{\mathfrak q},\theta}({\mathcal E}_L)
\in \Coh^{Q\cap K}({\mathcal N}^*_{{\mathfrak q},\theta})
\end{equation}
is obtained by first regarding ${\mathcal E}_L$ as an $S({\mathfrak
  q})$ module (in which ${\mathfrak u}$ acts by zero), and then
tensoring over $S({\mathfrak q})$ with $S({\mathfrak g})$ to extend
scalars. We can define
\begin{equation}
{\mathcal E}_G \eqdef K\times_{Q\cap K} {\mathcal E}_Q \in
\Coh^K({\mathcal N}^*_{{\mathcal Q},\theta})
\end{equation}
as in \eqref{e:Kfiber}. Because $\pi_{{\mathfrak q},\theta}$ is
proper, the higher direct images
\begin{equation}
R^i\mu_*({\mathcal E}_G) \in \Coh^K({\mathcal N}^*_\theta)
\end{equation}
are all coherent sheaves on the nilpotent cone.
\end{subequations} 

\begin{proposition}[Zuckerman's Blattner formula] In the setting
  \eqref{se:cohind},
  suppose $Z\in {\mathcal M}_f({\mathfrak l},L\cap K)$, with
$$[\gr Z] \in \K^{L\cap K}({\mathcal N}^*_{{\mathfrak l},\theta})$$
the corresponding class in equivariant $\K$-theory. Define
$${\mathcal R}[\gr Z] = \sum_i (-1)^i [R^i\mu_*([\gr Z]_L)] \in
\K^K({\mathcal N}^*_\theta)$$
(notation as in \eqref{se:indKnil}). Then
$$[\gr{\mathcal R}(Z)] = {\mathcal R}[\gr Z].$$
\end{proposition}

This is Zuckerman's proof of the Blattner formula; the representations
of $K$ appearing on the right are computable from the $L\cap K$-types
of $Z$ (by Theorem \ref{thm:BW}). Those on the left are the
$K$-types of ${\mathcal R}(Z)$.

\section{Geometric basis for $\K$-theory: ${\mathbb R}$
  case}\label{sec:realalgweights}
\setcounter{equation}{0}

In this section we will explain how to compute one extension of a
geometric parameter (Theorem \ref{thm:realgeomexpand}).
As in the complex case, we will proceed in the aesthetically
distasteful way of using the Jacobson-Morozov theorem (and so discussing not the
nilpotent elements in ${\mathfrak g}^*$ that we care about, but rather
the nilpotent elements in ${\mathfrak g}$).

\begin{subequations}\label{se:KR}
We begin therefore with an arbitrary $K$-nilpotent element $E_\theta\in
{\mathcal N}_\theta$ (see \eqref{e:Kadnilcone}). The Kostant-Rallis
result Theorem \ref{thm:KR} finds elements $F_\theta \in
{\mathcal N}_\theta$ and $D_\theta$ in ${\mathfrak k}$ so that
\begin{equation}\label{e:tthetatds}
[D_\theta,E_\theta] = 2E_\theta,\quad [D_\theta,F_\theta] =
-2F_\theta,\quad [E_\theta,F_\theta] = D_\theta.
\end{equation}
We use the eigenspaces of $\ad(D_\theta)$ to define a $\theta$-stable
parabolic subgroup
$Q=L U$ of $G$, with Levi factor
$L=G^{D_\theta}$. We will be concerned with the equivariant vector bundle
\begin{equation}
 {\mathcal R}_\theta =_{\text{\textnormal{def}}} K\times_{Q\cap
 K} {\mathfrak s}[{\scriptstyle\ge}\,2]
  \buildrel{\pi}\over{\longrightarrow} K/Q\cap K.
\end{equation}
(The reason ${\mathcal R}_\theta$ is of interest is that Corollary
\ref{cor:KRbdle} below says that it is a $K$-equivariant resolution
of singularities of the nilpotent orbit closure $\overline{K\cdot
  E_\theta}$. The ${\mathcal R}$ is meant to stand for {\em resolution}.)
According to \eqref{e:Kfiber} and \eqref{e:Kvec},
\begin{equation}
  \K^{K}({\mathcal R}_\theta) \simeq
  \R(Q\cap K) \simeq \R(L\cap K);
\end{equation}
If $(\sigma,S)$ is an irreducible representation of $L\cap K \simeq
Q\cap K/U\cap K$, we write
\begin{equation}
{\mathcal S}_0(\sigma) = K\times_{Q\cap K} S
\end{equation}
for the induced vector bundle on $K/Q\cap K$. The corresponding
basis element of the equivariant
$\K$-theory is represented by the equivariant vector bundle
\begin{equation}\label{e:KRthetarep}
  {\mathcal S}(\sigma) = \pi^*({\mathcal S}_0(\sigma)) =
  K\times_{Q\cap K} \left({\mathfrak
    s}[{\scriptstyle \ge}\,2] \times S \right) \rightarrow {\mathcal
    R}_\theta.
\end{equation}

\end{subequations}

We are in the setting of Proposition \ref{thm:KR}. As
a consequence of that Proposition, we have

\begin{corollary}\label{cor:KRbdle} Suppose we are in the setting of
\eqref{se:KR}.
\begin{enumerate}
\item The natural map
  $$\mu\colon {\mathcal R}_\theta \rightarrow {\mathcal N}_\theta,
  \qquad (k,Z)\mapsto \Ad(k)Z$$
is a proper birational map onto $\overline{K\cdot E_\theta}$. We
may therefore {\em identify} $K\cdot E_\theta$ with its preimage $U$:
$$K/K^{E_\theta} \simeq K\cdot E_\theta\simeq U \subset {\mathcal
  R}_\theta.$$
Because $K\cdot E_\theta$ is open in $\overline{K\cdot E_\theta}$,
$U=\mu^{-1}(K\cdot E_\theta)$ is open in ${\mathcal R}_\theta$.
\item The classes
  $$\{ [{\mathcal S}(\sigma)] \mid \sigma \in \widehat{L\cap K} =
  \widehat{Q\cap K}\}$$
  of \eqref{e:KRthetarep} are a basis of the equivariant $\K$-theory
 $\K^K({\mathcal R}_\theta)$.
\item Since $\mu$ is proper, higher direct images of coherent sheaves
  are always coherent. Therefore
  $$[\mu^*({\mathcal S}(\sigma)] \eqdef \sum_i (-1)^i
[R^i\mu^*{\mathcal S}(\sigma)] \in \K^{K}({\mathcal N}_\theta)$$
is a well-defined virtual coherent sheaf. This gives a map in
equivariant $\K$-theory 
$$\mu_*\colon \K^K({\mathcal R}_\theta) \rightarrow \K^K({\mathcal
  N}_\theta).$$
Restriction to the open set $U\simeq K\cdot E_\theta$ commutes with $\mu_*$.
\end{enumerate}
Suppose $\sigma$ is a (virtual) algebraic representation of $L\cap
K$. Write
$$[\sigma] = \sum_i m_i(\sigma)[I(\Gamma^i_{L\cap K})]_\theta \in \K^{L\cap
  K}({\mathcal N}^*_{{\mathfrak l},\theta})$$
(computably, as explained in \cite{Vres}); here the $H^i$ are $\theta$-stable
maximal tori in $L$, and
$$\Gamma^i_{L\cap K} = (H^i,\gamma_{L\cap K}^i,\Psi^i_{L}) \in
{\mathcal P}_{L\cap K\dashLL}(L,L\cap K).$$
The left side $[\sigma]$ is a finite-dimensional virtual
representation of $L\cap K$, regarded as a class in the equivariant
$\K$-theory of the nilpotent cone supported at $\{0\}$.
\begin{enumerate}[resume]
\item As a representation of $K$,
\begin{align*}
[\mu^*({\mathcal S}(\sigma))] &= \sum_j (-1)^j {\mathcal R}\left(
  [\textstyle{\bigwedge^j}
{\mathfrak s}[1]^*\otimes \sigma] \right)\\
&= \sum_j (-1)^j {\mathcal R}\left( \sum_i m_i[I(\Gamma^i_{L\cap K})]\otimes
\textstyle{\bigwedge^j} {\mathfrak s}[1]^*\right) \\
&= \sum_i m_i \sum_{A\subset \Delta({\mathfrak s}[1],H^i\cap K)}
(-1)^{|A|} {\mathcal R}[\gr I(\Gamma^i_{L\cap K} - 2\rho(A))]\\
&= \sum_i m_i \sum_{A\subset \Delta({\mathfrak s}[1],H^i \cap K))}
(-1)^{|A|}  [\gr I(\Gamma^i_K - 2\rho(A))].\end{align*}
Here $\Delta({\mathfrak s}[1],H^i \cap K))$ is the set of weights of
$H^i\cap K$ on ${\mathfrak s}[1]$.
\item If every continued standard representation $[I(\Gamma^i_K -
  2\rho(A)]|_K$ in (6) is replaced by an integer linear combination of
  $K$-Langlands parameters in accordance with Theorem
  \ref{thm:normalize}, we get a computable formula
  $$[\mu^*({\mathcal S}(\sigma))] = \sum_{\Lambda_K \in {\mathcal
      P}_{K\dashLL}(G,K)} m_\sigma(\Lambda_K)[\Lambda_K]_\theta.$$
\end{enumerate}
\end{corollary}

\begin{corollary}\label{cor:KRres} We continue in the setting \eqref{se:KR}.
  \begin{enumerate}
  \item The restriction map in equivariant $K$-theory
    $$\R(Q\cap K) \simeq \K^K({\mathcal R}_\theta)
    \rightarrow \K^K(U) \simeq \R(K^{E_\theta}) = \R((Q\cap K)^{E_\theta})$$
    sends a (virtual) representation $[\sigma]$ of $Q\cap K$ to
    $[\sigma|_{(Q\cap K)^{E_\theta}}]$.
\item Any virtual (algebraic) representation $\tau$ of
  $K^{E_\theta}=(Q\cap K)^{E_\theta}$ can be
  extended to a virtual (algebraic) representation $\sigma$ of
  $Q\cap K$. That is, the restriction map of representation rings
$$\R(L\cap K) \simeq \R(Q\cap K) \twoheadrightarrow
  \R((Q\cap K)^{E_\theta}) \simeq \R((L\cap K)^{E_\theta})$$
is surjective.
\item Suppose $[\tau]$ is a virtual algebraic representation of
  $K^{E_\theta}$, corresponding to a virtual coherent sheaf ${\mathcal
  T}$ on $K\cdot E_\theta$. Choose a virtual algebraic representation
  $\sigma$ of $Q\cap K$ extending $\tau$. Then the virtual coherent
  sheaf
  $$[\mu_*({\mathcal S}(\sigma))] \eqdef [\widetilde{\mathcal T}]$$
  is a virtual extension of ${\mathcal T}$. We have a formula
  $$[\widetilde{\mathcal T}] = \sum_{\Lambda_K \in {\mathcal
      P}_{K\dashLL}(G,K)} m_{\widetilde{\mathcal
      T}}(\Lambda_K)[\Lambda]_\theta.$$
  \end{enumerate}
  {\em Computability} of $\sigma$ in (3) is a problem in finite-dimensional
  representation theory of reductive algebraic groups, for which we do
  {\em not} offer a general solution. Except for this issue, the
  formula in (3) is computable.
\end{corollary}

The last formula in Corollary \ref{cor:KRres} relates the geometric
basis of Theorem \ref{thm:finorb} to the representation-theoretic
basis of Corollary \ref{cor:KKrealbasis}.

\begin{algorithm}[A geometric basis for equivariant
    $\K$-theory]\label{alg:acharReal}

\addtocounter{equation}{-1}
\begin{subequations}\label{se:acharRealOrbit}

We begin in the setting \eqref{se:KR} with a nilpotent orbit
 \begin{equation}\label{e:acharRealA}
    Y = K\cdot E_\theta \subset {\mathcal N}_\theta \simeq {\mathcal N}_\theta^*.
  \end{equation}
 The goal is to produce a collection of explicit elements
  \begin{equation}\label{e:acharRealB}
    {\mathcal E}^{\orbalg}_j(Y) = \sum_{\Lambda_K\in {\mathcal
    P}_{K\dashLL}(G,K)} m_{{\mathcal E}^{\orbalg}_j(Y)}(\Lambda_K)[\Lambda_K]_\theta
      \in \K^K(\overline Y)
  \end{equation}
  which are a {\em basis} of $\K^K(\overline Y)/\K^K(\partial\overline Y)$.
 (The superscript ``$\orbalg$'' stands for ``orbital algorithm.'' The
  subscript $j$ is just an indexing parameter for the basis vectors,
  running over either $\{0,1,\ldots,M-1\}$ or ${\mathbb N}$.)
 The algorithm proceeds by induction on $\dim Y$; so we assume that
 such a basis is available for every boundary orbit $Y' \subset
 \partial\overline Y$.

 Given an arbitrary (say irreducible) representation
 $\sigma$ of $L\cap K$, Corollary
 \ref{cor:KRbdle}(5) provides a formula
 \begin{equation}\label{e:acharRealSpan}
   [\mu_*{\mathcal S}(\sigma)] = \sum_{\Lambda_K \in {\mathcal P}_{K\dashLL}(G,K)}
   m_{\sigma}(\Lambda_K)[\Lambda_K]_\theta.
 \end{equation}
 The ``sheaf''  $\mu_*{\mathcal S}(\sigma)$ (actually it is a formal
 alternating sum of higher direct image sheaves, but the higher terms
 are supported on the boundary) restricts to a vector bundle over
 $Y$, of rank
 \begin{equation}\label{e:acharRealRankA}
\rank([\mu_*{\mathcal S}(\sigma)]|_{K\cdot E_\theta}) = \dim(\sigma),
 \end{equation}
This dimension (of an irreducible of $L\cap K$) is easy to compute.

 According to Corollary \ref{cor:KRbdle}, the classes $\{[\mu_*{\mathcal
   S}(\sigma)] \mid \sigma \in \widehat{L_1}\}$, after restriction to
 $\K^{K}(Y)$, are a spanning set. Furthermore the kernel of the
 restriction map has as basis the (already computed) set
   \begin{equation}\label{e:kerResReal}
     \bigcup_{Z\subset {\partial Y}} \{{\mathcal E}_k^{\orbalg}(Z)\}.
   \end{equation}
   Now extracting a subset
   \begin{equation}
{\mathcal E}_j^{\orbalg}(Y) = \sum_{\sigma\in \widehat{L_1}}
n_j(\sigma) [\mu_*{\mathcal S}(\sigma)]
   \end{equation}
of the span of the $[\mu_*{\mathcal S}(\sigma)]$ restricting to a basis of
the image of the restriction is a linear algebra problem. Because the
rank (the virtual dimension of fibers over $K\cdot E_\theta = Y$) is additive
in the Grothendieck group, we can compute each integer
 \begin{equation}\label{e:acharRealRankB}
\rank([{\mathcal E}_j^{\orbalg}] = \sum_{\sigma\in \widehat{L_1}}
n_j(\sigma)\dim(\sigma).
 \end{equation}
   \end{subequations} 
\end{algorithm}

Just as in the complex case, we have swept under the rug the issue of doing
finite calculations. It can be addressed along the same lines as in
the complex case; we omit the details.

\section{Associated varieties for real
  groups}\label{sec:realassvar}
In the setting of Section \ref{sec:realalg}, suppose that $X$ is a finite
length $({\mathfrak g},K)$-module.
\begin{subequations}\label{se:realassvar}
Kazhdan-Lusztig theory allows us (if $X$ is specified as a sum
of irreducibles in the Langlands classification) to find an explicit
formula (in the Grothendieck group of finite length Harish-Chandra modules)
\begin{equation}\label{e:realKL}
X = \sum_{\Lambda\in {\mathcal P}_{\LL}(G,K)}
m_X(\Lambda) I(\Lambda).
\end{equation}
Fix a $K$-invariant good filtration of the Harish-Chandra module $X$,
so that $\gr X$ is a finitely generated $S({\mathfrak g}/{\mathfrak
  k})$-module supported on ${\mathcal N}^*_\theta$. The class in
equivariant $\K$-theory
\begin{equation}
[\gr X] \in \K^K({\mathcal N}^*_\theta)
\end{equation}
is independent of the choice of good filtration. If we rewrite each
$I(\Lambda)|_K$ in terms of $K$-Langlands parameters using Theorem
\ref{thm:normalize}(6), we find a computable formula
\begin{equation}\label{e:realKformula}
\ [\gr X] = \sum_{\Lambda_K \in {\mathcal P}_{K\dashLL}(G,K)}
m_X(\Lambda_K)[\Lambda_K]_\theta.
\end{equation}
Recall now the classes $[{\mathcal E}_k^{\orbalg}(Z)]$
constructed in Achar's Algorithm \ref{alg:acharReal}. Comparing their known
formulas with \eqref{e:realKformula}, we can do a change of basis
calculation, and get an explicit formula
\begin{equation}\label{e:realgeomKform}
[\gr X] = \sum_{{\mathcal E}_k^{\orbalg}(Z)} n_X({\mathcal
  E}_k^{\orbalg}(Z))[{\mathcal E}_k^{\orbalg}(Z)],
\end{equation}
with computable integers $n_X({\mathcal E}_k^{\orbalg}(Z))$.

\end{subequations} 

\begin{theorem}\label{thm:realassvar} Suppose $X$ is a $({\mathfrak
  g},K)$-module. Use the notation of \eqref{se:realassvar}.
\begin{enumerate}
\item The associated variety of $X$ (Definition \ref{def:assvar} is
  the union of the closures of the maximal $K$-orbits $Z\subset
  {\mathcal N}^*_\theta$ with some $n_X({\mathcal E}_k^{\orbalg}(Z))
  \ne 0$.
\item The multiplicity of a maximal orbit $Z$ in the associated
  cycle of $X$ is
$$\sum_{{\mathcal E}_k^{\orbalg}(Z)} n_X({\mathcal
    E}_k^{\orbalg}(Z))\rank({\mathcal E}_k^{\orbalg}(Z)).$$
\end{enumerate}
\end{theorem}
The proof is identical to that of Theorem \ref{thm:cplxassvar} above.


%
%

\end{document}